\documentclass[11pt]{amsart}

\usepackage[vcentermath]{youngtab}
\def\mfd{M}
 \newcommand\mat[2]{\left( \begin{array}{#1} #2 \end{array} \right)}
 
 \newcommand\ttD{\mathtt{D}}
 \newcommand\ttP{\mathtt{P}}
 \newcommand\ttS{\mathtt{S}}
 
 \newcommand\Mat{\mathrm{Mat}}
 \newcommand\diag{\mathrm{diag}}
 
 \newcommand\Killing[2]{(#1,#2)}
 \newcommand\ptimes{\raisebox{-3pt}{\boldmath$\times$\unboldmath}}

\usepackage{amssymb,latexsym,amsmath}
\usepackage[mathscr]{eucal}

\numberwithin{equation}{section}

\newtheorem{corollary}[equation]{Corollary}
\newtheorem*{corollary*}{Corollary}
\newtheorem{lemma}[equation]{Lemma}
\newtheorem*{lemma*}{Lemma}
\newtheorem{proposition}[equation]{Proposition}
\newtheorem*{proposition*}{Proposition}
\newtheorem{theorem}[equation]{Theorem}
\newtheorem*{theorem*}{Theorem}

\theoremstyle{remark}
\newtheorem*{assume*}{Assume}
\newtheorem{claim}[equation]{Claim}
\newtheorem*{claim*}{Claim}

\newtheorem{definition}[equation]{Definition}
\newtheorem*{definition*}{Definition}
\newtheorem{example}[equation]{Example}
\newtheorem*{example*}{Example}
\newtheorem*{hint*}{Hint}
\newtheorem*{notation*}{Notation}
\newtheorem*{question*}{Question}
\newtheorem*{answer*}{Answer}
\newtheorem{remark}[equation]{Remark}
\newtheorem*{remark*}{Remark}
\newtheorem*{remarks*}{Remarks}

\numberwithin{HWeq}{section}
\theoremstyle{definition}

\def\a{\alpha}
\def\b{\beta}
\def\c{\gamma}
\def\d{\delta}

\def\e{\varepsilon}

\def\z{\zeta}
\def\m{\mu}
\def\n{\nu}

\def\w{\omega}

\def\x{\xi}

\def\bC{\mathbb C}

\def\bP{\mathbb P}

\def\bZ{\mathbb Z}
\def\cA{\mathcal A}
\def\cB{\mathcal B}
\def\cC{\mathcal C}

\def\cF{\mathcal F}
\def\cG{\mathcal G}

\def\cH{\mathcal H}

\def\cL{\mathcal L}
\def\cM{\mathcal M}

\def\cR{\mathcal R}

\def\fa{\mathfrak{a}}
\def\fb{\mathfrak{b}}
\def\fd{\mathfrak{d}}

\def\fe{\mathfrak{e}}
\def\ff{\mathfrak{f}}

\def\fg{{\mathfrak{g}}}

\def\fgl{\mathfrak{gl}}
\def\fh{\mathfrak{h}}

\def\fk{\mathfrak{k}}

\def\fm{\mathfrak{m}}
\def\fn{\mathfrak{n}}

\def\fp{\mathfrak{p}}

\def\fq{\mathfrak{q}}

\def\fs{\mathfrak{s}}

\def\ft{\mathfrak{t}}

\def\fz{\mathfrak{z}}

\def\sfa{\mathsf{a}}

\def\bb{\mathbf{b}}

\def\tAd{\mathrm{Ad}}
\def\tad{\mathrm{ad}}

\def\td{\mathrm{d}}
\def\sfd{\mathsf{d}}

\def\tdim{\mathrm{dim}}

\def\be{\mathbf{e}}

\def\texp{\mathrm{exp}}

\def\tGr{\mathrm{Gr}}
\def\sfH{\mathsf{H}}

\def\tti{\mathtt{i}}

\def\ttI{\mathtt{I}}
\def\bI{\mathbf{I}}

\def\tId{\mathrm{Id}}

\def\tim{\mathrm{im}}

\def\ttj{\mathtt{j}}

\def\ttJ{\mathtt{J}}

\def\ttk{\mathtt{k}}

\def\tker{\mathrm{ker}}

\def\tM{\mathrm{M}}
\def\sfm{\mathsf{m}}
\def\tmax{\mathrm{max}}

\def\tmod{\mathrm{mod}}

\def\sfN{\mathsf{N}}

\def\sfp{\mathsf{p}}

\def\sfq{\mathsf{q}}

\def\sfr{\mathsf{r}}

\def\sfs{\mathsf{s}}

\def\tspan{\mathrm{span}}

\def\bv{\mathbf{v}}

\def\sfv{\mathsf{v}}

\def\sfw{\mathsf{w}}

\def\op{\oplus}
\def\ot{\otimes}
\def\wt{\widetilde}
\def\wh{\widehat}

\def\lefthook{\hbox{\small{$\lrcorner\, $}}}
\def\tw{\hbox{\small $\bigwedge$}}
\def\sw{\hbox{\tiny $\bigwedge$}}

\def\half{\tfrac{1}{2}}

\newcommand{\stack}[2]{\ensuremath{\genfrac{}{}{0pt}{}{#1}{#2}}}
\newcommand{\rb}[1]{\raisebox{1.5ex}[0pt]{#1}}
\newcounter{numcnt}

\newcounter{cnt}

\newcounter{acnt}
\newenvironment{a_list}{ 
  \begin{list}{{\emph{(\alph{acnt})}}}
   {\usecounter{acnt} \setlength{\itemsep}{3pt}
    \setlength{\leftmargin}{25pt} \setlength{\labelwidth}{20pt} }
   }
   {\end{list}}
\newcounter{Acnt}

\newcounter{icnt}
\newenvironment{i_list}{ 
  \begin{list}{{\emph{(\roman{icnt})}}}
   {\usecounter{icnt} \setlength{\itemsep}{3pt}
    \setlength{\leftmargin}{25pt} \setlength{\labelwidth}{20pt} }
   }
   {\end{list}}
\newenvironment{i_list_nem}{ 
  \begin{list}{{(\roman{icnt})}}
   {\usecounter{icnt} \setlength{\itemsep}{3pt}
    \setlength{\leftmargin}{25pt} \setlength{\labelwidth}{20pt} }
   }
   {\end{list}}
\newcounter{Icnt}

\newcounter{exam_cnt}

\newcounter{mccnt}

\begin{document}
\title[Rigid Schubert varieties -- ROUGH DRAFT]{Rigid Schubert varieties in compact Hermitian symmetric spaces}
\author[Robles]{C. Robles${}^\ast$}
\address{Mathematics Department, TAMU Mail-stop 3368, College Station, TX  77843-3368}
\thanks{${}^\ast$Partially supported by NSF-DMS 1006353.  ${}^\dagger$Partially supported by an NSERC Postdoctoral Fellowship.}
\email{robles@math.tamu.edu, dthe@math.tamu.edu}
\author[The]{D. The${}^\dagger$}
\date{\today}
\begin{abstract}
Given a singular Schubert variety $X_w$ in a compact Hermitian symmetric space $X$ it is a longstanding question to determine when $X_w$ is homologous to a smooth variety $Y$.  We identify those Schubert varieties for which there exist first-order obstructions to the existence of $Y$.  This extends (independent) work of M. Walters, R. Bryant and J. Hong.

Key tools include: (i) a new characterization of Schubert varieties that generalizes the well known description of the smooth Schubert varieties by connected sub-diagrams of a Dynkin diagram; and (ii) an algebraic Laplacian (\`a la Kostant), which is used to analyze the Lie algebra cohomology group associated to the problem.
\end{abstract}
\maketitle

\section{Introduction}

\subsection{Motivation} \label{S:mot}

Let $X$ be an irreducible compact Hermitian symmetric space.  The integral homology $H_*(X)$ is generated by the classes $[X_w]$ of the Schubert varieties $X_w \subset X$ \cite[(6.4.5)]{MR0142697}.  (The Schubert varieties are indexed by elements $w$ of the Hasse diagram associated to $X$; see Section \ref{S:Hasse}.)  The majority of the Schubert varieties are singular.  The following is a special case of a question posed by Borel and Haefliger in \cite{MR0149503}: does there exist a smooth complex variety $Y \subset X$ that is homologous to $X_w$; that is, $[Y]=[X_w]$?  

Consider the case that $X = \tGr(m,n)$ is the Grassmannian of $m$--planes in $\bC^n$.  The following examples are discussed in \cite{SchurRigid}.  Given $k \le n-m$, fix a subspace $W = \bC^{n+1-m-k}\subset\bC^n$.  The set $\sigma(W) = \{ E \in \tGr(m,n) \ | \ E \cap W \not= 0 \}$ is a codimension $k$ Schubert variety.  When $k=1$, $\sigma(W)$ is singular, but homologous to a smooth variety.  When $k=2$, Hartshorne, Rees and Thomas prove that homology class of $\sigma(W)$  cannot be represented by any integral linear combination of smooth, oriented submanifolds in $\tGr(3,6)$ of (real) codimension four \cite{MR0357402}.  On the other hand, while $\sigma(W) \subset \tGr(2,5)$ is not homologous to a smooth variety, its homology class can be expressed as the difference of the homology classes of two smooth subvarieties.

Throughout, all (sub)manifolds are assumed to be connected, all (sub)varieties are assumed to be irreducible, and CHSS will denote an irreducible compact Hermitian symmetric space.

\subsection{History} \label{S:hist}

Given a complex variety $Y \subset X$, let $[Y] \in H_{2k}(X)$ denote its homology class, $\tdim_\bC Y = k$.  The varieties $Y$ satisfying $[Y] = r[X_w]$ for a positive integer $r$ are characterized by the Schur differential system $\cR_w \subset \tGr({|w|},TX)$, with $|w| = \tdim_\bC X_w$; see Section \ref{S:Rw}.  The Schubert variety $X_w$ is \emph{Schur rigid} if, for every integral variety $Y$ of $\cR_w$, there exists $g \in G$ such that $Y = g\cdot X_w$; otherwise we say $X_w$ is \emph{Schur flexible}.  When $X_w$ is Schur rigid the only complex subvarieties of $X$ with homology class $r[X_w]$ are the $G$--translates $g \cdot X_w$.  Therefore 
\begin{quote}
\emph{if the Schubert variety $X_w$ is singular and Schur rigid, then there exist no smooth varieties that are homologous to $X_w$.}
\end{quote}

The Schur system was first studied independently by Bryant \cite{SchurRigid} and M. Walters \cite{Walters}.  Walters studied the Schur rigidity of codimension two Schubert varieties in $\tGr(2,n)$, and smooth Schubert varieties in $\tGr(m,n)$.  Bryant obtained similar results in addition to studying some of the singular Schubert varieties in $\tGr(m,n)$, the smooth Schubert varieties in the Lagrangian grassmannians, and the maximal linear subspaces (which are smooth Schubert varieties) in the classical CHSS.  

In the case that $X_w$ is smooth, the results of Bryant and Walters were generalized to arbitrary CHSS and given a uniform proof by J. Hong.  

\begin{theorem*}[Hong {\cite{MR2276624}}]
Let $X$ be an irreducible compact Hermitian symmetric space in its minimal homogeneous embedding, excluding the quadrics of odd dimension.  Let $X_w \subset X$ be a smooth Schubert variety, excluding the non-maximal linear subspaces and $\bP^1 \subset C_n / P_n$.  Then $X_w$ is Schur rigid.\footnote[2]{The maximal linear space $\bP^1 \subset C_n / P_n$ was accidentally omitted in Hong's theorem.}
\end{theorem*}

It is straight-forward to see that the cases excluded from Hong's theorem are not Schur rigid:  when $H_{2|w|}(X)$ is generated by a single Schubert variety $X_w$, every $|w|$--dimensional subvariety $Y \subset X$ must satisfy $[Y] = r[X_w]$ for some $r \in \bZ_{>0}$.  Thus, every $|w|$--dimensional $Y$ is an integral variety of $\cR_w$.  This is the case when $X$ is a projective space or an odd dimensional quadric hypersurface.  For similar reasons, we can rule out the cases where $X_w$ is a non-maximal linear space or $X_w = \bP^1 \subset C_n / P_n$.

Making use of \cite{MR2276624} and the foliation structure of Schubert varieties,  Hong also proved that a large class of the singular Schubert varieties in the Grassmannian are Schur rigid \cite{MR2191767}.

There are variants of this `smoothing problem.'  For example, I. Choe and Hong have studied the related notion of Schubert rigidity of linear subspaces of arbitrary homogeneous varieties with second betti number equal to one \cite{MR2076680}.  In another direction, S. Kleiman considered the problem of deforming a cycle $Z$ in an arbitrary projective variety by rational equivalence into the difference of two effective cycles $Z_1 - Z_2$ whose prime components are smooth.  In \cite{MR0285535} Kleiman and J. Landolfi specialized the problem to case the projective variety is a Grassmannian.

\subsection{Strategy} 

The problem is approached as follows.  The bundle $\cR_w$ contains the sub-bundle $\cB_w$ of ${|w|}$--planes tangent to a smooth point of $g\cdot X_w$, for some $g \in G$.  When the only integral varieties of $\cB_w$ are $g\cdot X_w$, we say $X_w$ is \emph{Schubert rigid}; otherwise $X_w$ is \emph{Schubert flexible}.  Hong proved that $X_w$ is Schur rigid if and only if $X_w$ is Schubert rigid and $\cB_w = \cR_w$ \cite{MR2276624}.   Because the Schubert system $\cB_w$ is more amenable to analysis than the Schur system $\cR_w$, the general approach to this problem is to first show that $X_w$ is Schubert rigid, and then prove that $\cB_w = \cR_w$.

The key observation in Hong's proof for the smooth Schubert varieties is that the (partial) vanishing of a certain Lie algebra cohomology group implies the Schubert rigidity of  $X_w$.  When the Schubert variety is smooth, the cohomology group satisfies the hypotheses of Kostant's famous theorem \cite[Theorem 5.14]{MR0142696}, and it is straightforward to determine when the vanishing holds.  Hong then directly computes $\cB_w = \cR_w$, establishing Schur rigidity.

The key difficulty in extending Hong's approach to the singular Schubert varieties is the absence of a Kostant--type theorem for the associated Lie algebra cohomology.  (We do not see a natural extension of Hong's method in the case that $X = \tGr(m,n)$ to arbitrary CHSS.)  

\subsection{Contents} 

The main result of this paper is Theorem \ref{T:RR} which identifies the Schubert varieties for which there exist first-order obstructions to Schur flexibility.    The theorem recovers (and extends) the results of Walters, Bryant and Hong to the general case.  The varieties are described in terms of a new characterization of the $X_w$ by a nonnegative integer $\sfa(w)$ and a marking $\ttJ(w)$ of the Dynkin diagram of $G$ (Proposition \ref{P:nw2graded} and Corollary \ref{C:bigone}).  This description is the \emph{sine qua non} of our analysis of the Schubert system $\cB_w$ and the equality $\cB_w = \cR_w$.  It generalizes the well-known characterization of the smooth Schubert varieties by connected sub-diagrams of the Dynkin diagram of $G$ -- the smooth $X_w$ correspond to $\sfa(w) = 0$ (Proposition \ref{P:a=0}).

Sections \ref{S:homogvars} and \ref{S:CHSS} present the necessary definitions and background on homogeneous varieties, their Schubert subvarieties and compact Hermitian symmetric spaces.  In Section \ref{S:Bsys} the Schubert system is lifted to a frame bundle where the analysis is performed.  To each $X_w$ there is associated a Lie algebra cohomology group $H^1(\fn_w,\fg_w^\perp)$ whose vanishing (in positive degree) ensures Schubert rigidity.  Following the constructions of Kostant in \cite{MR0142696} we define a Laplacian $\square$ in Section \ref{S:cohcomp} and show that there is a natural bijection between $H^1(\fn_w , \fg^\perp_w)$ and $\tker\,\square =: \cH^1$ (Proposition \ref{P:new_coh}).  

There is a reductive Lie algebra $\fg_{0,0}\subset\fg$ with respect to which the  Laplacian acts as a $\fg_{0,0}$--module morphism $\square : \fg^\perp_w\ot\fn_w^* \to \fg^\perp_w\ot\fn_w^*$.  Schur's Lemma implies that $\fg^\perp_w\ot\fn_w^*$ admits a $\fg_{0,0}$--module decomposition into $\square$--eigenspaces.  (The eigenvalues are non-negative.)  Through analysis of the spectrum of $\square$ we see that the desired vanishing occurs if and only if two representation theoretic conditions hold.  In particular, $X_w$ is Schubert rigid when the two conditions are satisfied (Theorem \ref{T:Bw_rigid}).  

Theorem \ref{T:BR} lists the corresponding Schubert rigid $X_w$.  This completes the first step of the approach to the Schur rigidity problem.  It remains to determine when $\cB_w = \cR_w$.  In Section \ref{S:schur} we develop a test \eqref{E:B=Rtest} to determine when $\cB_w = \cR_w$ holds.  In Section \ref{S:RR} we apply the test to the varieties of Theorem \ref{T:BR}; we find that they are all Schur rigid (Theorem \ref{T:RR}).

The present paper leaves a question to address.  The Schubert varieties listed in Theorem \ref{T:BR} are those for which there exist first-order obstructions to the existence of nontrivial integral varieties of the Schur system.  (An integral variety is \emph{trivial} if it is of the form $g\cdot X_w$ for some $g \in G$.)  In particular, if $X_w$ is not listed in Theorem \ref{T:BR}, it does not immediately follow that $X_w$ is Schur flexible: there may exist higher-order obstructions to the existence of nontrivial integral varieties.  It remains to determine which of these $X_w$ are flexible.  This problem will be addressed in a sequel.

\subsection*{Acknowledgements}  
We thank J.M. Landsberg for suggesting the problem.  We are grateful to Landsberg, Bryant, Hong and N. Ressayre for illuminating discussions.  

\setcounter{tocdepth}{1}
\tableofcontents\listoftables

\section{Homogeneous varieties}  \label{S:homogvars}
\subsection{Notation} \label{S:not}
We employ the root and weight conventions of \cite{MR1920389}.  

Let $\fg$ be a complex semi-simple Lie algebra.  Fix a Borel subalgebra $\fb \subset \fg$, and let $\fh$ be the associated Cartan subalgebra.  Let $\Delta$ denote the roots of $\fg$, and $\Delta^\pm$ the positive and negative roots.  Given $\b \in \Delta$, let $\fg_\b \subset \fg$ denote the corresponding root space.  Given a direct sum $\fs \subset \fg$ of root spaces, let $\Delta(\fs) \subset \Delta$ denote the corresponding roots. That is, $\Delta(\fs)$ is defined by 
$$ \textstyle
  \fs \ = \ \bigoplus_{\b\in\Delta(\fs)} \fg_\b \, .
$$

Fix simple roots $\{ \a_1 , \ldots , \a_n \} = \Sigma \subset \Delta^+$.  Let $\fp \supset \fb$ denote the parabolic subalgebra generated by a subset $\Sigma_\fp \subset \Sigma$.  For example, $\Sigma_\fb = \emptyset$, and $\Sigma_\fp = \{ \a_1 , \ldots , \wh \a_\tti , \ldots , \a_n\}$ generates a maximal parabolic.  Let $I_\fp = \{ i \ | \ \a_i \in \Sigma_\fp \}$ denote the corresponding index set.  

Given a dominant integral weight $\nu$ of $\fg$, let $V_\nu$ denote the unique irreducible $\fg$--representation of highest weight $\nu$.  Let $\{ \w_1 , \ldots , \w_n \}$ denote the fundamental weights of $\fg$ and 
$$ \textstyle
  \rho \ := \ \textstyle \sum_{i=1}^n \w_i
  \ = \ \half \sum_{\a\in\Delta^+} \a \, , \quad \hbox{and} \quad
  \rho_\fp \ := \ \sum_{i \not\in I_\fp} \w_i \, .
$$
For example, if $\fp$ is maximal, $\rho_\fp = \w_\tti$.  For the Borel subalgebra, $\rho_\fb = \rho$.

Let $P \subset G$ be connected, complex semi-simple Lie groups with Lie algebras $\fp \subset \fg$.  The $G$--orbit $X \subset \bP V_{\rho_\fp}$ of the highest weight line in $V_{\rho_\fp}$ is the \emph{minimal homogeneous embedding of $G/P$}.  Let $o\in X$ denote the highest weight line.  As a $\fp$--module,
\begin{equation} \label{E:tgtsp}
  T_oX \ = \ \fg/\fp \, .
\end{equation}

Throughout, $g\cdot$ will denote the action of $g \in G$, and $\xi . $ the action of $\xi \in \fg$.

\subsection{Grading elements}\label{S:Z}

The parabolic $\fp$ determines a \emph{grading element} $Z = Z_\fp \in \fh$ by 
$$
  \a_i(Z) \ = \ \left\{ \begin{array}{ll}
    0 \,, & \hbox{ $\a_i \in \Sigma_\fp$,} \\
    1 \,, & \hbox{ $\a_i \not\in \Sigma_\fp$.}
  \end{array} \right.
$$
Let $\fg_j = \{ u \in \fg \ | \ [Z,u] = j\,u\}$ be the $Z$--eigenspace with eigenvalue $j$.  Then 
\begin{equation*} 
  \fg \ = \ \underbrace{ \fg_{q} \op \cdots \op \fg_1 }_{\fg_+} \
  \op \ \fg_0 \ \op \ 
  \underbrace{ \fg_{-1} \op \cdots \op \fg_{-q} }_{\fg_-} \, 
\end{equation*}
is a graded decomposition.  That is, 
  $[\fg_i\,,\,\fg_j] \subset \fg_{i+j} \, .$
In particular, each $\fg_j$ is a $\fg_0$--module.  Note that $\fh \subset \fg_0$ and $\fp = \fg_{\ge0}$.  So \eqref{E:tgtsp} yields a natural identification
\begin{equation} \label{E:tgtsp0}
  T_o X \ = \ \fg_{-} \, 
\end{equation}
as $\fg_0$--modules.  Additionally, $\fg_0$ is a reductive Lie subalgebra of $\fg$.  Indeed 
\begin{equation*} 
  \fg_0 \ = \ \fz \ \op \ \ff \, ,  
\end{equation*}
where $\fz \subset \fh$ is the center of $\fg_0$ and $\ff$ is the semi-simple subalgebra of $\fg$ with simple roots $\Sigma_\fp$.  Let $\{ Z_i \}$ be the basis of $\fh$ dual to the root basis $\{\a_i\}$ of $\fh^*$.  Then $\fz = \tspan \{ Z_i \ | \ i \not\in I_\fp \}$, and the grading element $Z = Z_\fp$ is the sum $\sum_{i \not\in I_\fp} Z_i$.  

\subsection{The Hasse diagram of $\fp$} \label{S:Hasse}

Let $W$ denote the Weyl group of $G$ and let $|w|$ denote the length of $w \in W$.  Define 
\begin{equation} \label{E:Dw}
  \Delta(w) \ := \ w \Delta^- \cap \Delta^+ \, .
\end{equation}

\begin{definition} \label{D:Hasse}
The \emph{Hasse diagram} \cite{MR1038279,MR2532439} of $P$ is 
\begin{eqnarray*}
  W^\fp & := & \{ w \in W \, | \, \Delta^+(\fg_0) \subset w ( \Delta^+) \} 
  \ = \ \{ w \in W \, | \, \Delta(w) \subset \Delta(\fg_+) \} \\
  & = & \{ w \in W \, | \, w(\lambda) \hbox{ is $\fg_0$--dominant } \forall 
           \hbox{ $\fg$--dominant weights } \lambda \} \, .
\end{eqnarray*}
\end{definition}

\begin{definition*}
A set $\Phi \subset \Delta$ is \emph{closed} or \emph{saturated} if given any two $\b,\c \in \Phi$ such that $\b+\c \in \Delta$, it is the case that $\b+\c \in \Phi$.
\end{definition*}

\begin{proposition} \label{P:Phi}
The mapping $w \mapsto \Delta(w)$ is a bijection of $W^\fp$ onto the family of all subsets $\Phi \subset \Delta(\fg_+)$ with the property that both $\Phi$ and $\Delta^+\backslash\Phi$ are closed.
\end{proposition}

\noindent For the proof, see \cite[Proposition 5.10]{MR0142696} or \cite[Proposition 3.2.12]{MR2532439}.

\begin{remark*}
We cite \cite{MR0142696, MR0142697, MR2532439} often in this paper.  We note that our $\Delta(w)$ is their $\Phi_w$.
\end{remark*}

\subsection{Schubert varieties} \label{S:schubvars}

Let $B = \texp(\fb) \subset G$ denote the Borel subgroup with Lie algebra $\fb$.  The Schubert varieties $X_w$ of $X$ are indexed by $w \in W^\fp$.  Their homology classes $[X_w]$ form a basis for the integral homology $H_*(X)$.   Let $C_w := Bw^{-1} \cdot o \subset X$ denote the \emph{Schubert cell}.  Then $X = \bigcup_{w \in W^\fp}\,C_w$ is a disjoint union.  Let $X_w := \overline{C_w} \subset X$ denote the \emph{Schubert variety}.  More generally, any $g \cdot X_w$, where $g \in G$, is a Schubert variety in $X$.  We will abuse notation by referring to any of these varieties as $X_w$.

The cell $w \cdot C_w$ is the orbit $N_w \cdot o$ of a unipotent subgroup $N_w \subset G$.  This is seen as follows.  Let $N \subset B$ be the maximal unipotent subgroup of the Borel.  Let $w_0 \in W$ be the longest element and $N^- := w_0 N w_0{}^{-1} \subset G$ the unipotent subgroup `opposite' to $N$.  Then $N_w := wNw^{-1} \cap N^-$.  The Lie algebra of $N_w$ is 
\begin{equation} \label{E:n_w}
  \fn_w \ := \ \bigoplus_{\a \in \Delta(w)} \fg_{-\a} \, .
\end{equation}
It is well-known that $|\Delta(w)| = {|w|}$ (\cite[Proposition 3.2.14 (3)]{MR2532439}).  Thus the Schubert variety $X_w$ is of dimension ${|w|}$.

\subsection{Conjugation and duality} \label{S:PD}

Recall that any automorphism $\varphi : \d_\fg \to \d_\fg$ of the Dynkin diagram induces automorphisms $\varphi: \Delta \to \Delta$ and $\varphi : W \to W$ of the root system and Weyl group.  The latter is given as follows: if $\sigma_j$ denotes the reflection corresponding to the simple root $\a_j$, and $w = \sigma_{i_1} \cdots \sigma_{i_r}$, then $\varphi(w)$ is $\sigma_{\varphi(i_1)} \cdots \sigma_{\varphi(i_r)}$.
 
\begin{definition*}
Given $w \in W$ and an automorphism $\varphi : \delta_\fg \rightarrow \delta_{\fg}$, let $w' = \varphi(w)$ denote the {\em $\varphi$--conjugate}.
\end{definition*}

\begin{remarks*}
It is clear from Definition \ref{D:Hasse} that $w \in W^{\fp_\tti}$ if and only if $w' \in W^{\fp_{\varphi(\tti)}}$.
The group of Dynkin diagram automorphisms is $\mathfrak{S}_3$ for $D_4$; $\bZ_2$ for $A_n$, $D_n$ ($n>4$) and $E_6$; trivial for the remaining complex simple Lie groups.
\end{remarks*}
 
The Weyl group of the semi-simple part $\ff$ of $\fg_0$ may be identified with the subgroup $W_\fp \subset W$ generated by $\{ \sigma_j \}_{j \in I_\fp}$.  Let $w^0_\fp$ be the longest word in $W_\fp$.  Then 
\begin{equation} \label{E:w0p}
  w^0_\fp \left( \Delta^+(\fg_0) \right) \ = \ \Delta^-(\fg_0) 
  \quad \hbox{ and } \quad
  w^0_\fp \left( \Delta(\fg_+) \right) \ = \ \Delta(\fg_+) \, .
\end{equation}
  
\begin{definition*}
The \emph{dual} of $w \in W$ is $w^* = w^0_\fp w w_0$.
\end{definition*}

\noindent Note that $\varphi(w^0_\fp w w_0) = \varphi(w^0_\fp) \varphi(w) \varphi(w_0) = \w^0_{\varphi(\fp)} \varphi(w) w_0$ implies
$$
  (w^*)' \ = \ (w')^* \, .
$$
Above, $\varphi(\fp)$ is the parabolic subalgebra with $I_{\varphi(\fp)} = \varphi(I_\fp)$.
 
\begin{lemma} \label{L:dual} \quad
\begin{a_list}
\item Duality is an involution $(w^*)^* = w$;
\item $w \in W^\fp$ if and only if $w^* \in W^\fp$;
\item $\Delta(w^*) = w^0_\fp(\Delta(\fg_+) \backslash \Delta(w)) = \Delta(\fg_+) \backslash w^0_\fp\Delta(w)$ for any $w \in W^\fp$.
\end{a_list}
\end{lemma}
 
\begin{proof}
\emph{(a)} is a consequence of $(w_0)^{-1} = w_0$ and $(w^0_\fp)^{-1} = w^0_\fp$.

To prove \emph{(b)}, let $\Killing{\cdot}{\cdot}$ denote the Killing form on $\fh^*$, $\lambda$ a $\fg$--dominant weight and $\a \in \Delta^+(\fg_0)$.  Then $\Killing{w^*(\lambda)}{\a}  =  \Killing{\lambda}{(w^*)^{-1}(\a)} =  \Killing{\lambda}{ w_0 w^{-1} w^0_\fp(\a)}$.  Since $w \in W^\fp$, we have $w^{-1}(\Delta^+(\fg_0)) \subset \Delta^+$.  Thus, $w^{-1} w^0_\fp(\a) \in \Delta^-$ and hence $w_0 w^{-1} w^0_\fp (\a) \in \Delta^+$.  From the $\fg$--dominance of $\lambda$ we conclude $\Killing{w^*(\lambda)}{\a} \geq 0$.  Thus, $w \in W^\fp$.  Conversely, if $w^* \in W^\fp$, then $w = (w^*)^* \in W^\fp$. 

The proof of \emph{(c)} requires the identities
\begin{i_list_nem}
\item $\Delta(ww_0) = \Delta^+ \backslash \Delta(w)$ for any $w \in W$;
\item $w^0_\fp \Delta(w) = \Delta(w^0_\fp w)\backslash\Delta^+(\fg_0)$ for any $w \in W^\fp$.
\end{i_list_nem}
The first identity is proved in \cite[p. 324]{MR2532439}.  To prove the second identity, note that \eqref{E:Dw}, Definition \ref{D:Hasse} and \eqref{E:w0p} imply $w^0_\fp \Delta(w) \subset \Delta(\fg_+)$ and $\Delta(w) = w \Delta^- \cap \Delta(\fg_+)$.  Therefore, $w^0_\fp\,\Delta(w) = w^0_\fp(w \Delta^-\cap \Delta(\fg_+)) = w^0_\fp w \Delta^- \cap \Delta(\fg_+) = \Delta(w^0_\fp w)\backslash\Delta^+(\fg_0)$.  
Hence,
\begin{eqnarray*}
 \Delta(\fg_+) \backslash w^0_\fp \Delta(w) 
 & \stackrel{\mathrm{(ii)}}{=} 
 & \Delta(\fg_+) \backslash (\Delta(w^0_\fp w)\backslash\Delta^+(\fg_0))
 \ = \ (\Delta^+ \backslash \Delta(w^0_\fp w)) \backslash \Delta^+(\fg_0) \\
 & \stackrel{\mathrm{(i)}}{=} 
 & \Delta(w^0_\fp w w_0) \backslash \Delta^+(\fg_0) 
 \ = \ \Delta(w^*) \backslash \Delta^+(\fg_0) \ = \ \Delta(w^*) \, ;
\end{eqnarray*}
the final equality is a consequence of part (b) of the lemma, and Definition \ref{D:Hasse}.
\end{proof}


Given $\a \in \Delta$, define $\a^* := -w^0_\fp(\a)$.  It follows from \eqref{E:w0p} that $\Delta(\fg_0)^* = \Delta(\fg_0)$.  Recall that $w^0_\fp (\Sigma_\fp) = -\Sigma_\fp$, cf. \cite[p. 324]{MR2532439}.  This induces $\ast : I_\fp \to I_\fp$ mapping $j \mapsto j^*$ by
\begin{equation} \label{E:j*}
  (\a_j)^* \ = \ -w^0_\fp(\a_j) \ = \ \a_{j^*} \, .
\end{equation}

\begin{remark} \label{R:dual}
Since $W$ preserves the Killing form $\Killing{\cdot}{\cdot}$, we have $\Killing{\a^*}{\beta^*} = \Killing{\a}{\beta}$ for any $\a,\beta \in \Delta$.  This implies that the map $*:I_\fp \to I_\fp$ corresponds to a Dynkin diagram automorphism of the subgraph $\d_\fp \subset \delta_\fg$ generated by $I_\fp$.
\end{remark} 

\section{Compact Hermitian symmetric spaces} \label{S:CHSS}

\subsection{Definition} \label{S:dfnCHSS}

The irreducible compact Hermitian symmetric spaces (CHSS, Table \ref{t:CHSS}) are those $G/P$ with $G$ simple and graded decomposition (cf. \S\ref{S:Z})
\begin{equation} \label{E:CHSS_gr}
  \fg \ = \ \fg_{-1} \op \fg_0 \op \fg_1 \, .
\end{equation}
In particular, the nilpotent subalgebras $\fg_+ = \fg_1$ and $\fg_- = \fg_{-1}$ are abelian 
\begin{equation} \label{E:abelian}
  [\fg_1 , \fg_1] \ = \ \{0\} \ = \ [\fg_{-1} , \fg_{-1}] \, .
\end{equation}
The parabolic $P$ is always maximal.  So $\Sigma \backslash \Sigma_\fp$ consists of a single simple root $\a_\tti$ and $V_{\rho_\fp} = V_{\w_\tti}$. 
\begin{table}[h] \renewcommand{\arraystretch}{1.2}
\caption{The CHSS.}
\begin{tabular}{|c|l|}
\hline
 & Grassmannians $\tGr(\tti,n+1) = A_n/P_\tti$. \\ \cline{2-2}
Classical & Quadric hypersurfaces $B_n/P_1$ with $n\ge 2$, and $D_n/P_1$ with $n \ge 4$. \\ \cline{2-2}
CHSS & Lagrangian grassmannians $C_n/P_n$ with $n\ge3$. \\ \cline{2-2}
 & Spinor varieties $D_n/P_{n-1} \simeq D_n/P_n$ with $n\ge4$. \\ \hline
Exceptional & The Cayley plane $E_6/P_1 \simeq E_6/P_6$. \\ \cline{2-2}
CHSS & The Freudenthal variety $E_7/P_7$. \\ \hline
\end{tabular} 
\label{t:CHSS}
\end{table}

\subsection{Characterization of \boldmath $\fn_w$ \unboldmath for CHSS} \label{S:CHSS_Dw}

Let $Z_\tti$ be the grading element associated to the CHSS $G/P_\tti$ (\S\ref{S:Z}), and $\fg = \fg_{-1} \op \fg_0 \op \fg_1$ the $Z_\tti$--graded decomposition of $\fg$ (\S\ref{S:dfnCHSS}).  Let $\tilde\a$ denote the highest root of $\Delta(\fg_1)$.

Given $J \subset I_\fp= \{ 1,\ldots,n\} \backslash \{\tti\}$, define $Z_J = \sum_{s\in J} Z_s$.  Let $\fg = \op \fg_{j,k}$ be the $(Z_\tti,Z_J)$--bigraded decomposition of $\fg$: that is,
\begin{equation} \label{E:ZiZw}
   \fg_{j,k} \ := \ 
   \{ u \in \fg \ | \ [Z_\tti,u] = j \, u \, , \ [Z_J,u] = k \, u \} \,.
\end{equation}

\begin{lemma} \label{L:graded2nw}
Given an integer $0 \le a \le \tilde\a(Z_J)$, there exists a unique $w = w(J,a) \in W^\fp$ such that $\fg_{-1,0} \op \cdots \op \fg_{-1,-a} = \fn_w$.
\end{lemma}

\begin{proof}
Define $\Phi  = \Delta( \fg_{1,0} \op\cdots\op \fg_{1,a} ) = \{ \a \in \Delta \ | \ \a(Z_\tti) = 1 \, , \ \a(Z_J) \le a \} \subset \Delta(\fg_1)$.  We will show that $\Phi = \Delta(w)$ for some $w \in W^\fp$.  By Proposition \ref{P:Phi} it suffices to show that $\Phi$ and $\Delta^+ \backslash \Phi$ are closed.  Since $[\fg_1,\fg_1] = \{0\}$, every subset of $\Delta(\fg_1)$ is closed.  Hence $\Phi = \Delta(w)$ for some $w \in W^\fp$ if and only if $\Delta^+ \backslash \Phi$ is closed.  This follows immediately from $\Delta^+\backslash\Phi = \Delta(\fg_{1,>a}) \sqcup \Delta^+(\fg_{0,\ge0})$ and $[\fg_{0,\ge0} \op \fg_{1,>a} , \fg_{0,\ge0} \op \fg_{1,>a}] \subset \fg_{0,\ge0} \op \fg_{1,>a}$.  Thus, $w = w(J,a)$ exists and moreover it is unique by Proposition \ref{P:Phi}.
\end{proof}

We will see below that Lemma \ref{L:graded2nw} has a converse.  We identify $\Sigma_\fp = \{ \a_j \ | \ j \not= \tti\}$ with the simple roots of $\ff$.  The stabilizer $\fq^w \subset \ff$ of $\fn_w \in \tGr(|w|,\fg_{-1})$ is a parabolic subalgebra.  (See Section \ref{S:schubert}.)  Let $I_{\fq^w} = I_w \subset I_\fp$ be the corresponding index set (\S\ref{S:not}), and set 
$$
  \ttJ \ = \ \ttJ(w) \ = \ I_\fp \backslash I_w \, .
$$
Note that
\begin{equation} \label{E:Jw}
  \ttj \in \ttJ \ \Longleftrightarrow \ \fg_{-\a_\ttj}
  \hbox{ does \emph{not} stabilize } \fn_w \, .
\end{equation}

\begin{remark} \label{R:Jw}
Note that $\ttJ = \emptyset$ if and only if $X_w = o$ or $X_w = X$.  We will assume that \emph{$X_w$ is a proper subvariety of $X$}.
\end{remark}

\noindent Define 
\begin{equation} \label{E:Zw}
  Z_w \ := \ Z_{\ttJ} \ = \ \sum_{\ttj \in \ttJ} Z_\ttj \quad \hbox{ and } \quad
  \sfa \ = \ \sfa(w) \ := \ 
  \tmax \{ \a(Z_w) \ | \ \a\in\Delta(w) \} \ \in \ \bZ_{\ge0} \, .
\end{equation}
The $(Z_\tti,Z_w)$--bigraded decomposition $\fg = \op \fg_{j,k}$ is given by \eqref{E:ZiZw} and
\begin{equation} \label{E:nwstab}
\hbox{$\fg_{0,\ge0} = \fz \op \fq^w$ is the stabilizer in $\fg_0$ of $\fn_w \subset \fg_{-1}$.}
\end{equation}  

We may now state the converse to Lemma \ref{L:graded2nw}.

\begin{proposition} \label{P:nw2graded}
Let $G/P_\tti$ be an irreducible compact Hermitian symmetric space.  Given $w \in W^\fp$, let $\fg = \op \fg_{j,k}$ be the $(Z_\tti,Z_w)$--bigraded decomposition \eqref{E:ZiZw} of $\fg$.  Then 
\begin{equation} \label{E:nw}
  \fn_w \ = \ \fg_{-1,0} \ \op \ \cdots \ \op \ \fg_{-1,-\sfa(w)} \, .
\end{equation}
\end{proposition}

\noindent See Corollary \ref{C:bigone} for a description of the $(\sfa,\ttJ)$ pairs that occur.  The smooth $X_w$ are characterized by $\sfa(w)=0$ (Proposition \ref{P:a=0}).

The proof of the proposition is given in five lemmas.  The final lemma is proved in a case-by-case argument for each of the classical CHSS (Table \ref{t:CHSS}).  We used the representation theory software LiE \cite{LiE} to confirm Proposition \ref{P:nw2graded} for the two exceptional cases.  The first lemma is due to Kostant.  See \cite[Theorem 8.13.3]{MR928600} for a more general statement and proof.

\begin{lemma}[Kostant] \label{L:Kostant}
Let $\fa$ be a complex semi-simple Lie algebra with a choice of Cartan subalgebra $\ft$ and simple roots $\{ \a_1,\ldots,\a_n \}$.  Define $Z_i \in \ft$ by $\a_i(Z_j) = \d_{ij}$.  Any $I =\{i_1,\ldots,i_p\} \subset \{1,\ldots,n\}$ defines a multi-graded decomposition $\fa = \bigoplus_{A\in\bZ^p} \fa_A$ by $\fa_A = \fa_{a_1,\ldots,a_\sfp} := \{ u \in \fa \ | \ [Z_{i_m} \,,\, u] = a_m\,u \,,\ \forall \ 1\le m\le q\}$.  Each $\fa_A \not= \fa_{0,\ldots,0} =: \fa_0$ is an irreducible $\fa_0$--module.
\end{lemma}

Define 
$$ 
  \{ \ttj_1 \,,\, \ttj_2\,,\ldots,\, \ttj_\sfp \} \ := \ \ttJ
  \quad \hbox{ with } \quad \ttj_1 < \ttj_2 < \cdots < \ttj_\sfp \,.
$$
Let $\fg_{-1} = \op_A \fg_{-1,-A}$ denote the multi-graded decomposition induced by $(Z_{\ttj_1},\ldots,Z_{\ttj_\sfp})$.  Since the roots $\Delta(\fg_{-1})$ are non-positive integer linear combinations of the simple roots, we have $A = (a_1,\ldots,a_\sfp)\ge(0,\ldots,0)$.  Define 
$$
  |A| \ := \ a_1 + \cdots + a_\sfp \,.
$$
Note that $|A| = \a(Z_w)$ for all $\a \in \Delta(\fg_{1,A})$, so that $\fg_{1,A} \subset \fg_{1,|A|}$.
Let
$$
  \fn_w^\perp \ := \bigoplus_{\a\in\Delta(\fg_1)\backslash\Delta(w)} 
  \fg_{-\a}\, .
$$

\begin{definition} \label{D:B<A}
Define $B \le A$ if there exists a sequence of $\fg_{0,-C_j} \subset \fg_{0,-1}$ such that $\fg_{1,B} = [\fg_{0,-C_1} \,,\, [ \fg_{0,-C_2} \,,\,\cdots [\fg_{0,-C_r}\,,\,\fg_{1,A}]\cdots]]$.   That is, $\fg_{1,B}$ may be obtained from $\fg_{1,A}$ by successive brackets with the irreducible $\fg_{0,0}$--submodules $\fg_{0,-C}$ of $\fg_{0,-1}$.
\end{definition}

\begin{lemma} \label{L:g_1A}
If $\fg_{-1,-A} \subset \fn_w$ and $B < A$, then $\fg_{-1,-B} \subset \fn_w$.
\end{lemma}

\begin{proof}
It suffices to consider the case that $\fg_{1,B} = [\fg_{0,-C} \,,\, \fg_{1,A}]$; the general case will follow by induction.  By Kostant's Lemma \ref{L:Kostant}, $\fg_{-1,-B}$ is an irreducible $\fg_{0,0}$--module.  So either $\fg_{-1,-B} \subset \fn_w$, or $\fg_{-1,-B} \subset \fn_w^\perp$.  Observe that $[\fg_{0,C} \,,\, \fg_{-1,-A}] = \fg_{-1,-B}$ if and only if $[\fg_{0,-C} \,,\, \fg_{-1,-B}] = \fg_{-1,-A}$.  It follows from Proposition \ref{P:Phi} that $\fg_{-1,-B} \subset \fn_w$.
\end{proof}

\begin{lemma} \label{L:nondeg}
Define $\sfm = \tilde\a(Z_w)$, where $\tilde\a$ is the highest root of $\Delta(\fg_1)$.

\emph{(a)}  If $\b\in\Delta(\fg_1)$ and $\b(Z_w) < \sfm$, then $[\fg_{0,1} \,,\, \fg_\b] \not= \{0\}$.

\emph{(b)}  If $0 \le b < \sfm$, then $\fg_{1,b} = \left[ \fg_{0,-1} \, , \, \fg_{1,b+1} \right]$. 
\end{lemma}

\begin{proof}
Argue by contradiction: suppose that $[\fg_{0,1} \,,\, \fg_\b] = \{0\}$.  Equivalently, the intersection of $\Delta(\fg_{1,b+1})$ with $\Delta(\fg_{0,1}) + \b$ is empty.  This implies $\fg_\b \not\subset [\fg_{0,-1} \,,\, \fg_{1,b+1} ]$, where $b = |B|$.  This is a contradiction as the root space $\fg_{\b}$ is obtained from the highest root space $\fg_{\tilde\a} \subset \fg_{1,\sfm}$ by successive brackets with the $\{\fg_{-\a_j} \ | \ \a_j$ a simple root$\}$.  This establishes (a) and (b).
\end{proof}

\noindent
The following lemma will establish Proposition \ref{P:nw2graded}.

\begin{lemma} \label{L:bigone}
Assume $G/P$ is one of the classical CHSS listed (cf. Table \ref{t:CHSS}).  If $|B| \le \sfa$, then $\fg_{-1,-B} \subset \fn_w$.  
\end{lemma}

\begin{proof}
From Lemma \ref{L:nondeg} we see that if $\fg_{1,B} \subset \fg_1$ and $|B| < \sfa$, then there exists $\fg_{1,B'} \subset \fg_{1}$ with $B < B'$ and $|B'| = \sfa$.  So it follows from Lemma \ref{L:g_1A} that it suffices to show that $|B| = \sfa$ implies $\fg_{-1,-B} \in \fn_w$.   

We will argue by contradiction, showing that the existence of $\fg_{-1,-B} \not\subset \fn_w$ with $|B| = \sfa$ implies that $\fn_w$ is stabilized by a simple root space $\fg_{-\a_\ell} \subset \fg_{0,-}$ that is not contained in the stabilizer $\fg_{0,\ge0}$ of $\fn_w$ in $\fg_0$.  The reader should be aware that the argument implicitly makes (very) frequent use of Lemma \ref{L:g_1A}, often without mention.  The proof proceeds case-by-case through the classical CHSS.  The assertions below regarding the expressions for $\a \in \Delta(\fg_1)$ as integral linear combinations of the simple roots $\{\a_1,\ldots,\a_n\}$ may be found in standard representation theory texts.

For notational convenience we will use superscripts to write multi-indices compactly.  For example,
$$
  ( 0^k \,,\, 1^\ell \,,\, 2^m) \ := \ 
  ( \underbrace{0,\ldots,0}_{k \ \mathrm{terms}} \,,\, 
    \underbrace{1,\ldots,1}_{\ell \ \mathrm{terms}} \,,\,
    \underbrace{2,\ldots,2}_{m \ \mathrm{terms}} ) \, .
$$

{\bf (I)} We begin with $B_n/P_1$.  The roots $\a \in \Delta(\fg_1)$ are of the form $\a = \a_1 + \cdots + \a_j$ or $\a = \a_1+\cdots+\a_{j-1} +2(\a_j+\cdots+\a_n)$, with $1<j\le n$.  Thus the $B = (b_1,\ldots,b_p)$ with $|B| = \sfa$ are of the form $(1^\sfa,0^{\sfp-\sfa})$ or $(1^{\sfa-2m},2^m)$, with $2m\le \sfa$.  Observe that $\sfp \le \sfa+1$, else Lemma \ref{L:g_1A} and some thought imply $\fg_{-\a_{\ttj_\sfp}} \subset \fg_{0,-}$ stabilizes $\fn_w$.  Similarly, $\sfa \le \sfp$, else $\fg_{-\a_{\ttj_\sfp}} \subset \fg_{0,-}$ stabilizes $\fn_w$.

If $\sfp = \sfa$, then $B = (1^\sfa)$ is uniquely determined and the lemma follows, in this case, from an application of Lemma \ref{L:g_1A}.  Moreover, if $\ttj \in \ttJ -\{\ttj_\sfp\}$, then $\fg_{-\a_\ttj} \subset \fg_{0,-1}$ stabilizes $\fn_w$.  Therefore $\sfp=1$, so that $\ttJ = \{\ttj_1\}$, and $\sfa = 1$.

If $\sfp = \sfa+1$, then $B = (1^\sfa,0)$ is uniquely determined and the lemma again follows from Lemma \ref{L:g_1A}.  As in the previous paragraph, we also see that $\ttJ = \{\ttj_1\}$ and $\sfa = 0$.

\smallskip

{\bf (II)} Now consider $D_n/P_1$.  The roots $\a \in \Delta(\fg_1)$ are of the form $\a = \a_1 + \cdots + \a_i$, with $i \le n$; $\a = \a_1 + \cdots + \a_{n-2} + \a_n$ or $\a = \a_1 + \cdots \a_{j-1} + 2(\a_j + \cdots + \a_{n-2}) + \a_{n-1} + \a_n$, with $1<j < n-1$.  

{\bf (II.A)}  Suppose $n-1,n \in \ttJ$.  The $B = (b_1,\ldots,b_p)$ with $|B| = \sfa$ are of the form $(1^\sfa,0^{\sfp-\sfa})$, if $\sfp \ge \sfa+2$; $(1^\sfa,0)$ or $(1^{\sfa-1},0,1)$, if $\sfp = \sfa+1$; $(1^{\sfa-2m-2},2^m,1,1)$, if $\sfp \le \sfa$.  Note that $\sfa+1\le p \le \sfa+2$, else $\fg_{-\a_n}\subset\fg_{0,-1}$ stabilizes $\fn_w$.  If $\ttj \in \ttJ -\{n-1,n\}$, then $\fg_{-\a_\ttj}$ stabilizes $\fn_w$.  Thus, $\ttJ = \{n-1,n\}$ and $\sfa = 0,1$.

If $\sfa=0$ then the lemma is immediate from the irreducibility (Lemma \ref{L:Kostant}) of $\fg_{-1,0}$.  If $\sfa = 1$, then $\fg_{-1,-1} = \fg_{-1,(-1,0)} \op \fg_{-1,(0,-1)}$.  Without loss of generality $\fg_{-1,(0,-1)} \subset \fn_w$.  If $\fg_{-1,(-1,0)} \not\subset \fn_w$, then $\fg_{-\a_n} \subset \fg_{0,-1}$ stabilizes $\fn_w$, a contradiction.  Lemma \ref{L:bigone} now follows, in this case, from Lemma \ref{L:g_1A}.

{\bf (II.B)}  Suppose $n \in \ttJ$ and $n-1\not\in \ttJ$.   The $B = (b_1,\ldots,b_p)$ with $|B| = \sfa$ are of the form: $(1^\sfa,0^{\sfp-\sfa})$, if $\sfa\le \sfp$; or $(1^{\sfa-2m-1},2^m,1)$, if $\sfp \le \sfa$.  Note that $\sfp = \sfa+1$, else $\fg_{-\a_n}\subset \fg_{0,-1}$ stabilizes $\fn_w$.  If $\ttj \in \ttJ - \{n\}$, then $\fg_{-\a_\ttj} \subset \fg_{0,-1}$ stabilizes $\fn_w$.  Thus $\ttJ = \{n\}$ and $\sfa=0$.

A similar argument will show that, if $n-1\in \ttJ$ and $n\not\in \ttJ$, then $\ttJ = \{n-1\}$ and $\sfa = 0$.  In both cases the lemma follows from the fact (Lemma \ref{L:Kostant}) that $\fg_{-1,-\sfa} = \fg_{-1,0}$ is irreducible.

{\bf (II.C)}  Suppose that $n-1,n\not\in \ttJ$.  The $B = (b_1,\ldots,b_p)$ with $|B| = \sfa$ are of the form: $(1^\sfa,0^{\sfp-\sfa})$, if $\sfa\le \sfp$; or $(1^{\sfa-2m},2^m)$, if $\sfp=\sfa-m\le \sfa$.  Note that $\sfa \le p \le \sfa+1$, else $\fg_{-\a_{\ttj_\sfp}} \subset \fg_{0,-1}$ stabilizes $\fn_w$.  In either case, if $\ttj \in \ttJ - \{\ttj_\sfp\}$, then $\fg_{-\a_\ttj} \subset \fg_{0,-1}$ stabilizes $\fn_w$.  Thus $\sfp=1$, so that $\ttJ = \{\ttj_1\}$, and $\sfa = 0,1$.  

If $\sfa=0$, then the lemma follows from the irreducibility (Lemma \ref{L:Kostant}) of $\fg_{-1,0}$.  If $\sfa=1$, then the lemma follows from both Lemma \ref{L:g_1A} and the irreducibility of $\fg_{-1,0}$ and $\fg_{-1,-1}$.

\smallskip

{\bf (III)}  Next consider $C_n/P_n$.  The roots $\a \in \Delta(\fg_1)$ are of the form $\a = \a_i + \cdots + \a_n$, $\a = \a_i + \cdots + \a_{j-1} + 2(\a_j + \cdots + \a_{n-1} ) + \a_n$ or $\a = 2(\a_j + \cdots + \a_{n-1}) + \a_n$.  So the $B = (b_1,\ldots,b_p)$ with $|B| = \sfa$ are of the form $B_m = (0^{\sfp-\sfa+m} , 1^{\sfa-2m} , 2^m)$, where $0\le2m \le \sfa$.  Arguing as above we note that $\sfp \le \sfa+1$, else the root space $\fg_{-\a_{\ttj_1}} \subset \fg_{0,-}$ stabilizes $\fn_w$.  Similarly, $\sfa \le \sfp$ else $\fg_{-\a_{\ttj_\sfp}} \subset \fg_{0,-}$ stabilizes $\fn_w$.

By definition of $\sfa$, there exists $A = B_{m(A)}$ such that $|A| = \sfa$ and $\fg_{-1,-A}\subset\fn_w$.  Suppose that $\fg_{-1,-B_{m(A)+1}} \not\subset \fn_w$.  Let $s = p-\ell(A)-m(A)+1$ be the number of leading zeros in $B_{m(A)+1}$.  Then, with some thought and Lemma \ref{L:g_1A}, we see that the root space $\fg_{-\a_{\ttj_{s}}} \subset \fg_{0,-}$ stabilizes $\fn_w$, a contradiction.  Arguing by induction, we conclude that $\fg_{-1,-B_m} \subset \fn_w$ for all $m \ge m(A)$.

For the other direction, suppose that $\fg_{-1,-B_{m(A)-1}} \not\subset \fn_w$.  Let $s = p-m(A)+1$ be the number of leading 0's and 1's in $B_{m(A)-1}$.  Then, with some thought and Lemma \ref{L:g_1A}, we see that the root space $\fg_{-\a_{\ttj_s}} \subset \fg_{0,-}$ stabilizes $\fn_w$.  Arguing by induction, we conclude that $\fg_{-1,-B_m} \subset \fn_w$ for all $m \le m(A)$.  This establishes the lemma for $C_n/P_n$.

\smallskip

{\bf (IV)}  Next consider $D_n/P_n$.  The roots of $\Delta(\fg_1)$ are of the form $\a = \a_i + \cdots + \a_n$, $\a = \a_i + \cdots + \a_{n-2} + \a_n$ and $\a = \a_i + \cdots + \a_{j-1} + 2 ( \a_j + \cdots + \a_{n-2}) + \a_{n-1} + \a_n$.

{\bf (IV.A)}  Suppose $n-1 \not\in \ttJ$.  The $B = (b_1,\ldots,b_p)$ with $|B| = \sfa$ are of the form $B_m = (0^{\sfp-\sfa+m},1^{\sfa-2m},2^m)$ with $0 \le 2m \le \sfa$.  Note that $\sfp \le \sfa+1$ else $\fg_{-\a_{\ttj_1}} \subset \fg_{0,-1}$ stabilizes $\fn_w$.  Similarly, $\sfp \ge \sfa$ else $\fg_{-\a_{\ttj_\sfp}} \subset \fg_{0,-1}$ stabilizes $\fn_w$.  Thus, $\sfa \le p \le \sfa+1$.

By definition of $A$, there exists an $A = B_{m(A)}$ such that $\fg_{-1,-A}\subset \fn_w$.  Suppose that $\fg_{-1,-B_{m(A)+1}} \not\subset \fn_w$.  Let $s = p-\sfa+m(A)+1$ be the number of leading zeros in $B_{m(A)+1}$.  Then $\fg_{-\a_{\ttj_s}} \subset \fg_{0,-}$ stabilizes $\fn_w$, a contradiction.  Arguing by induction, we conclude that $\fg_{-1,-B_m} \subset \fn_w$ for all $m \ge m(A)$.

In the other direction, suppose that $\fg_{-1,-B_{m(A)-1}} \not\subset \fn_w$.  Let $s = p-m(A)+1$ be the number of leading 0's and 1's in $B_{m(A)-1}$.  Then, with some thought, we see that the root space $\fg_{-\a_{\ttj_s}} \subset \fg_{0,-}$ stabilizes $\fn_w$.  Arguing by induction, we conclude that $\fg_{-1,-B_m} \subset \fn_w$ for all $m \le m(A)$.  This establishes the lemma for $n-1 \not\in \ttJ$.  The following observations will be helpful in Section \ref{S:Dn}.

\begin{subequations} \label{R:Dn}
\begin{remark} \label{R:Dn.A}
Assume that $n-1 \not\in \ttJ$.  If $\sfa=1$, then $\ttJ \not=\{1\}$ (else $\fn_w = \fg_{-1}$ and $X_w = G/P$).

Consider the tuple $B = (0^{\sfp-\sfr},2^\sfr)$ with $\sfr\ge1$, which occurs when $\sfa = 2\sfr\ge2$.  This tuple will only appear if $1<\ttj_{\sfp-\sfr+1} - \ttj_{\sfp-\sfr}$.  If it does not appear, then $\fg_{-\a_{\ttj_{\sfp-\sfr}}}$ will stabilize $\fn_w$, a contradiction.  We may conclude that $1 < \ttj_{\sfp-\sfr+1} - \ttj_{\sfp-\sfr}$.
%

Consider the tuple $B = (0^{\sfp-\sfr},1,2^{\sfr-1})$ with $\sfr\ge1$, which occurs when $\sfa = 2\sfr-1\ge1$.  Note that $\fg_{-\a_{\ttj_{\sfp-\sfr+1}}}$ will stabilize $\fn_w$ if $1 = \ttj_{\sfp-\sfr+1} - \ttj_{\sfp-\sfr}$.  Therefore $1< \ttj_{\sfp-\sfr+1} - \ttj_{\sfp-\sfr}$.
\end{remark}

{\bf (IV.B)}  Suppose $n-1 \in \ttJ$.  The $B = (b_1,\ldots,b_p)$ with $|B| = \sfa$ are of the form $B^0 = (0^{\sfp-\sfa-1},1^\sfa,0)$, or $B_m = (0^{\sfp-\sfa+m},1^{\sfa-2m-1},2^m,1)$ with $0 \le 2m \le \sfa-1$.  Note that $\sfp \le \sfa+2$, else $\fg_{-\a_{\ttj_1}} \subset \fg_{0,-1}$ stabilizes $\fn_w$.  Similarly $\sfa+1 \le \sfp$, else $\fg_{-\a_{n-1}} \subset \fg_{0,-1}$ stabilizes $\fn_w$.  

An argument similar to that for $n-1 \not \in \ttJ$ implies if $\fg_{-1,-B_m} \subset \fn_w$ for some $m$, then $\fg_{-1,-B_m} \subset \fn_w$ for every $m$.  We leave this to the reader.

Assume that no $B_{-1,B_m} \subset \fn_w$.  Then by definition of $\sfa$ the irreducible $\fg_{-1,-B^0}$ must lie in $\fn_w$.  With some thought we see that $\fg_{-\a_{\ttj_{\sfp-\sfa}}} \subset \fg_{0,-}$ will stabilize $\fn_w$, a contradiction.  So we may conclude that every $\fg_{-1,-B_m}$ lies in $\fn_w$.  It remains to show that $\fg_{-1,-B^0}$ also lies in $\fn_w$.  But if it does not, then $\fg_{-\a_{\ttj_1}} \subset \fg_{0,-}$ stabilizes $\fn_w$, a contradiction.  This establishes the lemma for $D_n/P_n$.  The following observations will be helpful in Section \ref{S:Dn}.

\begin{remark} \label{R:Dn.B}
Assume $n-1 \in \ttJ$.  If $\sfa=0$, then $n-2 \not\in \ttJ$ (else $\fg_{-\a_{n-1}}$ stabilizes $\fn_w$, contradicting \eqref{E:Jw}).  If $\sfa=1$, then $n-2 \not \in \ttJ$ (else $\fg_{-\a_{n-2}}$ stabilizes $\fn_w$).
%

Consider the tuple $B = (0^{\sfp-\sfr-1},2^\sfr,1)$ with $\sfr\ge1$, which occurs when $\sfa = 2\sfr+1\ge3$.  This tuple will appear only if $1 < \ttj_{\sfp-\sfr} - \ttj_{\sfp-\sfr-1}$.  If the tuple does not appear, $\fg_{-\a_{\ttj_{\sfp-\sfr-1}}}$ will stabilize $\fn_w$.  Thus, $1 < \ttj_{\sfp-\sfr} - \ttj_{\sfp-\sfr-1}$.

Consider the tuple $B = (0^{\sfp-\sfr-1},1,2^{\sfr-1},1)$ with $\sfr\ge1$, which occurs when $\sfa = 2\sfr\ge2$.  Note that $\fg_{-\a_{\ttj_{\sfp-\sfr}}}$ will stabilize $\fn_w$ if $1 = \ttj_{\sfp-\sfr} - \ttj_{\sfp-\sfr-1}$.  Therefore, $1< \ttj_{\sfp-\sfr} - \ttj_{\sfp-\sfr-1}$.
\end{remark}
\end{subequations}

{\bf (V)}  We finish with $A_n/P_\tti$, for any $1 \le \tti \le n$.  Any root $\a \in \Delta(\fg_1)$ is of the form $\a = \a_i + \cdots + \a_j$ with $i \le \tti \le j$.  In particular, the $B = (b_1,\ldots,b_p)$ with $|B| = \sfa$ are of the form $B_\ell = (0^\ell,1^\sfa,0^{\sfp-\sfa-\ell})$, with $0 \le \ell \le p -\sfa$.  In particular, $\sfa \le \sfp$.  If $\sfa = \sfp$, then there is a unique $B$ with $|B| = \sfa$ and $\fn_w = \fg_{-1}$.  In particular, $X_w = X$ is not a proper subvariety.  

So $\sfa < \sfp$.  By definition of $\sfa$, there exists $A = B_{\ell(A)}$ such that $|A| = \sfa$ and $\fg_{-1,-A} \subset \fn_w$.  Suppose that $\fg_{-1,-B_{\ell(A)+1}}\not \subset \fn_w$.  Then with some thought we see that the root space $\fg_{-\a_{\ttj_{\ell(A)+1}}} \subset \fg_{0,-}$ stabilizes $\fn_w$, a contradiction.  By induction, it follows that $\fg_{-1,-B_\ell} \subset \fn_w$ for all $\ell \ge \ell(A)$.  

Next suppose that $\fg_{-1,-B_{\ell(A)-1}} \not\subset \fn_w$.  Again, with some thought, we see that the root space $\fg_{-\a_{\ttj_{\ell(A)+\sfa}}} \subset \fg_{0,-}$ stabilizes, a contradiction.  We conclude that $\fg_{-1,-B_\ell}$ for all $\ell\le\ell(A)$.  

Given $\ttJ\not=\emptyset$, suppose that $\tti < \ttj_1$.  Observe that $\fg_{-\a_{\ttj_\sfp}}$ will stabilize $\fn_w$ if $\sfp > \sfa+1$.  Thus $|\ttJ| = \sfa+1$.  Moreover, if $\ttj \in \ttJ - \{\ttj_\sfp\}$, then $\fg_{-\a_\ttj} \subset \fg_{0,-1}$ stabilizes $\fn_w$.  Thus $|\ttJ|=1$ and $\sfa=0$.  Similarly, if $\ttj_\sfp < \tti$, then $|\ttJ| = 1$ and $\sfa=0$.

Suppose that $\ttj_1 < \tti < \ttj_\sfp$.  Define $\sfq\in\bZ$ by $\ttj_\sfq < \tti < \ttj_{\sfq+1}$.  Observe that $\sfq \le \sfa+1$, else $\fg_{-\a_{\ttj_1}}$ will stabilize $\fn_w$.  Similarly, $\sfp-\sfq \le \sfa+1$, else $\fg_{-\a_{\ttj_\sfp}}$ will stabilize $\fn_w$.  Also $\sfa \le \sfq$, else $\fg_{-\a_{\ttj_{\sfq+1}}}$ stabilizes $\fn_w$.  Similarly, $\sfa \le \sfp-\sfq$, else $\fg_{-\a_{\ttj_\sfq}}$ stabilizes $\fn_w$.
\end{proof}

\noindent 
The analyses of Sections \ref{S:BR} and \ref{S:RR} require a detailed description of $(\sfa,\ttJ)$.  To that end, the observations made in the proof of Lemma \ref{L:bigone} are collected in the corollary below.  For convenience we set 
$$
  \ttj_0 \ := \ 0 \quad \hbox{ and } \quad 
  \ttj_{\sfp+1} \ := \ 1 + \tmax\{1,\ldots,\hat\tti,\ldots,n\} \,.
$$
so that $\ttj_0 < \ttj_1 < \cdots < \ttj_\sfp < \ttj_{\sfp+1}$.

\begin{corollary} \label{C:bigone}
 Let $G/P_\tti$ be a classical compact Hermitian symmetric space.  Suppose that $X_w \subset X$ is a proper Schubert variety.  Then $0 \le \sfa(w) \in \bZ$ and $\emptyset \neq \ttJ(w) = \{ \ttj_1 , \ldots , \ttj_\sfp \} \subset I_\fp$ satisfy the criteria in Table \ref{t:bigone}.  Conversely, any $(\sfa,\ttJ)$ meeting the criteria in Table \ref{t:bigone} is realized by a unique Schubert variety $X_w \subset X$; that is, there exists a unique $w \in W^\fp$ such that $\sfa = \sfa(w)$ and $\ttJ = \ttJ(w)$.
\end{corollary}

\begin{table}[h] 
\caption{$(\sfa,\ttJ)$ for the classical $G/P_\tti$}
\begin{tabular}{|c|c|l|l|} \hline 
$G/P_\tti$ & Upper bound on $\sfa$
  & \multicolumn{2}{|l|}{Realizability criteria for $(\sfa,\ttJ)$} \\ \hline\hline
  & & \multicolumn{2}{|l|}{Defining $0 \le \sfq \le \sfp$ by 
  $\ttj_\sfq < \tti < \ttj_{\sfq+1}$, we have:} \\ 
\raisebox{1.5ex}[0pt]{$A_n / P_\tti$} 
  & \raisebox{1.5ex}[0pt]{$\min(\tti-1,n-\tti)$}
  & \multicolumn{2}{|l|}{$(\sfp,\sfq) \in \{ (2\sfa,\sfa)\,,\, 
    (2\sfa+1,\sfa)\,,\, (2\sfa + 1, \sfa + 1)\,,\,  (2\sfa + 2, \sfa + 1)\}$.}
\\ \hline
$B_n / P_1$ & $1$ & \multicolumn{2}{|l|}{$\ttJ = \{ \ttj \}$} \\ \hline
 & & \multicolumn{2}{|l|}{$\sfa=0$, $\ttJ = \{ \ttj\}$ or $\{n-1,n\}$;} \\
\raisebox{1.5ex}[0pt]{$D_n / P_1$} & \raisebox{1.5ex}[0pt]{$1$} 
  & \multicolumn{2}{|l|}{$\sfa = 1$, $\ttJ = \{ \ttj \leq n-2 \}$ or $\{n-1,n\}$.}
  \\ \hline
$C_n / P_n$ & $n-1$ & \multicolumn{2}{|l|}{$\sfp = \sfa$ or $\sfa + 1$.} \\ \hline
 & & $\sfp=\sfa$ and $n-1\not\in\ttJ$, or & \\
 & \raisebox{1.5ex}[0pt]{$n-3$} & $\sfp=\sfa+1$ and $n-1\in\ttJ$ 
 & \raisebox{1.5ex}[0pt]{$\ttj_\sfs - \ttj_{\sfs-1} \geq 2$ 
   for $\sfs = \left\lceil \frac{\sfp+1}{2} \right\rceil$;}\\ \cline{2-4}
\raisebox{1.5ex}[0pt]{$D_n / P_n$} & 
 & $\sfp=\sfa+1$ and $n-1\not\in\ttJ$, or & \\
 & \raisebox{1.5ex}[0pt]{$n-4$}
 & $\sfp=\sfa+2$ and $n-1\in\ttJ$
 & \raisebox{1.5ex}[0pt]{$\ttj_{\sfs+1} - \ttj_{\sfs} \geq 2$ for 
   $\sfs = \left\lceil \frac{\sfp}{2} \right\rceil$.} \\ \hline
\end{tabular} 
\label{t:bigone}
\end{table}

\begin{remark*}
 Schubert varieties in the Grassmannian $\tGr(\tti,n+1) = A_n / P_\tti$ are indexed by partitions $\pi \in \ttP(\tti,n+1)$.  Proposition \ref{P:part-aJ} describes $(\sfa,\ttJ)$ in terms of the partition.  
\end{remark*}

We close this section with some definitions.  Given a Schubert variety $X_w$, from this point on $\fg = \op \fg_{j,k}$ will denote the $(Z_\tti,Z_w)$--bigraded decomposition of $\fg$.  Define 
\begin{equation} \label{E:gw_dfns} \renewcommand{\arraystretch}{1.3}
\begin{array}{rcl}
  \fg_w \ := \  \fn_w \ \op \ \fg_{0,\ge0} \ \op \ \fg_{1,\ge\sfa}\,,
  & \quad & \fn_w \ = \ \fg_{-1,\ge-\sfa} \,;\\
  \fg_w^\perp \ := \ \fn_w^\perp \ \op \ \fg_{0,<0} \ \op \ \fg_{1,<\sfa} \,,
  & \quad & \fn_w^\perp \ = \ \fg_{-1,<-\sfa} \, .
\end{array}\end{equation}
Both $\fg_w$ and $\fg_w^\perp$ are subalgebras of $\fg$, and $\fg_w$ is the largest subalgebra of $\fg$ containing $\fn_w$.  Both $\fg = \fg_w \op \fg_w^\perp$ and $\fg_{-1} = \fn_w \op \fn_w^\perp$ are $\fg_{0,0}$--module decompositions.

\subsection{Smooth Schubert varieties}\label{S:sm}

\begin{proposition} \label{P:a=0}
Let $X$ be a compact Hermitian symmetric space.  The Schubert variety $X_w \subset X$ is smooth if and only if the integer $\sfa(w)$ of Proposition \ref{P:nw2graded} is zero.
\end{proposition}
 
\noindent The proposition follows almost immediately from Proposition \ref{P:nw2graded} and the following.

\begin{proposition}[{\cite[\S2.3]{MR2276624}}]  \label{P:sm_subdiag}
Let $G/P_\tti$ be an irreducible compact Hermitian symmetric space.  The Schubert variety $X_w$ is smooth if and only if $X_w$ is homogeneous.  The smooth Schubert varieties (modulo the action of $G$) are in bijective correspondence with the connected sub-diagrams $\d$ of the Dynkin diagram $\d_\fg$ of $G$ that contain the node $\tti$.  In particular, if $D \subset G$ is the simple Lie group associated to $\d \subset \d_\fg$, then $X_w = D/(D\cap P)$, and $\fn_w = \fd \cap \fg_{-1}$.
\end{proposition}

\begin{proof}[Proof of  Proposition \ref{P:a=0}]
Assume $X_w$ is smooth and let $\d$ be the associated sub-diagram.  Let $\ttJ$ be given by the nodes of $\d_\fg \backslash \d$ that are adjacent to $\d$.  Then the sub-diagram $\d$ clearly corresponds to the $\fn_w$ associated to $\ttJ$ with $\sfa = 0$.

Conversely, given a Schubert variety $X_w$ with associated $\ttJ$ and $\sfa = 0$, let $\d$ be the largest connected sub-diagram of $\d_\fg$ containing the node $\tti$ and with the property that no node of $\ttJ$ is contained in $\d$.  If $\fd \subset \fg$ is the simple subalgebra associated to $\d$, then we clearly have $\fn_w = \fd \cap \fg_{-1}$.
\end{proof}

\begin{example*} 
The following examples illustrate the relationship between the sub-diagram $\d$ of Proposition \ref{P:sm_subdiag} and $\ttJ(w)$.  The $\tti$--th node is marked with a \boldmath$\times$\unboldmath, and nodes of $\ttJ$ are circled.  These examples all have $\sfa = 0$.

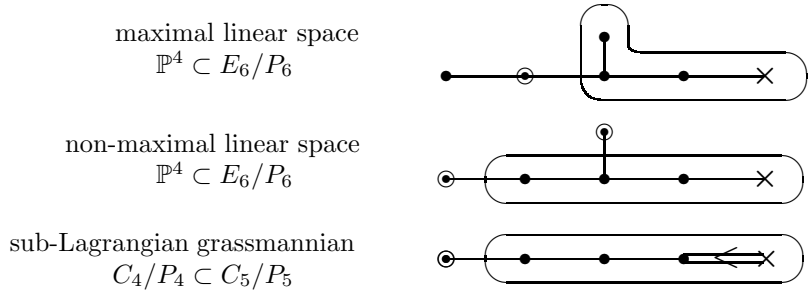
\begin{center} \setlength{\unitlength}{3pt}
\begin{picture}(100,36)
%
\put(13.3,30.5){\small{maximal linear space}}
\put(18.5,26.5){\small{$\bP^4 \subset E_6/P_6$}}
\put(55,26){\line(1,0){40}} 
\put(55,26){\circle*{1.2}}\multiput(75,26)(10,0){2}{\circle*{1.2}}
\put(93.6,26.1){\ptimes}
\put(65,26){\circle*{1}}\put(65,26){\circle{2}}
\put(75,26){\line(0,1){5}} \put(75,31){\circle*{1.2}}
\put(80,26){\oval(41,6)[r]} \put(75,30){\oval(6,10)[t]}
\put(72,31){\line(0,-1){5}}\qbezier(72,26)(72,23)(75,23)
\qbezier(78,30)(78,29)(80,29) \put(75,23){\line(1,0){5}}
\put(7,16.5){\small{non-maximal linear space}}
\put(18.5,12.5){\small{$\bP^4 \subset E_6/P_6$}}
\put(55,13){\line(1,0){40}} 
\put(55,3){\circle*{1}}\put(55,3){\circle{2}}
\multiput(65,13)(10,0){3}{\circle*{1.2}}\put(93.6,13.1){\ptimes}
\put(55,13){\circle*{1}}\put(55,13){\circle{2}}
\put(75,19){\circle*{1}}\put(75,19){\circle{2}}
\put(75,13){\line(0,1){6}} \put(80,13){\oval(40,6)}  
\put(0,4){\small{sub-Lagrangian grassmannian}}
\put(13,0){\small{$C_4/P_4 \subset C_5/P_5$}}
\put(55,3){\line(1,0){30}} \multiput(85,2.5)(0,1){2}{\line(1,0){10}}
\put(88.5,1.9){\Large{$<$}} 
\put(55,3){\circle*{1}}\put(55,3){\circle{2}}
\multiput(65,3)(10,0){3}{\circle*{1.2}}\put(93.8,3.1){\ptimes}
\put(80,3){\oval(40,6)}
\end{picture}\end{center}
\end{example*}

 \subsection{Conjugation and duality in the CHSS} \label{S:CD}
 
Fix $w \in W^\fp$ with associated $\sfa = \sfa(w)$ and $\ttJ = \ttJ(w)$.  Let $w'$ denote the conjugate, defined in Section \ref{S:PD} with respect to a choice of Dynkin diagram automorphism $\varphi : \d_\fg \to \d_\fg$.  It is straight-forward to check that 
\begin{equation} \label{E:a'J'}
  \sfa' \ := \ \sfa(w') \ = \ \sfa \quad\hbox{ and }\quad 
  \ttJ' \ := \ \ttJ(w') \ = \varphi(\ttJ) \, .
\end{equation}

By Corollary \ref{C:pd}, $X_{w^*}$ is Poincar\'e dual to $X_w$.  The rest of this section is devoted to the determination of $\sfa(w^*)$ and $\ttJ(w^*)$.  Given $* : I_\fp \to I_\fp$ from Section \ref{S:PD} which induces \eqref{E:j*}, let $\ttJ^* := *\ttJ$.  As noted in Remark \ref{R:dual}, there is an induced automorphism $\psi$ of the Dynkin sub-diagram $\d_\fp \subset \d_\fg$.  By Lemma \ref{L:PD}(a), identifying $\psi$ will determine $\ttJ(w^*)$.

 \begin{lemma} \label{L:PD} 
Let $G/P_\tti$ be CHSS.  Let $w\in W^\fp$ with $(\sfa,\ttJ) = (\sfa(w), \ttJ(w))$.  Then
 \begin{a_list}
 \item $\ttJ(w^*) = \ttJ^*$;
 \item $w^0_\fp \alpha_{\tti}$ is the highest root $\tilde \a\in\Delta(\fg_1)$ of $\fg$;
 \item $\sfa^* := \sfa(w^*) = \tilde\a(Z_w) - \sfa-1$.
 \end{a_list}
 \end{lemma}
 
\begin{proof}
(a) is deduced from Lemma \ref{L:dual} as follows:
\begin{eqnarray*}
 \ttj \in \ttJ & \stackrel{\eqref{E:Jw}}{\Longleftrightarrow} & 
 \exists\ \beta \in \Delta(w) \mbox{ with } 
 \beta + \alpha_\ttj \in \Delta(\fg_1) \backslash \Delta(w)\\
 & \Longleftrightarrow &  
 \exists\ \c \in \Delta(w^*) \mbox{ with } \c - w^0_\fp\alpha_\ttj \in 
 \Delta(\fg_1) \backslash \Delta(w^*) 
 \quad\hbox{[Set $\c = w^0_\fp(\b+ \alpha_\ttj)$.]}\\
 & \stackrel{\eqref{E:Jw}}{\Longleftrightarrow} &  \ttj^* \in \ttJ(w^*) \, .
\end{eqnarray*}
 
 For (b), we note from \eqref{E:w0p} that $w^0_\fp \alpha_\tti \in \Delta(\fg_1)$, so it suffices to show that $w^0_\fp(\a_\tti) + \a_j \not\in\Delta$ for all simple roots $\a_j$.  The equation \eqref{E:abelian} implies $w^0_\fp \alpha_{\tti} + \alpha_{\tti}\not\in\Delta$.  Let $j \in I_\fp$.  If $w^0_\fp \alpha_{\tti} + \alpha_j \in \Delta(\fg_1)$, then $\alpha_{\tti} - \alpha_{j^*} = \alpha_{\tti} + w^0_\fp \alpha_j \in \Delta(\fg_1)$, a contradiction.

To prove (c), note first that by Proposition \ref{P:nw2graded}, $\Delta(\fg_1)\backslash\Delta(w) = \{ \a
\in \Delta(\fg_1) \ | \ \a(Z_w) > \sfa \}$.  Given $\a \in \Delta(\fg_1)$, write $\a = \a_\tti + \sum_{j\in I_\fp} m^j \a_j$, where $m^j \in \bZ_{\geq 0}$.  Then $\a \in \Delta(\fg_1)\backslash\Delta(w)$ if and only if $\sum_{\ttj\in\ttJ} m^\ttj > \sfa$.  By Lemma \ref{L:dual}(c), every $\a^* \in \Delta(w^*)$ is of the form $w^0_\fp(\a)$ for some $\a \in \Delta(\fg_1)\backslash\Delta(w)$.  From Lemma \ref{L:PD}(b), we deduce
$$
  \a^*(Z_{w^*}) \ = \ w^0_\fp(\a)(Z_{w^*}) \ = \ 
  \tilde\a(Z_{w^*}) \ - \ \textstyle \sum_{\ttj\in\ttJ} m^\ttj \ < \ 
  \tilde\a(Z_{w^*}) \ - \ \sfa \, .
$$
Thus, $\sfa^* = \tilde\a(Z_{w^*}) - \sfa-1$.  It is straight-forward to check that $\tilde\a(Z_{w^*}) = \tilde\a(Z_w)$.
\end{proof}

\begin{remark} \label{R:a*}
If $G/P$ is a classical CHSS, then the values of $\tilde\a(Z_w)$ are as follows.
\begin{a_list}
\item for $A_n/P_\tti$, we have $\tilde\a(Z_w) = | \ttJ(w) |$; 
\item for $B_n/P_1$ and $C_n/P_n$, we have $\tilde\a(Z_w) = 2\,| \ttJ(w) |$; 
\item for $D_n/P_\tti$, with $\tti=1,n-1,n$ and $\ttI = \{1,n-1,n\}\backslash\{\tti\}$, we have 
\begin{i_list_nem}
  \item $\tilde\a(Z_w) = 2\,| \ttJ(w) |$, if $\ttI \cap \ttJ(w) = \emptyset$,
  \item $\tilde\a(Z_w) = 2\,| \ttJ(w) |-2$, if $\ttI \subset \ttJ(w)$, and
  \item $\tilde\a(Z_w) = 2\,| \ttJ(w) |-1$ otherwise.
\end{i_list_nem}
\end{a_list}
\end{remark}
 
\begin{proposition} \label{P:pd}
Let $G/P_\tti$ be an irreducible compact Hermitian symmetric space.  The automorphism $\psi$ of the Dynkin sub-diagram $\d_\fp = \d_\fg \backslash\{\tti\}$ corresponding to $*: I_\fp \to I_\fp$ preserves each connected component $\d_0 \subset \d_\fp$.  In the case $D_n/P_1$, with $n$ odd, $\psi$ acts trivially.  In all other cases $\psi$ acts on $\d_0$ by the unique nontrivial automorphism, when it exists.
\end{proposition}

\begin{remark*}
The sub-diagram $\d_\fp$ is connected for all CHSS, except $A_n/P_\tti$, with $1 < \tti < n$.
\end{remark*}

\begin{proof}
The first half of the proposition follows from Lemma \ref{L:PD} and the following well-known results from representation theory.  Let $\d_0 \subset \d_\fp$ be a connected component.  Since $w^0_\fp$ preserves the root subsystem associated to $\d_0$, it is immediate from the definition \eqref{E:j*} that $\psi$ preserves $\d_0$.   If $\d_0$ is of type $\fa_\ell$, then $w^0_\fp$ acts on $\d_0$ by the unique nontrivial (if $\ell>1$) automorphism.  If $\d_0$ is of type $\fd_\ell$, then $w^0_\fp$ acts on $\d_0$ by the unique nontrivial automorphism if $\ell$ is odd; and trivially if $\ell$ is even.  If $\d_0$ is of type $\fe_6$, then $w^0_\fp$ acts on $\d_0$ by the unique nontrivial automorphism.  In all other cases $\d_0$ admits no nontrivial automorphism.  
\end{proof}

\begin{example} \label{X:A10P5a}
In $A_{10} / P_5$, consider $\sfa = 1$ and $\ttJ = \{ 1,3,7,10 \}$ which corresponds by Corollary \ref{C:bigone} to some $X_w$.  By Remark \ref{R:a*}(a), $\tilde\a(Z_\ttJ) = 4$.  Using \eqref{E:a'J'}, Lemma \ref{L:PD}, and Proposition \ref{P:pd}, we obtain: 
\begin{center} \setlength{\unitlength}{3pt}
\hspace{-0.5in}
\begin{picture}(120,30)
  \put(0,26){\line(1,0){90}}
  \put(0,26){\circle*{1.0}}\put(0,26){\circle{2.0}}
  \put(10,26){\circle*{1.2}}
  \put(20,26){\circle*{1.0}}\put(20,26){\circle{2.0}}
  \put(30,26){\circle*{1.2}}
  \put(38.5,26){$\ptimes$}
  \put(50,26){\circle*{1.2}} 
  \put(60,26){\circle*{1.0}}\put(60,26){\circle{2.0}}
  \put(70,26){\circle*{1.2}}
  \put(80,26){\circle*{1.2}}
  \put(90,26){\circle*{1.0}}\put(90,26){\circle{2.0}}
  \put(99,25.1){\footnotesize $\sfa = 1, \,\,\ttJ = \{1,3,7,10\}$}
  \put(0,19){\line(1,0){90}}
  \put(0,19){\circle*{1.2}}
  \put(10,19){\circle*{1.0}}\put(10,19){\circle{2.0}}
  \put(20,19){\circle*{1.2}}
  \put(30,19){\circle*{1.0}}\put(30,19){\circle{2.0}}
  \put(38.5,19){$\ptimes$}
  \put(50,19){\circle*{1.0}}\put(50,19){\circle{2.0}}
  \put(60,19){\circle*{1.2}} 
  \put(70,19){\circle*{1.2}} 
  \put(80,19){\circle*{1.0}}\put(80,19){\circle{2.0}}
  \put(90,19){\circle*{1.2}}
  \put(99,18.1){\footnotesize $\sfa^* = 2, \,\,\ttJ^* = \{2,4,6,9\}$}
  \put(0,12){\line(1,0){90}}
  \put(0,12){\circle*{1.0}}\put(0,12){\circle{2.0}}
  \put(10,12){\circle*{1.2}}
  \put(20,12){\circle*{1.2}}
  \put(30,12){\circle*{1.0}}\put(30,12){\circle{2.0}}
  \put(40,12){\circle*{1.2}} 
  \put(48.5,12){$\ptimes$}
  \put(60,12){\circle*{1.2}} 
  \put(70,12){\circle*{1.0}}\put(70,12){\circle{2.0}}
  \put(80,12){\circle*{1.2}}
  \put(90,12){\circle*{1.0}}\put(90,12){\circle{2.0}}
  \put(99,11.1){\footnotesize $\sfa' = 1, \,\,\ttJ' = \{1,4,8,10\}$}
  \put(0,5){\line(1,0){90}}
  \put(0,5){\circle*{1.2}}
  \put(10,5){\circle*{1.0}}\put(10,5){\circle{2.0}}
  \put(20,5){\circle*{1.2}}
  \put(30,5){\circle*{1.2}} 
  \put(40,5){\circle*{1.0}}\put(40,5){\circle{2.0}}
  \put(48.5,5){$\ptimes$}
  \put(60,5){\circle*{1.0}}\put(60,5){\circle{2.0}}
  \put(70,5){\circle*{1.2}}
  \put(80,5){\circle*{1.0}}\put(80,5){\circle{2.0}}
  \put(90,5){\circle*{1.2}}
  \put(99,4.1){\footnotesize $(\sfa^*)' = 2, \,\,(\ttJ^*)' = \{2,5,7,9\}$}
\end{picture}
\end{center}
\end{example}

\subsection{Schubert varieties in Grassmannians}\label{S:Sch-Gr}
 
It is well-known that Schubert varieties in $\tGr(\tti, n+1) = A_n / P_\tti$ are indexed by partitions
 \begin{align} \label{E:partition-set}
 \ttP(\tti,n+1) = \{ (a_1,\ldots,a_\tti) \in \bZ^\tti: n+1-\tti \geq a_1 \geq a_2 \geq \cdots \geq a_\tti \geq 0 \}.
 \end{align}
Thus, the Hasse diagram $W^\fp$ and $\ttP(\tti,n+1)$ are in one-to-one correspondence, and we may label a Schubert variety as $X_\pi$ for some $\pi \in \ttP(\tti,n+1)$.  In this section we derive the formulas for $\sfa=\sfa(\pi)$ and $\ttJ = \ttJ(\pi)$.  This will provide a dictionary between the familiar partition description of the Schubert varieties, and our $(\sfa,\ttJ)$-description.

An element $\pi = (a_1,\ldots, a_\tti) \in \ttP(\tti,n+1)$ is a Young diagram with $a_\ell$ boxes in row $\ell$, drawn on an $\tti \times (n+1-\tti)$ rectangular grid.  We define $|\pi| = a_1 + \cdots + a_\tti$. As $\pi$ may contain repetitions, write $\pi = (p_1{}^{q_1},\ldots,p_r{}^{q_r})$, where $\{ p_\ell \}_{\ell=1}^r$ is strictly decreasing, $p_\ell,q_\ell \geq 1$ for all $\ell$, and
 \begin{equation} \label{E:pq-ineq}
 p_1 \leq n+1 - \tti \quad \hbox{and} \quad q_1 + \cdots + q_r \leq \tti \, .
 \end{equation}
 The first $q_1$ rows of $\pi$ each have $p_1$ boxes, et cetera. (The $\tti- \sum_{j=1}^r q_j$ zeros have been suppressed.)
  
We will consider two operations on partitions: conjugation $' : \ttP(\tti,n+1) \rightarrow \ttP(n+1-\tti,n+1)$ and duality $* : \ttP(\tti,n+1) \rightarrow \ttP(\tti,n+1)$.  Given $\pi = (p_1{}^{q_1},\ldots,p_r{}^{q_r}) \in \ttP(\tti,n+1)$, 
\begin{enumerate}
 \item $\pi'$ is the transposed Young diagram, and 
 \item $\pi^*$ is the complement within the $\tti \times (n+1-\tti)$ rectangular grid, rotated 180 degrees.  
\end{enumerate}
It is clear that $(\pi')' = \pi = (\pi^*)^*$, $(\pi^*)' = (\pi')^*$, $|\pi'| = |\pi|$, and $|\pi| + |\pi^*| = \tti(n+1-\tti)$.

If $\pi' = (p_1'{}^{q_1'}, \ldots, p_{r'}'{}^{q_{r'}'})$, then $r' = r$ and
\begin{equation} \label{E:conj}
  p_i' \ = \ q_1 + \cdots + q_{r-i+1} \quad\hbox{and}\quad
  q_i' \ = \ p_{r-i+1} - p_{r-i+2}
\end{equation}
for all $i=1,\ldots,r$ and $p_{r+1} := 0$.  Note that $p_1' = q_1 + \cdots + q_r$.  We decompose the partitions into four types.  

\begin{definition} \label{D:suits}
We decompose $\ttP(\tti,n+1)$ in to four types:
\begin{center}\begin{tabular}{ll}
  $\pi \in \spadesuit$ if $p_1=n+1-\tti$ and $p'_1 = \tti$; $\quad$ & 
  $\pi\in\heartsuit$ if $p_1=n+1-\tti$ and $p'_1 < \tti$; \\
  $\pi\in\diamondsuit$ if $p_1<n+1-\tti$ and $p'_1 = \tti$; &
  $\pi\in\clubsuit$ if $p_1<n+1-\tti$ and $p'_1 < \tti$.
\end{tabular}\end{center}
\end{definition}

\begin{remark*}
Note that $p_1 = n+1 - \tti$ [respectively, $p_1' = \tti$] precisely when the first row [respectively, column] of $\pi$ achieves the maximum possible length.  Hence,
 \begin{center}
 \begin{tabular}{llll}
 Conjugation: & $\spadesuit' = \spadesuit$, & $\clubsuit' = \clubsuit$, & $\heartsuit'=\diamondsuit$.\\
 Duality: & $\spadesuit^* = \clubsuit$, & $\heartsuit^*=\heartsuit$, & $\diamondsuit^*=\diamondsuit$.
 \end{tabular}
 \end{center}
\end{remark*}

The partition $\pi^* = (p_1^*{}^{q_1^*}, \ldots, p_{r^*}^*{}^{q_{r^*}^*} )$ is given by
$\displaystyle
   r^* \ = \ \left\{ \begin{array}{cl}
    r-1, & \hbox{ if $\pi \in \spadesuit$,}\\
    r, & \hbox{ if $\pi \in \heartsuit \cup \diamondsuit$,} \\
    r+1, & \hbox{ if $\pi \in \clubsuit$;}
  \end{array} \right.
$
and
\begin{eqnarray*}
  (p^*_1,q^*_1) & = & \left\{ \begin{array}{ll}
    (n+1-\tti \,,\, \tti-p'_1), & 
    \hbox{ if $\pi \in \heartsuit \cup \clubsuit$,}\\
    (n+1-\tti-p_r \,,\, q_r), & 
    \hbox{ if $\pi \in \spadesuit \cup \diamondsuit$,} 
  \end{array} \right. \\
  (p^*_\ell,q^*_\ell) & = & \left\{ \begin{array}{ll}
    (n+1-\tti-p_{r-\ell+2} \,,\, q_{r-\ell+2}), & 
    \hbox{ if $\pi \in \heartsuit \cup \clubsuit$,}\\
    (n+1-\tti-p_{r-\ell+1} \,,\, q_{r-\ell+1}), & 
    \hbox{ if $\pi \in \spadesuit \cup \diamondsuit$,} 
  \end{array} \right. \quad\mbox{where } 1 < \ell < r \,, \\
  (p^*_r,q^*_r) & = & \left\{ \begin{array}{ll}
    (n+1-\tti-p_2 \,,\, q_2), & 
    \hbox{ if $\pi \in \heartsuit \cup \clubsuit$,}\\
    (n+1-\tti-p_1 \,,\, q_1), & 
    \hbox{ if $\pi \in \diamondsuit$,} 
  \end{array} \right. \\
  (p^*_{r+1},q^*_{r+1}) & = & (n+1-\tti-p_1 \,,\, q_1), 
  \quad\hbox{ if $\pi \in \clubsuit$.}
\end{eqnarray*}

\begin{proposition} \label{P:part-aJ} Given $\pi = (p_1{}^{q_1},\ldots,p_r{}^{q_r}) \in \ttP(\tti,n+1)$, the Schubert variety $X_\pi \subset A_n / P_\tti$ has $\dim(X_\pi) = |\pi^*|$ and 
\begin{eqnarray*}
   \sfa(\pi) & = & r^* - 1 = \mbox{the number of interior corners of } \pi^*,\\
   \ttJ(\pi) & = &\{ q_1\,,\, q_1 + q_2\,,\ldots\,,\, q_1 + \cdots + q_r\,,\, 
   n+1 - p_1\,,\, n+1 - p_2\,,\, \ldots\,,\, n+1 - p_r \} \backslash \{ \tti \}.
\end{eqnarray*}
\end{proposition}

\begin{remark*}
 By \eqref{E:pq-ineq}, the $\ttJ(\pi)$ in Proposition \ref{P:part-aJ} is a strictly increasing sequence.
 \end{remark*}
 
\begin{example*} 
The example $X_\pi = \sigma(W) \subset \tGr(m,n) = A_{n-1} / P_m$ discussed in Section \ref{S:mot} corresponds to the partition $\pi = (k)$; that is, $(p_1,q_1) = (k,1)$.  If $k=n$, then $\pi \in \spadesuit$, $\sfa = 0$ and $\ttJ = \{ 1 \}$; if $1 \leq k < n$ then $\pi \in \clubsuit$, $\sfa = 1$ and $\ttJ = \{ 1, n-k \} \backslash \{ m \}$.
\end{example*}

There are two immediate corollaries to Proposition \ref{P:part-aJ}.  The first is a second proof of Corollary \ref{C:bigone}, with the correspondences of Table \ref{t:suit}.  The second corollary is the observation that the conjugate and duality notations of Sections \ref{S:PD} \& \ref{S:Sch-Gr} are consistent:  if $w \in W^\fp$ corresponds to $\pi \in \ttP(\tti,n+1)$, then $w^*$ corresponds to $\pi^*$ and $w'$ to $\pi'$.

\begin{table}[h] 
\renewcommand{\arraystretch}{1.2}
\caption{Suits versus $(\sfp,\sfq)$.}
\begin{tabular}{|r||c|c|c|c|}
\hline
Type & $\spadesuit$ & $\heartsuit$ & $\diamondsuit$ & $\clubsuit$ \\
\hline
$(\sfp,\sfq)$ & $(2\sfa+2,\sfa+1)$ & $(2\sfa+1,\sfa+1)$
              & $(2\sfa+1,\sfa)$ & $(2\sfa,\sfa)$\\
\hline
\end{tabular} 
\label{t:suit}
\end{table}

\begin{example} \label{X:A10P5b} By Proposition \ref{P:part-aJ}, the Schubert variety of Example \ref{X:A10P5a} with $\sfa=1$, $\ttJ = \{ 1, 3, 7, 10 \}$ corresponds to $\pi = (6,4^2,1^2) \in \ttP(5,11)$, sitting in a $5\times6$ rectangle.
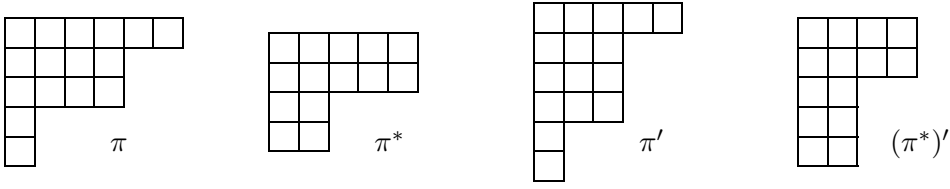
\begin{center}\setlength{\unitlength}{5pt}
\begin{picture}(80,15)
\put(5,7){\small{$\yng(6,4,4,1,1)$}}  \put(13,3){$\pi$}
\put(25,7){\small{$\yng(5,5,2,2)$}}  \put(33,3){$\pi^*$}
\put(45,7){\small{$\yng(5,3,3,3,1,1)$}} \put(53,3){$\pi'$}
\put(65,7){\small{$\yng(4,4,2,2,2)$}} \put(72,3){$(\pi^*)'$}
\end{picture}\end{center}
We have $\pi,\pi' \in \spadesuit$, while $\pi^*, (\pi^*)' \in \clubsuit$.
\end{example}
 
\begin{proof}[Proof of Proposition \ref{P:part-aJ}]
The stabilizer $P = P_\tti \subset A_n = SL_{n+1}(\bC)$ has block form {\footnotesize $\mat{cc}{* & *\\ 0 & *}$}, where the diagonal square blocks are of sizes $\tti$ and $n+1-\tti$ respectively.  Hence, $T_o (A_n / P_\tti) \cong $ {{\footnotesize $\mat{cc}{ 0 & 0 \\ * & 0}$}}$= \fg_{-1}$ and $T_o X_\pi$ is identified with $\fn_\pi$.  Here, $\fn_\pi \subset \Mat_{n+1-\tti,\tti}(\bC)$ is the matrix subspace such that $(\fn_\pi)_{ij} = 0$ if $i > n+1- \tti - a_j$ and we embed this as {\footnotesize $\mat{cc}{ 0 & 0 \\ \fn_\pi & 0}$}$ \subset \fg_{-1}$, cf. \cite[\S2.2]{SchurRigid}.  Pictorially, the nonzero entries of $\fn_\pi$ are precisely the entries of  $\pi^*$ rotated clockwise by 90 degrees.  Consequently, $\dim(X_\pi) = |\pi^*| = (n+1-\tti)\tti - |\pi|$.

Note that $\fg_0$ is block diagonal with blocks of size $\tti$ and $n+1-\tti$.  Recall that $\ttJ$ is determined by the parabolic subalgebra in $\fg_0$ stabilizing $\fn_\pi$.  We have
$$
  \left[ \mat{cc}{K & 0 \\ 0 & L}, \mat{cc}{0 & 0 \\ \fn_\pi & 0} \right] 
  \ = \ \mat{cc}{0 & 0 \\ L\fn_\pi - \fn_\pi K & 0} \,.
$$
So $\fg_{0,0} = \{ (K,L) \ | \ L \z - \z K \in\fn_\pi \ \forall \ \z\in\fn_\pi\}$.  Before identifying $\fg_{0,0}$, we pause to consider an example.
 
\begin{example} \label{X:A10P5d} 
It will be useful to draw a grid on $\Mat_{n+1-\tti,\tti}(\bC)$: at every corner in $\fn_\pi$, draw a horizontal and vertical line.  This yields a partition of the row and column spaces $\Mat_{n+1-\tti,\tti}(\bC)$.  In turn, these respectively induce natural partitions of $\Mat_{n+1-\tti,n+1-\tti}(\bC)$ and $\Mat_{\tti,\tti}(\bC)$.  For $\pi \in \ttP(5,11)$ as in Example \ref{X:A10P5b}, $\pi^* = (5^2,2^2)$ and the induced grid is
{\small $$
 \fn_\pi \ = \ \mat{c|cc|cc}{ 
 0 & * & * & * & *\\ 
 0 & * & * & * & *\\ \hline
 0 & 0 & 0 & * & *\\
 0 & 0 & 0 & * & *\\ 
 0 & 0 & 0 & * & *\\ \hline
 0 & 0 & 0 & 0 & 0} 
 \ \sim \ \mat{cccc}{0 & \star & \star\\ 0 & 0 & \star\\ 0 & 0 & 0 },
$$ }
where the final matrix above is a block simplification of $\fn_\pi$.  The starred entries in $\fn_\pi$ above are {\em arbitrary}, so to satisfy $L\fn_\pi - \fn_\pi K \subset \fn_\pi$, both $K$ and $L$ are of the form {\footnotesize $\mat{ccc}{ \star & \star & \star \\ 0 & \star & \star\\ 0 & 0 & \star}$}, but the diagonal block sizes for each are different.  For $K$, they are sizes $1,2,2$, while for $L$, they are sizes $2,3,1$.  The parabolic $\fg_{0,\geq 0} \subset \fg_0$ determined by $\diag(K,L)$ has associated index set $\ttJ = \{ 1,3,7,10 \}$.  (Note that we do not include 5 since $\tti = 5$ here.)  The reductive subalgebra $\fg_{0,0}$ is block diagonal with blocks of sizes $1,2,2,2,3,1$.  In the last column of $\fn_\pi$, we have from top to bottom the eigenspaces corresponding to the roots
 \begin{align*}
 -\alpha_5\,,\quad 
 -\alpha_5 - \alpha_6\,, \quad 
 -\alpha_5 - \alpha_6 - \alpha_7\,,\quad
 -\alpha_5 - \cdots - \alpha_8\,, \quad
 -\alpha_5 - \cdots - \alpha_9 \,.
 \end{align*}
These roots all have $Z_\tti$-grading $-1$, and their respective $Z_\ttJ$-gradings are $0,0,-1,-1,-1$.  In the block simplification of $\fn_\pi$, the $Z_\ttJ$-gradings are given below, with $\fn_\pi$ also indicated:
 \[
  \mat{c@{\,}c@{}c@{}c}{ \cline{2-3} -2 & \multicolumn{1}{|c}{-1} & \multicolumn{1}{c|}{0}\\ \cline{2-2} -3 & -2 & \multicolumn{1}{|c|}{-1}\\ \cline{3-3} -4 & -3 & -2}
 \]
 Thus, $\fn_\pi = \fg_{-1,0} \oplus \fg_{-1,-1}$, so $\sfa = 1$ and this is the same as the number of interior corners of $\pi^*$.  Finally, $[\fg_{0,0},\fn_\pi] \subset \fn_\pi$, so $\fg_{0,0}$ preserves $\fn_\pi$.  Moreover, $\fg_{0,0}$ acts {\em irreducibly} on every block of $\fn_\pi$.  Hence, $\fg_{-1,0}$ and $\fg_{-1,-1}$ have one and two $\fg_{0,0}$--irreducible components of dimension $4$ and $4,6$ respectively.
 \end{example}

We now return to the proof of the proposition.  The argument proceeds exactly as in the above example.  There are four types of $\pi^*$ to consider (Definition \ref{D:suits}).  In all cases, $\ttJ$ is found to be
 \[
 \ttJ = \left\{ \tti - \sum_{\ell=1}^{r^*} q_\ell^*, \tti - \sum_{\ell=1}^{r^*-1} q_\ell^*, \ldots, \tti - q_1^*, \tti + p_{r^*}^*, \tti + p_{r^*-1}^*, \ldots, \tti + p_1^* \right\} \backslash \{ 0, n+1 \}
 \]
 and $\sfa = r^* - 1$ is the number of interior corners of $\pi^*$.
Using formulas for $(p_j^*,q_j^*)$ that precede the proposition, we recover the expression for $\ttJ$ given in the statement.
\end{proof}

\begin{remark*}
The Schubert varieties in any classical CHSS admit a partition description; and there exist analogous recipes relating the partitions to the $(\sfa,\ttJ)$ data.
\end{remark*}

\section{The Schubert system $\cB_w$} \label{S:Bsys}

\subsection{Definition}\label{S:schubert}

Corollary 8.2 of \cite{MR0142696} asserts that $\tw^k \fg_{-1}$ decomposes as a direct sum of irreducible $\fg_0$--modules
\begin{equation} \label{E:tw^k}
  \tw^k \fg_{-1} \ = \ \bigoplus_{w \in W^\fp(k)} \bI_w \, .
\end{equation}
Above, $W^\fp(k)$ is the set of elements of the Hasse diagram $W^\fp$ of length $k$ and $\bI_w$ is the irreducible $\fg_0$--module of highest weight $w(\rho) - \rho$.  Let $X \subset \bP V_{\w_\tti}$ be an irreducible compact Hermitian symmetric space in its minimal homogeneous embedding.  By \eqref{E:tgtsp0} and \eqref{E:CHSS_gr}, $T_o X\simeq\fg_{-1}$ as $\fg_0$--modules.  So \eqref{E:tw^k} determines a $\fg_0$--module decomposition of $\tw^k T_o X$.

The ${|w|}$-plane $\fn_w \in \tGr({|w|},\fg_{-1})$ is a highest weight line of $\bI_w$.  Define $B_w$ to be the $G_0$--orbit of $\fn_w \in \tGr({|w|},\fg_{-1})$.  Under the $\fp$--module identification $T_oX \simeq \fg/\fp$ the abelian $\fg_+$ acts trivially, and $B_w$ is $P$--stable.  The orbit $B_w \subset \tGr({|w|},\fg_{-1})$ is precisely the set of tangent spaces $T_o (p\cdot X_w)$, $p \in P$.   

Given $z = gP \in X$, the \emph{Schubert system} $\cB_w \subset \tGr({|w|},TX)$ is defined by $\cB_{w,z} := g_* B_w \in \tGr( {|w|} , T_zX )$.  The fact that $B_w$ is stable under $P$ implies that $\cB_{w,z}$ is well-defined.  A ${|w|}$--dimensional complex submanifold $\mfd \subset X$ is an \emph{integral manifold of $\cB_w$} if $T\mfd \subset \cB_w$.   A subvariety $Y \subset X$ is an \emph{integral variety of $\cB_w$} if the smooth locus $Y^0 \subset Y$ is an integral manifold.  The Schubert system is \emph{rigid} if for every integral manifold $M$, there exists $g \in G$ such that $g\cdot M \subset X_w$.  If every integral variety $Y$ of $\cB_w$ is of the form $g\cdot X_w$, then we say $X_w$ is \emph{Schubert rigid}.  

\begin{remark*} If $Y$ is an (irreducible) integral variety of $\cB_w$, then its smooth locus $Y^0$ is a (connected) integral manifold of $\cB_w$ whose closure is $Y$.  If $\cB_w$ is rigid, then there is some $g \in G$ such that $Y^0 \subset g \cdot X_w$ is open.  Since $g \cdot X_w$ is irreducible, then the closure of $Y^0$ is $g \cdot X_w$, and hence $Y = g \cdot X_w$.  Thus, if $\cB_w$ is rigid, then $X_w$ is Schubert rigid.
\end{remark*}


The first step in our analysis of $\cB_w$ is to lift the problem up to a frame bundle $\cG$ over $X \subset \bP V_{\w_\tti}$.  

\subsection{A frame bundle}\label{S:frame}

Set $V = V_{\w_\tti}$ and $\tdim \, V = \sfN+1$.  Given $v \in V \backslash \{0\}$, let $[v] \in \bP V$ denote the corresponding point in projective space.  Given any subset $Y \subset \bP V$, let 
$\widehat Y = \{ v \in V\backslash\{0\} \ | \ [v] \in Y \}$ denote the \emph{cone} over $Y$.  Suppose that $y \in Y$ is a smooth point and $v \in \hat y$.  Then the tangent space $T_v\widehat Y$ (an intrinsic object) may be naturally identified with a linear subspace $\widehat T_y Y \subset V$ (an extrinsic object).  If $X \subset \bP V$ is the minimal homogeneous embedding of $G/P$, then $o \in G/P$ is naturally identified with $[\sfv_0] \in \bP V$, where $\sfv_0 \in V \backslash \{0\}$ is a highest weight vector.  Because $\sfv_0$ is a highest weight vector we have 
\begin{equation} \label{E:g.v0}
  \wh T_o X \ = \ 
  \fg \, . \, \sfv_0 \ = \ \bC \sfv_0 \op \fg_{-1} \, . \, \sfv_0 \, .
\end{equation}
Fix a basis $\sfv = ( \sfv_0 , \sfv_1 , \ldots , \sfv_\sfN )$ for $V$, so that 
$$
  \tspan\{ \sfv_1 , \ldots , \sfv_{|w|} \} \ = \ \fn_w \, . \, \sfv_0 \,.
$$
Define $\cG$ to be the set of frames (bases) 
$$
  \cG \ := \ \{ g \cdot \sfv \ | \ g \in G \} \, .
$$
Note that the bundle $\cG$ is naturally identified with the Lie group $G$.  Elements of $\cG$ are sometimes expressed as $v = (v_0 , \ldots , v_\sfN)$ where $v_i = g \cdot \sfv_i$.  The set of frames $\cG$ is a right $P$--bundle over $X$ under the projection
$$
  \pi : \cG \to X
$$ 
mapping $v = (v_0,\ldots,v_\sfN)$ to $[v_0]$.  (Alternatively, under the identification $\cG \simeq G$, $\pi$ maps $g \in G$ to $[g\cdot\sfv_0] \in \bP V$.)

\subsection{The Maurer-Cartan form}

Let $\vartheta$ denote the $\fg$--valued, left-invariant \emph{Maurer-Cartan 1-form} on $\cG$.  At a point $v = g\cdot\sfv \in \cG$, $\vartheta_v = \vartheta_g$ is defined by $\td v_j = \vartheta^i_j \, v_i = \vartheta \, . \, v_j$.  Regard each $v_i$ as a map $G \to V$ sending $g \mapsto g \cdot \sfv_i$.  Then \eqref{E:g.v0} yields
\begin{equation} \label{E:de0}
  \td v_0 \ \equiv \ \tAd_g(\vartheta_{\fg_{-1}}) \, . \, v_0 
  \quad \hbox{mod } \ \tspan\{v_0\} \, .
\end{equation}
In particular, the semi-basic forms for the projection $\cG \to X$ at $v = g \cdot \sfv$ are spanned by the $\vartheta_{\fg_{-1}}$. 

Given $\fg$--valued 1-forms $\varphi$ and $\chi$, let $[\varphi,\chi]$ denote the $\fg$--valued 2-form given by 
$$
  [\varphi,\chi](u,v) \ := \ 
  [\varphi(u),\chi(v)] \ - \ [\varphi(v),\chi(u)] \, .
$$
Note that $[\varphi,\chi] = [\chi,\varphi]$; we make frequent use of this in computations without mention.  The derivative of the Maurer-Cartan form is given by the \emph{Maurer-Cartan equation} 
\begin{equation} \label{E:mce}
  \td \vartheta = - \half \,[\vartheta,\vartheta] \, .
\end{equation}

Given any $\fg$--valued form $\varphi$ on $\cG$ and a direct sum decomposition $\fg = \fa \op \fa^\perp$, let $\varphi_\fa$ denote the component of $\varphi$ taking value in $\fa$ at $v \in \cG$.  The following lemma is classical.  

\begin{lemma} \label{L:Frob_sys}
Let $\fa \subset \fg$ be a subalgebra and $\fg = \fa \op \fa^\perp$ vector space direct sum decomposition.  Set $\cA = \texp(\fa)\cdot\sfv \subset \cG$.   The system $\vartheta_{\fa^\perp} = 0$ is Frobenius.  The maximal integral manifolds are $g \cdot \mathcal{A} \subset \cG$, where $g \in G$ is fixed.
\end{lemma}

\begin{proof}
To see that $\vartheta_{\fa^\perp} = 0$ is Frobenius note that the Maurer-Cartan equation implies $\td\vartheta_{\fa^\perp} = - \half\,\left[ \vartheta , \vartheta \right]_{\fa^\perp}$.  Using $[\fa,\fa] \subset \fa$ and $\vartheta = \vartheta_{\fa} + \vartheta_{\fa^\perp}$ we have $\td\vartheta_{\fa^\perp} \equiv 0$ modulo $\vartheta_{\fa^\perp}$; the system is Frobenius.  The remainder of the lemma follows from Cartan's Theorem \cite[Theorem 1.6.10]{MR2003610} on Lie algebra valued 1-forms satisfying the Maurer-Cartan equation.
\end{proof}

To simplify notation set $\vartheta_{j,k} :=\vartheta_{\fg_{j,k}}$.

\subsection{Lifting the Schubert system to the frame bundle}\label{S:lift}

Let $\mfd \subset X$ be a (connected) complex submanifold of dimension $|w|$.  (We are interested in the case that $\mfd$ is the set of smooth points $Y^0$ of a variety $Y$.)  By definition $\mfd$ is an integral manifold of $\cB_w$ if and only if $T\mfd \subset \cB_w$.  Equivalently, every point in $\mfd$ admits an open neighborhood $U \subset \mfd$ with local section $\sigma : U \to \cG$ such that $\{ \sigma_0(u) , \ldots , \sigma_{|w|}(u) \}$ spans $\widehat T_u\mfd$ for all $u \in U$.   

Define the \emph{adapted frame bundle $\cF$} over an integral manifold $\mfd$ of $\cB_w$ to be 
\begin{equation}\label{E:cY}
 \cF \ := \ \left\{ v \in \cG \ : \ [v_0] \in \mfd \, , \ 
 \widehat T_{[v_0]} \mfd = \tspan\{ v_0 , \ldots , v_{|w|} \} \right\} \, .
\end{equation}
Note that $\cF$ is a right $P_w$--bundle over $\mfd$, where $P_w \subset P$ is the parabolic subgroup preserving the flag $\bC \sfv_0 \subset \bC \sfv_0 \,\op\, \fn_w \, . \, \sfv_0 \subset V$.

Since $\fp = \fg_0 \op \fg_1$ is the stabilizer of $\bC \sfv_0$, the Lie algebra of $P_w$ is the subalgebra $\fp_w \subset \fp$ stabilizing $\bC \sfv_0 \,\op\, \fn_w \, . \, \sfv_0$.  Given Proposition \ref{P:nw2graded}, it is straight-forward to see that $\bC \sfv_0 \,\op\, \fn_w \, . \, \sfv_0$ is stabilized by $\fg_{0,\ge0}$ in $\fg_0$ and by all of $\fg_1$.  Thus 
\begin{equation} \label{E:pw}
  \fp_w \ = \ \fg_{0,\ge0 } \ \op \ \fg_{1} \, .
\end{equation}

Observe that \eqref{E:de0} and \eqref{E:cY} imply $\vartheta_{\fn_w^\perp} = 0$ when restricted to $\cF$, cf. \eqref{E:gw_dfns}.  In particular, $\vartheta_{\fn_w}$ spans the semi-basic 1-forms on $\cF$.  Since $\vartheta_{\fp_w}$ spans the vertical 1-forms, it follows that $\vartheta_{\fn_w\op\fp_w}$ is a coframing of $\cF$.

Let $\{ \sfw_1 , \ldots , \sfw_s \}$ be a basis of $\fn_w \op \fp_w$.  Then $\vartheta_{\fn_w\op\fp_w} = \theta^a \, \sfw_a$ uniquely defines 1-forms $\theta^a$, $a = 1 , \ldots , s$.  Set $\tw \vartheta_{\fn_w\op\fp_w} := \theta^1 \wedge \cdots \wedge \theta^s$.  While the form $\tw{\vartheta_{\fn_w\op\fp_w}}$ depends on our choice of basis, its vanishing (or nonvanishing) on any $s$--dimensional tangent subspace does not.  The statement that $\vartheta_{\fn_w \op \fp_w}$ is a coframing on $\cF$ is equivalent to the statement that $\tw \vartheta_{\fn_w\op\fp_w}$ is nowhere vanishing on $\cF$.

\begin{definition*}  
The \emph{Schubert system on $\cG$} is the linear Pfaffian system
  \begin{equation} \label{E:Bw_eds}
  \vartheta_{\fn_w^\perp} \ = \ 0 \quad\hbox{ with }\quad
  \tw \vartheta_{\fn_w\op\fp_w} \ \not= \ 0 \, .
  \end{equation}
An \emph{integral manifold} of \eqref{E:Bw_eds} is a submanifold $\cM \subset \cG$ such that $\tdim\,\cM = \tdim\,\fn_w\op\fp_w$, $\vartheta_{\fn_w^\perp}$ vanishes when restricted to $\cM$, and $\tw \vartheta_{\fn_w\op\fp_w}$ is nowhere zero on $\cM$.
\end{definition*}

\begin{lemma} \label{L:Bw_eds}
The integral manifolds of \eqref{E:Bw_eds} are precisely the adapted bundles $\cF$ over integral manifolds $\mfd$ of the Schubert system $\cB_w$.
\end{lemma}

\begin{proof}
First assume that $\cF$ is adapted frame bundle over an integral manifold $\mfd$ of $\cB_w$.  That $\cF$ is an integral manifold of \eqref{E:Bw_eds} follows directly from the discussion following \eqref{E:pw}.

Now assume that $\cM$ is an integral manifold of \eqref{E:Bw_eds}.  Set $\mfd = \pi(\cM)$.  We will show that \emph{(1)} $\mfd$ is a submanifold of $X$, \emph{(2)} $\mfd$ is an integral manifold of $\cB_w$, and \emph{(3)} $\cM$ is the adapted bundle \eqref{E:cY} over $\mfd$.  The independence condition and the vanishing of $\vartheta_{\fn_w^\perp}$ imply that $\pi : \cM \to X$ is of constant rank ${|w|}$: so every $e \in \cM$ admits a neighborhood $U \subset \cM$ such that $\pi(U)$ is a submanifold of $X$.  To see that $\mfd$, the union of the various $\pi(U)$, is itself a submanifold of $X$ we need the following.

\begin{claim} \label{cl:Bw_eds}
The form $\vartheta_{0,-}$ vanishes when restricted to fibres of $\cM$.
\end{claim}

\noindent The claim is proven below.  Notice that the fibre-wise vanishing of $\vartheta_{0,-}$ implies that the linear span of $\{v_0 , v_1 , \ldots , v_{|w|} \}$ is constant on the fibre $\cM_{[v_0]}$ over $[v_0] \in \mfd$.  Thus $\widehat T_{[v_0]} \mfd$ is well-defined and $\mfd$ is a submanifold of $X$.  

That $\mfd$ is an integral manifold of $\cB_w$ is a consequence of $\vartheta_{\fn_w^\perp} = 0$.  Finally, it follows from \eqref{E:Bw_eds} and the definition \eqref{E:cY} that $\cM$ is the adapted frame bundle $\cF$ over $\mfd$.
\end{proof}

\begin{proof}[Proof of Claim \ref{cl:Bw_eds}]
The fibre-wise vanishing of $\vartheta_{0,-}$ is equivalent to $\vartheta_{0,-} \equiv 0$ modulo $\vartheta_{\fn_w}$ on $\cF$.  Differentiating $\vartheta_{\fn_w^\perp} = 0$ with the Maurer-Cartan equation \eqref{E:mce} yields
\begin{equation} \label{E:0neg}
  0 \ = \ -\td \vartheta_{\fn_w^\perp} 
  \ = \ [ \vartheta_{0,-} \, , \, \vartheta_{\fn_w} ]_{\fn_w^\perp} \, .
\end{equation}
Suppose that $\xi \in \fg_{0,-}$ and $[\xi , \fn_w]_{\fn_w^\perp} = \{0\}$.  Then $\xi$ preserves $\fn_w$ and \eqref{E:nwstab} implies $\xi=0$.  We may apply Cartan's Lemma \cite{MR1083148, MR2003610} to conclude that $\vartheta_{0,-} \equiv 0$ modulo $\vartheta_{\fn_w}$.
\end{proof}

\subsection{Distinguishing Schubert varieties}\label{S:Xw_eds}

Now consider the case that $Y = X_w$ is a Schubert variety, and review the definitions \eqref{E:nw} and \eqref{E:gw_dfns}.  Let $N_w\subset G_w$ be the connected Lie subgroups of $G$ with Lie algebras $\fn_w \subset \fg_w$.  The orbit $N_w \cdot o \subset X$ is a dense subset of the smooth points $X_w^0$ of $X_w$, and $\cG_w := G_w \cdot \sfv \subset \cG$ is a sub-bundle of the adapted frames \eqref{E:cY} over $N_w \cdot o$.

When restricted to $\cG_w$, the Maurer-Cartan form takes values in $\fg_w$.  In particular, $\cG_w \subset \cG$ is a maximal integral submanifold of the system $\vartheta_{\fg_w^\perp} = 0$.  Conversely, Lemma \ref{L:Frob_sys} implies that every integral manifold of 
\begin{equation} \label{E:Xw_eds}
  0 \ = \vartheta_{\fg_w^\perp}
  \quad \hbox{with} \quad
  \tw \vartheta_{\fg_w} \not= 0
\end{equation}
is an open submanifold of $\cG_w := G_w \cdot \sfv$ (modulo the action of $G$ on $\cG$).  
%
%
\noindent This establishes the following.  

\begin{proposition} \label{P:Bw_rigid}
The Schubert system $\cB_w$ is rigid if and only if every integral manifold $\cF$ of \eqref{E:Bw_eds} admits a sub-bundle on which $\vartheta_{\fg_w^\perp}$ vanishes.
\end{proposition}

\subsection{Torsion} \label{S:torsion}

In proving Claim \ref{cl:Bw_eds}, we observed that $\vartheta_{0,-} \equiv 0$ modulo $\vartheta_{\fn_w}$ on any integral manifold $\cF$ of \eqref{E:Bw_eds}.  Equivalently, there exists a function $\lambda : \cF \to \fg_{0,-} \ot \fn_w^*$ such that 
\begin{equation} \label{E:th0neg}
  \vartheta_{0,-} \ = \ \lambda( \vartheta_{\fn_w})
\end{equation}
on $\cF$.  The $\lambda$ are not arbitrary -- they are constrained by torsion.  To identify the torsion constraints substitute \eqref{E:th0neg} into \eqref{E:0neg}.  We find 
\begin{equation} \label{E:tor}
  0 \ = \ \Big(
  \big[ u \,,\, \lambda(v) \big] \ - \ \big[ v \,,\, \lambda(u) \big]
  \Big)_{\fn_w^\perp} \,, \qquad \forall \ u,v \in \fn_w \, .
\end{equation}
Equivalently, $\lambda$ must lie in the kernel of the map
\begin{equation}\label{E:uld1}
  \d^1 : \fg_{0,-} \ot \fn_w^* \to \fn_w^\perp \ot \tw^2 \fn_w^*
\end{equation}
defined by 
$$
  (\d^1\lambda)\, u\wedge v \ := \ \Big(
  \big[ u \,,\, \lambda(v) \big] \ - \ \big[ v \,,\, \lambda(u) \big]
  \Big)_{\fn_w^\perp} \, , \quad u , v \in \fn_w \, .
$$
These are the torsion constraints on $\lambda$.

\subsection{Fibre motions in \boldmath$\cF$\unboldmath} \label{S:fibre}

Let $\cF$ be a maximal integral manifold of the Schubert system \eqref{E:Bw_eds}.  In this section we investigate the variation of $\lambda$ along the fibres of $\cF$.  It is often possible to normalize (components of) $\lambda$ to zero via fibre motions within $\cF$.  Equivalently, $\cF$ will admit a sub-bundle on which $\lambda$ (equivalently, $\vartheta_{0,-}$) vanishes.

These normalizations are obtained as follows.  Recall that $\cF$ is a right $P_w$--bundle.  Any smooth function $p : \cF \to P_w$ naturally induces a bundle map $\overline{p} : \cF \to \cF$ mapping $e \in \cF$ to $e \cdot p(e)$.  There are two types of fibre motions to consider:  those with values in $G_{0,\ge0} = \texp(\fg_{0,\ge0}) \subset P_w$, and those with values in $G_1 = \texp(\fg_1) \subset P_w$.  The $G_{0,\ge0}$--valued fibre motions are essentially changes of coordinates and so can not be used to normalize $\lambda$ to zero; see the proof of Lemma \ref{L:norm_lambda}.  So we assume that $p$ takes value in $G_1$.

Let $\varrho$ be the left-invariant, $\fg_1$--valued Maurer-Cartan form on $G_1$.  Then
$$
  \overline p^* \, \vartheta_{\overline p(e)} 
  \ = \ p^* \varrho_{\sfp(e)} \ + \ \tAd_{\sfp(e)^{-1}} \vartheta_e \, .
$$

Pick $t : \cF \to \fg_1$ with $\mathrm{exp}(-t) = p^{-1}$.  Dropping the base point $e$ from our notation, and recalling that $\tAd \circ \texp = \texp \circ \tad$, we have
\begin{eqnarray*}
  (\overline p^* \, \vartheta)_{\fn_w} & = & 
  \left( \tAd_{\sfp^{-1}} \vartheta \right)_{\fn_w} \ = \ 
  \vartheta_{\fn_w} \\ 
  (\overline p^* \, \vartheta)_{0-} & = & 
  \left( \tAd_{\sfp^{-1}} \vartheta \right)_{0,-}  \ = \ 
  \vartheta_{0,-} \ - \ 
  \left( \tad_t \, . \, \vartheta_{\fn_w} \right)_{0,-}\, .
\end{eqnarray*}
(The second line above, along with $(\overline p^* \, \vartheta)_{\fn_w^\perp}  =  0$, confirms that $\overline p^*$ preserves the EDS.)  Since $t \in \fg_1$ and $\fn_w \subset \fg_{-1}$ we may regard $\tad_t$ as an element of $\fg_0 \ot \fn_w^*$.  Let $\tad_{t,-}$ be the projection of $\tad_t$ to $\fg_{0,-} \ot \fn_w^*$.  Then the second equation above reads $(\overline p^* \, \vartheta)_{0,-} = \vartheta_{0,-} - \tad_{t,-} (\vartheta_{\fn_w})$.  On the other hand, $(\overline p^* \, \vartheta)_{0,-} = \overline p^* \, \lambda(\vartheta_{\fn_w}) = ( \lambda \circ \overline p ) \, \vartheta_{\fn_w}$.  Thus, 
\begin{equation} \label{E:p*lambda} 
  \overline p^* \lambda \ = \ \lambda \circ \overline p 
  \ = \ \lambda \ - \ \tad_{t,-} \, .
\end{equation}
In particular, $\lambda - \overline p^* \lambda$ lies in the image of the $\fg_{0,0}$--module map 
\begin{equation} \label{E:uld0}
  \d^0 \,:\, \fg_1 \ \to \ \fg_{0,-} \ot \fn_w^*
\end{equation} 
defined by 
\begin{equation} \label{E:d0} 
  (\d^0 t) u \ := \ [ t,u ]_{0,-} \, ,
\end{equation}
where $t \in \fg_1$ and $u \in \fn_w$.  

\begin{lemma} \label{L:kerd0}
The kernel of $\d^0$ is $\fg_{1,\ge \sfa}\ \subset\ \fg_w$.  Thus, $\d^0 : \fg_{1,<\sfa} \to \fg_{0,-} \ot \fn_w^*$ is injective.
\end{lemma}

\begin{proof}
It is clear from Proposition \ref{P:nw2graded} that $\fg_{1,\ge\sfa} \subset \tker\,\d^0$.  So suppose that $t \in \fg_{1,<\sfa}$ lies in the kernel of $\d^0$.  That is, $[t,u]_{0,-} = 0$ for all $u \in \fn_w$.  It suffices to show that $t = 0$.  Express $t$ as a linear combination $t=\sum_{\a\in\Delta(\fg_{1,<\sfa})} t^\a \, E_\a$ of root vectors $E_\a$ in $\fg_{1,<\sfa}$ with coefficients $t^\a \in \bC$.

Observe that $\d^0(t)(\fg_{-1,-1}) = [ t_0 \,,\, \fg_{-1,-1}] \in \fg_{0,-1}$, where $t_a$ is the component of $t$ taking values in $\fg_{1,a}$.  So $t \in \ker\,\d^0$ forces $[E_\a , \fg_{-1,-1}] = \{0\} \subset \fg_{0,-1}$ for every $\a \in \Delta(\fg_{1,0})$ with $t^\a \not=0$.  This in turn implies that $E_\a \not\in [\fg_{0,-1} \,,\, \fg_{1,1} ]$.  This is a contradiction as every root space $\fg_{\a}$ is obtained from the highest root space $\fg_{\tilde\a} \subset \fg_{1,\sfm}$ by successive brackets with $\fg_{-\a_j} \subset \fg_{0,-1}$, with $\a_j$ simple.  It follows that $t_0 = 0$.

The claim that $t = 0$ now follows by induction: assume that $t_a = 0$ for all $a < a_0$.  Then, as above, $\{0\} = \d^0(t)(\fg_{-1,-1-a_0}) = [t_{a_0} \,,\, \fg_{-1,-1-a_0}] \subset \fg_{0,-1}$ forces $t_{a_0} = 0$.
\end{proof}

Next we claim that $\tim \, \d^0 \subset \tker \, \d^1$.  To see this compute
\begin{eqnarray*}
  \d^1 \circ \d^0(t) \, (u \wedge v) & = & 
  \left[ u \, , \, \d^0(t) \, v \right]_{\fn_w^\perp} \ - \ 
  \left[ v \, , \, \d^0(t) \, u \right]_{\fn_w^\perp} \\
  & = & \left[ u \, , \, [t,v]_{0,-} \right]_{\fn_w^\perp} \ - \ 
  \left[ v \, , \, [t,u]_{0,-} \right]_{\fn_w^\perp} \\
  & = & \left[ u \, , \, [t,v] \right]_{\fn_w^\perp} \ - \ 
  \left[ v \, , \, [t,u] \right]_{\fn_w^\perp} 
  \ \stackrel{(\star)}{=} \ 
  [ t \, , \, [ u \, , \, v ] ]_{\fn_w^\perp} \ \stackrel{(\dagger)}{=} \ 0 \, .
\end{eqnarray*}
The equality $(\star)$ is the Jacobi identity; and the equality $(\dagger)$ follows from \eqref{E:abelian}.  

\begin{lemma} \label{L:norm_lambda}
Let $\tker\,\d^1 = \tim\,\d^0 \op (\tim\,\d^0)^\perp$ be a $\fg_{0,0}$--module decomposition.  Then any integral manifold $\cF$ of \eqref{E:Bw_eds} admits a normalization, via the map $\overline p$, to a sub-bundle $\cF^0 := \overline p(\cF)$ on which $\lambda$ takes values in $(\tim\,\d^0)^\perp$.  Additionally, the $\vartheta_{\fg_w}$ compose a coframing of $\cF^0$ and there exists $\mu : \cF^0 \to \fg_{1,<\sfa} \ot \fn_w^*$ such that $\vartheta_{1,<\sfa} = \mu(\vartheta_{\fn_w})$.
\end{lemma}

\begin{proof}
The first statement follows from the observations above; it remains to establish the second half of the lemma.  Let $\cF^0 \subset \cF$ denote the sub-bundle on which the normalization holds.  From Lemma \ref{L:kerd0} it follows that $\vartheta_{1,<\sfa}$ vanishes when pulled back to the fibres of $\cF^0$.  Equivalently, $\vartheta_{1,<\sfa}$ is semi-basic:  there exists $\mu : \cF^0 \to \fg_{1,<\sfa} \ot \fn_w^*$ such that $\vartheta_{1,<\sfa} = \mu(\vartheta_{\fn_w})$.

To see that the $\vartheta_{\fg_w}$ are linearly independent on $\cF^0$ and thus a coframing (cf. Lemma \ref{L:Bw_eds}), it suffices to observe that the fibre motions associated to $\fg_{0,\ge0} \op \fg_{1,\ge\sfa} \subset \fg_w$ preserve $\tim\,\d^0$.  This is immediate for fibre motions in the directions of $\fg_{1,\ge\sfa}$ by Lemma \ref{L:kerd0}.  In the case that $t : \cF \to \fg_{0,\ge0}$ a computation similar to that above for \eqref{E:p*lambda} shows that 
$$
  (\bar p^*\lambda)\,(u - [t,u]) \ = \ 
  (\lambda\circ\bar p) \,( u - [t,u]) \ = \ \lambda(u) - [t,\lambda(u)]_{0,-}
$$
for all $u \in \fn_w$.  Thus, $\tim\,\d^0$ is preserved under fibre motions in the directions of $\fg_{0,\ge0}$.
\end{proof}

\subsection{When \boldmath $\lambda$ \unboldmath can be normalized to zero} \label{S:lambda=0}

Assume $\cF$ admits a sub-bundle $\cF^0$ on which $\lambda$ vanishes.  Equivalently, $\vartheta_{0,-} = 0$ on $\cF^0$.  Thus, 
\begin{equation} \label{E:preT}
  0 \ = \ -\td \vartheta_{0,-} \ = \ \half \, [\vartheta , \vartheta]_{0,-}
  \ = \ [ \vartheta_{\fn_w} , \vartheta_{1,<\sfa} ]_{0,-} \, .
\end{equation}
By Lemma \ref{L:norm_lambda}, $\vartheta_{1,<\sfa} = \mu(\vartheta_{\fn_w^\perp})$.  Substituting this into \eqref{E:preT} yields torsion, forcing $\mu$ to lie in the kernel of the map
\begin{equation} \label{E:ule}
  \e^1 : \fg_{1,<\sfa} \ot \fn_w^* \to \fg_{0,-} \ot \tw^2 \fn_w^* \, ,
\end{equation}
defined by 
$$
  \e^1(\mu) u \wedge v \ := \ [ u , \mu(v) ]_{0,-} \ - \ [v,\mu(u)]_{0,-} \, .
$$
Definition \eqref{E:gw_dfns}, Proposition \ref{P:Bw_rigid}, Lemma \ref{L:norm_lambda} and the discussion above yield the following key observation.



\subsection{Lie algebra cohomology} \label{S:LAC}

Let $\fa$ be a Lie algebra and $\Gamma$ a $\fa$--module.  Then the Lie algebra cohomology differential $\partial = \partial^k : \Gamma \ot \tw^k\fa^* \to \Gamma \ot \tw^{k+1}\fa^*$ is defined as follows.  Given $\phi \in \Gamma \ot \tw^k\fa^*$ and $u_0,\ldots, u_k \in \fa$,
\begin{equation} \label{E:d} \renewcommand{\arraystretch}{1.4}
\begin{array}{rcl}
  (\partial \phi)(u_0,\ldots,u_k) & := &
  \sum_i \,(-1)^i \,u_i \cdot \phi(u_0 , \ldots , \wh u_i , \ldots , u_k )  \\
  &  & + \ \sum_{i<j} \,(-1)^{i+j}\,\phi([u_i,u_j],u_0,\ldots,\wh u_i,\ldots,\wh u_j,\ldots,u_k ) \,.
\end{array}\end{equation} 
The differential satisfies $\partial\circ\partial = 0$ and so defines cohomology groups
$$
  H^k(\fa,\Gamma) \ := \ 
  \frac{\tker\{\partial: \Gamma \ot \tw^k \fa^* \ 
        \to \ \Gamma \ot \tw^{k+1}\fa^*\}}
       {\tim\{\partial: \Gamma \ot \tw^{k-1} \fa^* \ 
        \to \ \Gamma \ot \tw^k\fa^* \}} \, .
$$

From \eqref{E:abelian} and $\fn_w \subset \fg_{-1}$ we see that $\fn_w$ is an abelian Lie algebra.  Given $A \in \fg$, let $A_{\fg_w^\perp}$ denote the image of $A$ under the natural projection $\fg = \fg_w \op \fg_w^\perp \to \fg_w^\perp$.  Define an action of $\fn_w$ on $\fg_w^\perp$ by 
$$
  u \cdot z \ := \ [u,z]_{\fg_w^\perp} \,, \quad u \in \fn_w \,,\ 
  z \in \fg_w^\perp \, .
$$
Making use of \eqref{E:nw}, \eqref{E:gw_dfns} and the Jacobi identity it is straight-forward to confirm that this action respects the Lie algebra structure of $\fn_w$; that is, $\fg_w^\perp$ is a $\fn_w$--module.  

In the case that $\fa = \fn_w$ and $\Gamma = \fg_w^\perp$, the differential $\partial$ is a $\fg_{0,0}$--module map.  Thus $H^k(\fn_w,\fg^\perp_w)$ is a $\fg_{0,0}$--module, and therefore admits a $Z_\tti$-graded decomposition.  For $k=1$,
\begin{equation} \label{E:Hk_deg}
  H^1(\fn_w,\fg_w^\perp) \ = \ H^1_{0}(\fn_w,\fg^\perp_w) \ \op \
  H^1_{1}(\fn_w,\fg^\perp_w) \ \op \ H^1_{2}(\fn_w,\fg^\perp_w)  \, .
\end{equation}
Here $H^1_m(\fn_w,\fg^\perp_w)$ is the $Z_\tti$--eigenspace associated to the eigenvalue $m$.

Keeping in mind that $\fn_w$ is abelian, we see that 
\begin{equation} \label{E:H1}
  H^1_{1}(\fn_w,\fg^\perp_w) \ = \ \frac{\tker \,\d^1}{\tim \,\d^0}
  \quad\hbox{ and } \quad 
  H^1_{2}(\fn_w,\fg^\perp_w) \ = \ \tker \,\e^1 \, .
\end{equation}


\section{Cohomological computations} \label{S:cohcomp}

\subsection{The adjoint operators} \label{S:adj}

In this section we construct an adjoint $\partial^*$ to the Lie algebra differential $\partial$ of Section \ref{S:LAC} for $\fa = \fn_w$ and $\Gamma = \fg^\perp_w$.

Define a linear map $\sfd^* : \fg \ot \tw^{k+1}\fg \to \fg \ot \tw^k\fg$ on  decomposable elements $v \ot (z_0 \wedge \cdots \wedge z_k)$ by  
\begin{equation*} 
\renewcommand{\arraystretch}{1.4}
\begin{array}{rcl}
  \sfd^* (v \ot (z_0 \wedge \cdots \wedge z_k)) & := &
  \sum_i\,(-1)^i \,[z_i , v] \ot (z_0 \wedge \cdots \wh z_i \cdots \wedge z_k ) \\
  &  &+ \sum_{i<j}\,(-1)^{i+j+1}\,v \ot ([z_i,z_j]\wedge z_0\wedge\cdots \wh z_i \cdots
  \wh z_j \cdots\wedge z_k)\,. 
\end{array}\end{equation*}

The Killing form $(\cdot,\cdot)$ on $\fg$ provides a canonical identification 
\begin{equation} \label{E:nw+}
  \fn_w^* \ \simeq \
  \fn_w^+ \ := \ \fg_{1,0} \,\op\,\cdots\,\op\,\fg_{1,\sfa}
  \quad \hbox{ by } \quad 
  u \in \fn_w^+ \ \mapsto \ (u,\cdot) \in \fn_w^*
\end{equation}
as $\fg_{0,0}$--modules.  In particular, 
\begin{equation} \label{E:*+}
  \fg^\perp_w \ot \tw^k\fn_w^* \ \simeq \ 
  \fg^\perp_w \ot \tw^k\fn_w^+ \ \subset \ \fg \ot \tw^k\fg \, .
\end{equation}
Using this identification, we define $\partial^* :\fg_w^\perp \ot \tw^{k+1} \fn_w^* \to \fg^\perp_w \ot \tw^k\fn_w^*$ to be the restriction
\begin{equation} \label{E:d*}
  \partial^* \ = \ {\sfd^*}_{| \fg^\perp_w \ot \sw^{k+1}\fn_w^+} \, .
\end{equation}
A priori, the image of $\partial^*$ may not lie in $\fg^\perp_w \ot \tw^k\fn_w^+$.  However, \eqref{E:gw_dfns} and \eqref{E:nw+} imply $[\fg_w^\perp,\fn_w^+] \subset \fg_w^\perp$.  Finally, note that $\partial^*$ is a $\fg_{0,0}$--module map.


\begin{proposition} \label{P:adjoint}
There exists a natural positive definite Hermitian inner product $\langle \cdot , \cdot \rangle$ on $\fg^\perp_w \ot \tw^k \fn_w^*$ with respect to which the operators $\partial$ and $\partial^*$ are adjoint.
\end{proposition}

\begin{proof}
The inner product is defined as follows.  Let $(\cdot,\cdot) : \fg \times \fg \to \bC$ denote the Killing form on $\fg$.  Let $\fk \subset \fg$ be a compact form of $\fg$.  Set $\mathrm{i} = \sqrt{-1}$.  Define a conjugate--linear Cartan involution $\theta : \fg \to \fg$ by $\theta_{|\fk} = \tId_{|\fk}$ and $\theta_{|\mathrm{i}\fk} = -\tId_{|\mathrm{i}\fk}$.  It is straight-forward to check that 
$\theta( [u,v] ) = [ \theta u \,,\, \theta v ]$, for all $u,v\in\fg$.  It is well-known that $\fk$ may be selected so that $\theta(Z_j) = -Z_j$ and $\theta(\fg_\b) = \fg_{-\b}$ for all roots $\beta$.  (See \cite[Section 2.3]{MR2532439}.)  Then 
\begin{equation} \label{E:herm_prod}
  \langle u,v \rangle \ := \ -( u , \theta v )
\end{equation}
defines a (positive definite) Hermitian inner product on $\fg$.  Moreover, 
\begin{equation} \label{E:ad_adj}
  \langle [ z , u] \,,\, v \rangle \ = \ 
  -\big( [z , u] \,,\, \theta v \big) \ = \ \big( u \,,\, [z,\theta v] \big) \ = \ 
  \big( u \,,\, \theta([\theta z , v ]) \big) \ = \ 
  -\langle u \, , \, [\theta z , v ] \rangle \, .
\end{equation}
Thus $\tad(z)^* = -\tad(\theta z)$ is the adjoint of $\tad(z) \in \fgl(\fg)$ with respect to $\langle \cdot , \cdot \rangle$.  (Alternatively, $\fg \subset \fgl(\fg)$ is skew-Hermitian.)  We also observe that $\langle \fg_\a , \fg_\b\rangle = 0$, for all roots $\a\not=\beta$, and $\langle \fh , \fg_\b\rangle = 0$ for all $\beta$.  Hence,
\begin{equation} \label{E:Gamma_perp}
  \langle {\fg_w^\perp} \,,\, \fg_w \rangle \ = \ 0 \, .
\end{equation}

Abusing notation, we also let $\langle \cdot , \cdot \rangle$ denote the induced Hermitian inner product on $\fg \ot \tw^k\fg$, and its restriction to $\fg^\perp_w \ot \tw^k\fn_w^*$ under the identification \eqref{E:*+}.  Proposition \ref{P:adjoint} will now follow from the lemma below.
\end{proof}

\begin{lemma} \label{L:adjoint}
$\partial$ and $\partial^*$ are adjoint with respect to $\langle \cdot ,\cdot \rangle$.
\end{lemma}

\begin{proof}
Let $\Phi \in {\fg_w^\perp} \ot \tw^{k+1}\fn_w^*$.  Choose $\gamma \in {\fg_w^\perp}$ and $u_0 ,\ldots , u_k \in \fn_w$.  Set $z_j  = -\theta(u_j)$.  By construction we have  
\begin{equation} \label{E:herm2}
  \big\langle \Phi \,,\, \gamma \ot ( z_0 \wedge \cdots \wedge z_k) 
  \big\rangle_{{\fg_w^\perp} \ot \bigwedge^{k+1}\fn_w^*} \ = \ 
   \big\langle \Phi(u_0,\ldots,u_k) \,,\, \gamma \big\rangle_{\fg_w^\perp} \,.
\end{equation}
Applying \eqref{E:herm2} with $\Phi = \partial\phi$ yields
\begin{eqnarray*}
  \langle \partial \phi \,,\, \gamma \ot (z_0 \wedge \cdots \wedge z_k) \rangle
  & = &  \langle \partial\phi(u_0,\ldots,u_k) \,,\, \gamma \rangle\\
  &  \stackrel{\eqref{E:d}}{=} &  \Big\langle \sum_{i=0}^k (-1)^i
  \left[ u_i \,,\, 
        \phi(u_0 , \ldots , \wh u_i , \ldots , u_k )\right]_{\fg_w^\perp} \,,\, 
  \gamma \Big\rangle\\
  & \stackrel{\eqref{E:Gamma_perp}}{=} &  \Big\langle \sum_{i=0}^k 
  (-1)^i \, \left[ u_i \,,\, 
                  \phi(u_0 , \ldots , \wh u_i , \ldots , u_k) \right] \,,\, 
  \gamma \Big\rangle \,.
\end{eqnarray*}
On the other hand
\begin{eqnarray*}
  \big\langle  \phi \,,\, \partial^* \big(\gamma \ot (z_0 \wedge \cdots \wedge z_k) 
  \big) \big\rangle & = & 
  \Big\langle \phi \,,\,  \sum_{i=0}^k (-1)^i\, [z_i , \gamma] \ot (z_0 \wedge \ldots \wedge \wh z_i \wedge \ldots \wedge z_k ) \Big\rangle \\
  & \stackrel{\eqref{E:herm2}}{=} & \sum_{i=0}^k (-1)^{i}\, \Big\langle
  \phi(u_0,\ldots , \wh u_i , \ldots , u_k ) \,,\, [z_i , \gamma]
  \Big\rangle \\
  & = & \sum_{i=0}^k (-1)^{i+1} \Big\langle
  \phi(u_0,\ldots , \wh u_i , \ldots , u_k ) \,,\, [\theta(u_i) , \gamma]
  \Big\rangle \\
  & \stackrel{\eqref{E:ad_adj}}{=} & \sum_{i=0}^k (-1)^{i} \Big\langle
  \left[u_i , \phi(u_0,\ldots , \wh u_i , \ldots , u_k )\right] \,,\, \gamma
  \Big\rangle \, .
\end{eqnarray*}
The lemma follows.
\end{proof}

The \emph{Laplacian} is the $\fg_{0,0}$--module map defined by 
$$
  \square 
  \ := \ \partial \partial^* \,+\, \partial^* \partial : {\fg_w^\perp} \ot \tw^k\fn_w^* \ \to \ 
  {\fg_w^\perp} \ot \tw^k \fn_w^* \, .
$$
Let $\cH^k := \tker\,\square \subset {\fg_w^\perp} \ot \tw^k\fn_w^*$.  By now standard arguments (see \cite[Proposition 2.1]{MR0142696} or \cite[Corollary 3.3.1]{MR2532439}) we have the following.

\begin{proposition} \label{P:new_coh}
Each element in $H^k(\fn_w,\fg^\perp_w)$ admits a unique representative in $\cH^k$.   The corresponding identification $H^k(\fn_w,\fg^\perp_w)\,\simeq\,\cH^k$ is a $\fg_{0,0}$--module isomorphism.
\end{proposition}

Let 
\begin{equation} \label{E:cH_deg}
  \cH^1 \ = \ \cH^1_0 \,\op\, \cH^1_1 \,\op\, \cH^1_2 \ \subset \ 
              \left( \fg_{-1,<-\sfa} \,\op\, \fg_{0,-} \,\op\, \fg_{1,<\sfa}
              \right) \ \ot \ \fn_w^+ \, .
\end{equation}
denote the $Z_\tti$--graded decomposition of $\cH^1$.  

\begin{corollary} \label{C:new_coh}
Let $\cF$ be the adapted frame bundle \eqref{E:cY} over an integral manifold of the Schubert system $\cB_w$.  The bundle $\cF$  admits a sub-bundle $\cF^0$ on which $\lambda : \cF \to \tker\,\d^1$ restricts to take values in $\cH^1_1$.  Additionally, the $\vartheta_{\fg_w}$ compose a coframing of $\cF^0$, and there exists $\mu : \cF^0 \to \fg_{1,<\sfa} \ot \fn_w^+$ such that $\vartheta_{1,<\sfa} = \mu(\vartheta_{\fn_w})$ on $\cF^0$.

If $\lambda$ vanishes on $\cF^0$, then $\mu$ takes value in $\cH^1_2$.  If, for any integral manifold of $\cB_w$, $\lambda$ and $\mu$ vanish on $\cF^0$, then the Schubert system is rigid.
\end{corollary}

\begin{proof}
From \eqref{E:H1} and Proposition \ref{P:new_coh} we see that 
$$
  \tker\,\d^1 \ = \ \tim\,\d^0 \,\op\, \cH^1_1
  \quad \hbox{ and } \quad 
  \ker\,\e^1 \ = \ \cH^1_2 \, ;
$$ 
moreover, the former is a $\fg_{0,0}$--module decomposition.  Setting $(\tim\,\d^0)^\perp := \cH^1_1$, Lemma \ref{L:norm_lambda} yields a sub-bundle $\cF^0$ on which $\lambda$ takes values in $\cH^1_1$ and $\vartheta_{1,<\sfa} = \mu(\vartheta_{\fn_w})$.  

If $\lambda$ vanishes on $\cF^0$, then from Section \ref{S:lambda=0} we see that $\mu$ takes values in $\cH^1_2$.  

Finally, Proposition \ref{P:Bw_rigid} asserts that the Schubert system is rigid if and only if every integral manifold $\cF$ of \eqref{E:Bw_eds} admits a sub-bundle on which $\vartheta_{\fg_w^\perp}$ vanishes.  Recall that $\fg_w^\perp = \fn_w^\perp \op \fg_{0,-} \op \fg_{1,<\sfa}$.  The vanishing of $\vartheta_{\fn_w^\perp}$ on $\cF^0\subset \cF$ is automatic by Lemma \ref{L:Bw_eds}.  By \eqref{E:th0neg}, $\vartheta_{0,-}$ vanishes on $\cF^0$ if and only if $\lambda$ does.   Similarly, the vanishing of $\vartheta_{1,<\sfa}$ is equivalent to the vanishing of $\mu$ by Lemma \ref{L:norm_lambda}.
\end{proof}
\begin{equation} \label{E:F0}
  \hbox{ \emph{From this point on, we restrict to the bundle 
  $\cF^0$ of Corollary \ref{C:new_coh}.} }
\end{equation}

\begin{remark*}
It is an immediate consequence of Corollary \ref{C:new_coh} that the Schubert system is rigid if $\cH^1_{+} = \{0\}$.  We will see in Sections \ref{S:induct} \& \ref{S:finish} that it suffices for subspaces $\cH^1_{1,\sfa-1} \subset \cH^1_1$ and $\cH^1_{2,2\sfa-1} \subset \cH^1_2$ to vanish.  These subspaces are components of the $(Z_\tti,Z_w)$--bigraded decomposition
\begin{equation} \label{E:cH1bi}
  \cH^1 \ = \ \oplus_{i,s} \, \cH^1_{i,s} \,.
\end{equation}
The integers $s=\sfa-1$ and $s=2\sfa-1$ are respectively the maximal $Z_w$--eigenvalues on $\cH^1_1$ and $\cH^1_2$.
\end{remark*}

\subsection{Action of the Laplacian} \label{S:action}

In this section we derive general formulas, \eqref{E:sqL1} and \eqref{E:sqL2} respectively, that will be used to determine $\cH^1_1$ and $\cH^1_2$.

Given $\xi \in \fg$, let $\epsilon_\xi : \fg\ot \tw^k \fg \, \to \, \fg\ot \tw^{k+1}\fg$ denote the natural map induced by exterior product with $\xi$.  Let $\iota_\xi : \fg \ot \tw^k\fg \, \to \, \fg \ot \tw^{k-1} \fg$ denote the natural map induced by the interior product with $\xi$; the interior product is computed with respect to the Killing form.  Let $L'_\xi : \fg \to \fg$ denote the adjoint action of $\xi$ on $\fg$, and let $L''_\xi : \tw \fg \to \tw\fg$ denote the induced action on the exterior algebra.  Then the induced action $ \cL_\xi : \fg \ot \tw \fg \, \to \, \fg \ot \tw\fg$ is given by $\cL_\xi = \cL'_\xi + \cL''_\xi$ where $\cL'_\xi = L'_\xi \ot \mathbf{1}$ and $\cL''_\xi = \mathbf{1} \ot L''_\xi$.  

The following lemma will be very useful.  

\begin{lemma}[{\cite[Lemma 3.3.2]{MR2532439}}] \label{L:formulas}
Fix $\xi,\z \in \fg$, and let $\{\xi_\ell\}$ and $\{\z_\ell\}$ be Killing dual bases of $\fg$.  Then $\cL'_\xi$ commutes with $\cL''_\z$, $\epsilon_\z$ and $\iota_\z$,
\begin{a_list}
\item $\sfd^* \circ \epsilon_\xi \,+\, \epsilon_\xi \circ \sfd^* \,=\, \cL_\xi$,
\item $\cL'_\xi \circ \sfd^* \,-\, \sfd^* \circ \cL'_\xi \,=\,
       \sum_\ell \iota_{[\z_\ell , \xi]} \circ \cL'_{\xi_\ell}$,
\item $\sum_\ell \epsilon_{\z_\ell} \circ \iota_{[\xi,\xi_\ell]} \,=\, -\cL''_\xi$.
\end{a_list}
\end{lemma}

\begin{remark*}
When consulting \cite{MR2532439}, note that their version of $\sfd^*$ differs from ours by a sign.
\end{remark*}

Fix an orthogonal basis $\{ H_1,\ldots , H_n\}$ of $\fh$.  Let $E_\a \in \fg_{\a}$ be root vectors scaled so that $(E_\a , E_{-\b}) = \d_{\a\b}$.  Then $\{\xi_\ell\} = \{ E_\a \}_{\a\in\Delta} \cup \{H_j\}_{j=1}^n$ defines a basis of $\fg$.  Let $\{\z_\ell\}$ denote the Killing dual basis.

Given $0\le a \in \bZ$, set
$$
   \Delta(w,a) \ = \ \{ \a \in \Delta(w) \ | \ \a(Z_w) = a \} \, .
$$
From \eqref{E:abelian} and \eqref{E:d}, we see that $\partial: \fg_w^\perp \ot \tw^k\fn_w^+ \to \fg_w^\perp \ot \tw^{k+1}\fn_w^+$ is given by 
\begin{eqnarray}
  \nonumber
  \partial_{|\fg_{-1,-s}\ot\sw^k\fn_w^+} & = & 0 \, ,
  \quad \hbox{ for } s > \sfa\,;\\
  \label{E:partial}
  \displaystyle \partial_{|\fg_{0,-\ell}\ot\sw^k\fn_w^+} & = &
  \sum_{\stack{\a \in \Delta(w,a)}{a+\ell>\sfa}}
  \epsilon_{E_\a} \circ \cL'_{E_{-\a}} \, ,
  \quad\hbox{ for }  \ell > 0 \,;\\
  \nonumber
  \partial_{|\fg_{1,b}\ot\sw^k\fn_w^+} & = & 
  \sum_{\stack{\a \in \Delta(w,a)}{b-a < 0}}
  \epsilon_{E_\a} \circ \cL'_{E_{-\a}} \, , \quad\hbox{ for }
  0 \le b < \sfa \, .
\end{eqnarray}

\subsection*{Bidegree \boldmath $(1,k)$ \unboldmath}
The component of $\fg_w^\perp \ot \fn_w^* \simeq \fg_{0,-} \ot \fn_w^+$ of $(Z_\tti,Z_w)$--bidegree $(1,k)$ is
\begin{subequations} \label{E:Lambdak}
\begin{equation}
  \Lambda^k \ := \ \bigoplus_{\stack{0<\ell}{k+\ell\le\sfa}}
  \Lambda^k_{-\ell} \quad \hbox{with} \quad
  \Lambda^k_{-\ell} \ := \ \fg_{0,-\ell} \ot \fg_{1,k+\ell} \,,\quad
  k \le \sfa-1 \, .
\end{equation}
Note that $\cH^1_{1,k}= \Lambda^k \cap \cH^1_1$, see \eqref{E:cH1bi}.  If $k+\ell<0$, then $\fg_{1,k+\ell} = \{0\}$ so that 
\begin{equation}
  \Lambda^k_{-\ell} \ = \ \{0\} \quad\hbox{when}\quad k+\ell < 0 \, .
\end{equation}
\end{subequations}  
We have $\partial^*(\Lambda^k) \subset \fg_{1,k}$.  Making use of \eqref{E:partial}, we compute
\begin{eqnarray*}
  \square_{|\Lambda^k_{-\ell}} & = & 
  \sum_{\stack{\a\in\Delta(w,a)}{a>k}} 
  \epsilon_{E_\a} \circ \cL'_{E_{-\a}} \circ \partial^*
  \ + \sum_{\stack{\a\in\Delta(w,a)}{a+\ell>\sfa}} 
  \partial^* \circ \epsilon_{E_\a} \circ \cL'_{E_{-\a}} \\
  & \stackrel{(\ast)}{=} & 
  \sum_{\stack{\a\in\Delta(w,a)}{k<a\le \sfa-\ell}} 
  \epsilon_{E_\a} \circ \cL'_{E_{-\a}} \circ \partial^*
  \ + \sum_{\stack{\a\in\Delta(w,a)}{a+\ell>\sfa}} 
  \left( \partial^* \circ \epsilon_{E_\a} \circ \cL'_{E_{-\a}} \,+\,
  \epsilon_{E_\a} \circ \cL'_{E_{-\a}} \circ \partial^* \right) \, .
\end{eqnarray*}
Lemma \ref{L:formulas}(a,b) allows us to rewrite the second summand in $(\ast)$ as
$$
  \partial^* \circ \epsilon_{E_\a} \circ \cL'_{E_{-\a}} \,+\,
  \epsilon_{E_\a} \circ \cL'_{E_{-\a}} \circ \partial^* \ = \ 
  \cL_{E_\a}\circ\cL'_{E_{-\a}} \ + \ \epsilon_{E_\a} \circ
  \textstyle \sum_j \iota_{[\z_j,E_{-\a}]} \circ \cL'_{\xi_j} \, .
$$
Given $\a \in \Delta(w)$, \eqref{E:abelian} implies $\cL_{E_{\a}} \circ \cL'_{E_{-\a}} = \cL'_{E_{\a}} \circ \cL'_{E_{-\a}}$ on $\fg_w^\perp \ot \tw^k\fn_w^+$.  Also note that 
\begin{equation} \label{E:iota}
  \hbox{$\iota_\z$ will vanish on $\Lambda^k_{-\ell}$ 
        unless $\z \in \fg_{-1,-k-\ell}$.}
\end{equation}  Therefore,  the second summand in $(\ast)$ may be further refined to
$$
  \partial^* \circ \epsilon_{E_\a} \circ \cL'_{E_{-\a}} \,+\,
  \epsilon_{E_\a} \circ \cL'_{E_{-\a}} \circ \partial^* \ = \ 
  \cL'_{E_\a}\circ\cL'_{E_{-\a}} \ + \ \epsilon_{E_\a} \circ
  \sum_{\stack{\z_j \in \fg_{0,m}}{m=a-k-\ell}} 
  \iota_{[\z_j,E_{-\a}]} \circ \cL'_{\xi_j} \, .
$$
Thus
\begin{eqnarray*}
  \square_{|\Lambda^k_{-\ell}} & = & 
  \sum_{\stack{\a\in\Delta(w,a)}{k<a\le \sfa-\ell}} 
  \epsilon_{E_\a} \circ \cL'_{E_{-\a}} \circ \partial^* \ + 
  \sum_{\stack{\a\in\Delta(w,a)}{a+\ell>\sfa}} \cL'_{E_\a} \circ \cL'_{E_{-\a}}
  \\ & & + 
  \sum_{\stack{\a\in\Delta(w,a)}{a+\ell>\sfa}} 
  \sum_{\stack{\z_j \in \fg_{0,m}}{m=a-k-\ell}}
  \epsilon_{E_\a} \circ \iota_{[\z_j,E_{-\a}]} \circ \cL'_{\xi_j} \, .
\end{eqnarray*}
Swapping the order of summation in the third term above yields
$$
  \sum_{\stack{\a\in\Delta(w,a)}{a+\ell>\sfa}} 
  \sum_{\stack{\z_j \in \fg_{0,m}}{m=a-k-\ell}}
  \epsilon_{E_\a} \circ \iota_{[\z_j,E_{-\a}]} \circ \cL'_{\xi_j} \ = \ 
  \sum_{\stack{\z_j \in \fg_{0,m}}{\sfa-k-2\ell<m\le\sfa-k-\ell}}
  \sum_{\stack{\a\in\Delta(w,a)}{a=m+k+\ell}} 
  \epsilon_{E_\a} \circ \iota_{[\z_j,E_{-\a}]} \circ \cL'_{\xi_j}
$$
Again, \eqref{E:iota} allows us to replace the summation over $\{E_{\a} \ | \ \a\in\Delta(w,a) \,, \ a=m+k+\ell\}$ with a summation over all $\z_p$.  Then Lemma \ref{L:formulas}(c) yields
$$
  \sum_{\stack{\a\in\Delta(w,a)}{a+\ell>\sfa}} 
  \sum_{\stack{\z_j \in \fg_{0,m}}{m=a-k-\ell}}
  \epsilon_{E_\a} \circ \iota_{[\z_j,E_{-\a}]} \circ \cL'_{\xi_j} \ = \ 
  -\sum_{\stack{\z_j \in \fg_{0,m}}{\sfa-k-2\ell<m\le\sfa-k-\ell}}
  \cL''_{\z_j} \circ \cL'_{\xi_j}
$$
The computations above yield the following expression for the Laplacian on $\Lambda^k_{-\ell}$
\begin{equation} \label{E:sqL1}
  \square_{|\Lambda^k_{-\ell}} \ = \ 
  \sum_{\stack{\a\in\Delta(w,a)}{\sfa-\ell<a}} 
  \underbrace{ \cL'_{E_\a}\circ\cL'_{E_{-\a}} 
  }_{\Lambda^k_{-\ell}} -
  \sum_{\stack{\z_j\in\fg_{0,m}}{\sfa-k-2\ell<m\le\sfa-k-\ell}}
  \underbrace{ \cL''_{\z_j}\circ\cL'_{\xi_j} 
  }_{\Lambda^k_{-\ell-m}} + 
  \sum_{\stack{\a\in\Delta(w,a)}{k<a\le\sfa-\ell}}
  \underbrace{ \epsilon_{E_\a}\circ\cL'_{E_{-\a}}\circ\partial^*
  }_{\Lambda^k_{k-a}} 
\end{equation}
The underbraces above indicate which $\Lambda^k_{-\bullet}$ the operator takes values in.  

\subsection*{Bidegree \boldmath $(2,k)$\unboldmath}  
Let 
\begin{subequations}\label{E:Muk}
\begin{equation}
  \tM^k \ := \ \bigoplus_{c =k-\sfa}^{\sfa-1} 
  \tM^k_c \quad \hbox{ with } \quad
  \tM^k_c \ := \ \fg_{1,c} \ot \fg_{1,k-c} \,, \quad 
  0 \le k \le 2\sfa-1 
\end{equation}
be the component of $\fg_w^\perp \ot \fn_w^*$ of $(Z_\tti , Z_w)$--bidegree $(2,k)$.  Note that $\cH^1_{2,k} = \tM^k \cap \cH^1_2$, cf. \eqref{E:cH1bi}.  If $c < 0$ then $\fg_{1,c} = \{0\}$; if $k < c$, then $\fg_{1,k-c} = \{0\}$.  Thus, 
\begin{equation}
  \tM^k_c \ = \ \{0\} \quad \hbox{if either} \quad c < 0 \quad\hbox{or}\quad 
  k < c \, .
\end{equation}
\end{subequations}

By \eqref{E:abelian} the adjoint $\partial^*$ vanishes on $\tM^k$.  So, \eqref{E:partial} yields
$$
    \square_{|\tM^k_c} \ = \ \sum_{\stack{\a \in \Delta(w,a)}{c<a}} \left(
      \partial^*\circ\epsilon_{E_\a}\circ\cL'_{E_{-\a}} \,+\, 
      \epsilon_{E_\a}\circ\cL'_{E_{-\a}}\circ\partial^*
    \right) \,.
$$
Applying Lemma \ref{L:formulas}, a computation similar to that for \eqref{E:sqL1} yields 
\begin{equation} \label{E:sqL2}
  \square_{|\tM^k_c} \ = \ 
  \sum_{\stack{\a\in\Delta(w,a)}{c<a}}
  \underbrace{ \cL'_{E_\a}\circ\cL'_{E_{-\a}}
  }_{\mathrm{valued \ in \ } \tM^k_c} \ - 
  \sum_{\stack{\z_j\in\fg_{0,m}}{2c-k<m}} 
  \underbrace{ \cL''_{\z_j}\circ\cL'_{\xi_j} 
  }_{\mathrm{valued \ in \ } \tM^k_{c-m}}
\end{equation}

\subsection*{Casimirs}

Note that the $\sum_{\z_j\in\fg_{0,0}} \cL''_{\z_j} \circ \cL'_{\x_j}$ term of \eqref{E:sqL1} and \eqref{E:sqL2} satisfies
\begin{equation}\label{E:C00}
  \sum_{\z_j\in\fg_{0,0}} \cL_{\z_j}'' \circ\cL'_{\x_j} \ = \ 
  \half \sum_{\z_j\in\fg_{0,0}} \left( 
    \cL_{\z_j}\circ\cL_{\x_j} \,-\, \cL'_{\z_j}\circ\cL'_{\x_j} \,-\, 
    \cL''_{\z_j}\circ\cL''_{\x_j} \right) \, .
\end{equation}
The terms $\sum \cL \circ \cL$, $\sum \cL' \circ \cL'$ and $\sum \cL'' \circ \cL''$ appearing in the right-hand side of \eqref{E:C00} are the $\fg_{0,0}$--Casimirs on $\fg^\perp_w\ot\fn_w^+$, $\fg^\perp_w$ and $\fn_w^+$, respectively.  

\begin{lemma} \label{L:C00}
Let $U_\b \subset \fg^\perp_w$ and $U_\c \subset \fn_w^+$ be irreducible $\fg_{0,0}$--modules of highest weights $\b,\c\in\Delta(\fg)$.  Let $U_\pi \subset U_\b \ot U_\c$ be an irreducible $\fg_{0,0}$--module of highest weight $\pi$.  Then \eqref{E:C00} acts on $U_\pi$ by a scalar $c \le (\b,\c)$ with equality if and only if $\pi = \b+\c$.
\end{lemma}

\begin{proof}
Recall that the Casimir acts on an irreducible module of highest weight $\nu$ by the scalar $|\n|^2 + 2(\n,\rho_0)$, where $\rho_0:=\half \sum_{\a\in\Delta^+(\fg_{0,0})} \a$.

The highest weight occurring in $U_\b \ot U_\c$ is $\b+\c$.  Because both $\pi$ and $\b+\c$ lie in the dominant Weyl chamber, we have $|\pi| \le |\b+\c|$.  The weight $\rho_0$ lies in the interior of the dominant Weyl chamber, and so has the property that $(\c+\b-\pi , \rho_0) \ge 0$, with equality if and only if $\pi = \c+\b$.  It follows that \eqref{E:C00} acts by the scalar
\begin{eqnarray*}
  c & = & \half \left( |\pi|^2 - |\b|^2 - |\c|^2  \right) 
  \ + \ (\pi-\b-\c,\rho_0) \\
  & \stackrel{(\star)}{\le} & \half \left( |\b+\c|^2 - |\b|^2 - |\c|^2  \right) \ + \ (\pi-\b-\c,\rho_0)  \\
  & = & (\c , \b) \ + \ (\pi-\b-\c,\rho_0) \ \stackrel{(\dagger)}{\le} \ (\c,\b) \, .
\end{eqnarray*}
Equality holds in $(\star)$ and $(\dagger)$ if and only if $\pi = \b+\c$.
\end{proof}

\begin{lemma} \label{L:knapp}
\begin{a_list}
\item  Let $\varsigma,\tau \in \Delta(\fg)$.  Then $L'_{E_\varsigma} \circ L'_{E_{-\varsigma}}$ acts on $\fg_\tau$ by a scalar $c\ge0$.  Moreover, $c = 0$ if and only if $\varsigma-\tau\not\in\Delta(\fg)$.
\item Assume $\varsigma\in\Delta(\fg_1)$ and $\tau \in \Delta(\fg_0)$.  If both $\varsigma\pm\tau\in\Delta(\fg)$, then $c = |\varsigma|^2$.  If $\varsigma-\tau \in \Delta(\fg)$ and $\varsigma+\tau\not\in\Delta(\fg)$, then  $c = (\varsigma,\tau) > 0$.
\item Assume $\varsigma,\tau\in\Delta(\fg_1)$.  If $\varsigma-\tau \in \Delta(\fg)$, then $c = (\varsigma,\tau)>0$.  Otherwise $c=0$.
\end{a_list}
\end{lemma}

\begin{proof}
The lemma is the specialization of a standard result \cite[Corollary 2.37]{MR1920389} in representation theory to the case that $\fg = \fg_1 \op \fg_0 \op \fg_{-1}$.
\end{proof}

\subsection{Bidegree \boldmath $(1,\sfa-1)$ \boldmath} \label{S:step1}

Recall the bi-graded decomposition \eqref{E:cH1bi}.  The computation that follows will determine necessary and sufficient conditions for $\cH^1_{1,\sfa-1} = \{0\}$, cf. Lemma \ref{L:1,a-1}.  

Equations \eqref{E:sqL1} and \eqref{E:C00} yield
\begin{equation}\label{E:La-1}
  \square_{|\Lambda^{\sfa-1}} \ = \
  \sum_{\a \in \Delta(w,\sfa)} \cL'_{E_\a} \circ \cL'_{E_{-\a}} \ -
  \half \sum_{\z_j\in\fg_{0,0}} \left( 
    \cL_{\z_j}\circ\cL_{\x_j} \,-\, \cL'_{\z_j}\circ\cL'_{\x_j} \,-\, 
    \cL''_{\z_j}\circ\cL''_{\x_j} \right) \, .
\end{equation}
The Laplacian acts on an irreducible $\fg_{0,0}$--submodule $U \subset \Lambda^{\sfa-1}$ by a scalar, which is determined as follows.  Recall from Section \ref{S:CHSS_Dw} that the irreducible $\fg_{0,0}$--submodules of $\fg_{0,-1}$ are the $\fg_{0,-B}$ with $|B| = 1$.  The highest weight of $\fg_{0,-B}$ is $-\b=-\a_\ttj$ for some $\ttj \in \ttJ$.  (Recall that $\a_\ttj$ is a simple root of $\fg$.)  Similarly, the irreducible $\fg_{0,0}$--submodules of $\fg_{1,\sfa}$ are the $\fg_{1,C}$ with $|C| = \sfa$.  Let $\c \in \Delta(w,\sfa)$ be the highest weight of $\fg_{1,C}$.  Let $U \subset \fg_{0,-B} \ot \fg_{1,C}$ be an irreducible $\fg_{0,0}$--submodule of highest weight $\pi$.  From Lemmas \ref{L:C00} \& \ref{L:knapp} it follows that 
\begin{equation} \label{E:scalar1}
  \square \hbox{ acts on $U$ by a scalar } 
  \stackrel{(\diamond)}{\ge} \ 
  (\c,\b) \ + 
  \underbrace{ \sum_{\a\pm\b\in\Delta} |\a|^2 \ - 
    \sum_{\stack{\a+\b\in\Delta}{\a-\b \not\in \Delta}} (\a,\b)
  }_{\mathrm{sums \ over \ } \a \in \Delta(w,\sfa)}
   \ \stackrel{(\star)}{\ge} \ 0 \, ,
\end{equation}
and equality holds at $(\diamond)$ if and only if $\pi = \c-\b$.

We now determine when this eigenvalue is zero (equivalently, when $\cH^1_{1,\sfa-1} \not=\{0\}$).  There are three cases to consider:

$\bullet$ $(\c,\b) = 0$:  If equality holds at $(\star)$, then it must be the case that $[\fg_{1,\sfa} \,,\, \fg_{-\b}] = \{0\}$.  By Proposition \ref{P:nw2graded}, this implies that $\fg_{-\b} \subset \fg_{0,-}$ stabilizes $\fn_w$.  This contradicts \eqref{E:nwstab}.  Thus, strict inequality holds at $(\star)$, and $\square$ acts on $U$ by a positive scalar. \smallskip

$\bullet$ $(\c , \b) > 0$:  Then $\square$ acts on $U$ by a scalar $\ge (\c , \b) > 0$. \smallskip

$\bullet$ $(\c , \b) < 0$:
Then $\square$ acts on $U$ by zero if and only if $\pi = \c-\b$ and $\{ \a \in \Delta(w,\sfa) \ | \ \a+\b\in\Delta\} = \{\c\}$.  The latter condition is equivalent to $[ \fg_\b \,,\, \fg_{1,\sfa}] = [ \fg_\b \,,\, \fg_\c] \not=\{0\}$.  Recall that $(\c,\b) < 0$ implies $\c-\b\not\in\Delta$, which is equivalent to $[\fg_\c \,,\, \fg_{-\b}] = \{0\}$.  \smallskip 

\begin{definition} \label{D:H1}
We say that \emph{$\sfH_1$ is satisfied} if there exist no irreducible $\fg_{0,0}$--sub-modules $\fg_{0,-B} \subset \fg_{0,-1}$ and $\fg_{1,C} \subset \fg_{1,\sfa}$ with highest weights $-\b \in \Delta(\fg_{0,-B})$ and $\c\in\Delta(\fg_{1,C})$, respectively, such that the following three conditions hold
$$
  \left[ \fg_{-\b} \,,\, \fg_\c \right] \ = \ \{0\} \ \not= \ 
  \left[ \fg_\b \,,\, \fg_\c \right] \ = \ 
  \left[ \fg_\b \,,\, \fg_{1,\sfa} \right] \, .
$$
\end{definition}

\noindent Let $\lambda^k$ denote the $\cH^1_{1,k}$--valued component of $\lambda$, see \eqref{E:cH1bi}.  We have established the following.

\begin{lemma}\label{L:1,a-1}
The eigenvalues of the Laplacian on $\Lambda^{\sfa-1}$ are strictly positive if and only if $\sfH_1$ is satisfied.  Equivalently, $\cH^1_{1,\sfa-1} = \{0\}$ if and only if $\sfH_1$ is satisfied.  In particular, if $\sfH_1$ is satisfied, then $\lambda^{\sfa-1}$ vanishes on $\cF^0$, cf. \eqref{E:F0}.

If $\sfH_1$ fails for a pair $-\b,\c$ then $\cH^1_{1,\sfa-1}$ contains an irreducible $\fg_{0,0}$--module of highest weight $\c-\b$.
\end{lemma} 



\subsection{Induction} \label{S:induct}

In this section we prove the following.

\begin{proposition} \label{P:lambda=0}
If $\sfH_1$ is satisfied, then $\lambda$ (equivalently, $\vartheta_{0,-}$) is identically zero on $\cF^0$.
\end{proposition}

We argue by induction.  By Lemma \ref{L:1,a-1}, $\lambda^{>\sfa-2}$ vanishes on $\cF^0$.  Given any $k < \sfa-1$ with the property that $\lambda^{>k}=0$ on $\cF^0$, we will show that $\lambda^{\ge k}$ vanishes on $\cF^0$.  This will establish Proposition \ref{P:lambda=0}.

Write $\lambda^k = (\lambda^k_{-1} , \cdots , \lambda^k_{k-\sfa})$, where $\lambda^k_{-\ell}$ denotes the $\Lambda^k_{-\ell}$ component of $\lambda^k$; see \eqref{E:Lambdak}, and recall that $\lambda^k_{-\ell} = 0$ if $k+\ell<0$.  

\begin{lemma} \label{L:lambdak=0} \quad
\begin{a_list}
\item Given $0 < \ell < \sfa-k$ and $\z \in \fg_{0,\sfa-k-\ell}$, we have $\cL'_\z(\lambda^k_{k-\sfa}) = -\cL''_\z(\lambda^k_{-\ell})$.
\item 
If $\lambda^k_{k-\sfa} = 0$, then 
$\lambda^k = (\lambda^k_{-1},\ldots,\lambda^k_{k-\sfa}) = 0$.
\item %
The component $\lambda^k_{k-\sfa}$ vanishes.
\end{a_list}
\end{lemma}

\noindent Lemma \ref{L:lambdak=0} yields $\lambda^{\ge k}=0$, completing the induction and establishing Proposition \ref{P:lambda=0}.  

\begin{proof}[Proof of Lemma \ref{L:lambdak=0}(a)]
Let $\lambda_{-\ell}$ denote the component of $\lambda$ taking values in $\fg_{0,-\ell} \ot \fn_w^*$, so that $\vartheta_{0,-\ell} = \lambda_{-\ell}(\vartheta_{\fn_w})$.  Applying the Maurer-Cartan equation \eqref{E:mce} to this equation yields 
\begin{equation} \label{E:dlambda_deg} \renewcommand{\arraystretch}{1.5}
\begin{array}{rcl}
  0 & \equiv & \td \lambda_{-\ell} \wedge \vartheta_{\fn_w}
  \ + \ \sum_{a=\ell}^\sfa [ \vartheta_{1, a-\ell} \,,\, \vartheta_{-1,-a} ]
  \ + \ \sum_{m\ge\ell} 
  [ \vartheta_{0,m-\ell} \,,\, \lambda_{-m}(\vartheta_{\fn_w}) ] \\
  &  & 
  \ - \ \lambda_{-\ell} \big( [ \vartheta_{0,\ge0} \,,\, \vartheta_{\fn_w}] \big) 
  \ - \ \lambda_{-\ell} 
  \big( [ \vartheta_{\fn_w} \,,\, \lambda(\vartheta_{\fn_w}) ] \big) \\
  & &  + \ \half \, \sum_{0 < m < \ell} 
  [ \lambda_{-m}(\vartheta_{\fn_w}) \, , \, \lambda_{m-\ell}(\vartheta_{\fn_w}) ]
\end{array}\end{equation}

The vanishing of $\lambda^{>k}$ on $\cF^0$ is equivalent to 
\begin{equation} \label{E:induct}
  \lambda_{-\ell}(\fg_{-1,-a}) \ = \ \{0\} \quad\hbox{when}\quad 
  0 \le a \le \sfa \ \hbox{ and } \ k+1 \le a-\ell \le \sfa-1 \, .
\end{equation}  

Given $u \in \fg_{-1,-a}$ and $\z \in \fg_{0,m}\subset\fg_{0,>0}$, the coframing of Lemma \ref{L:norm_lambda} defines vector fields $\hat u$ and $\hat \z$ on $\cF^0$ by the conditions $\vartheta_{\fg_w}(\hat u) = u$ and $\vartheta_{\fg_w}(\hat \z) = \z$.  Evaluating the 2-form \eqref{E:dlambda_deg} on $\hat u \wedge \hat\z$, and making use of \eqref{E:induct} yields\footnote{In the context of Section \ref{S:torsion}, \eqref{E:more_tor} is additional torsion in the system imposed by \eqref{E:induct}.}
\begin{equation}\label{E:more_tor}
  \left[ \z \,,\, \lambda_{-\ell-m}(u) \right] \ = \ 
  \lambda_{-\ell}\left( [ \z , u ] \right)\,,
\end{equation} 
for all $u \in \fg_{-1,-a}\subset \fn_w$, such that $k+1 \le a-\ell \le \sfa-1$, and $\z\in\fg_{0,m} \subset\fg_{0,>0}$.  Setting $a = \sfa$ and $m = \sfa-k-\ell$ yields $\lambda_{-\ell}\left( [\z , u]\right) = \left[ \z \,,\, \lambda_{k-\sfa}(u) \right]$ for all $u \in \fg_{-1,-\sfa}$ and $\z \in \fg_{0,m}$.  Under the identification $\fn_w^* \simeq \fn_w^+$ this is equivalent to $\cL'_\z(\lambda^k_{k-\sfa}) = -\cL''_\z(\lambda^k_{-\ell})$
\end{proof}

\begin{proof}[Proof of Lemma \ref{L:lambdak=0}(b)]
By \eqref{E:more_tor} we have $[ \z \,,\, \lambda^k_{k-a}(u)] = \lambda^k_{k-a+1}( [\z,u])$, for all $\z \in \fg_{0,1}$ and $u \in \fg_{-1,-a}$.  We will show that $\lambda^k_{k-a} = 0$ implies $\lambda^k_{k-a+1}=0$.  Part (b) will then follow from an inductive argument.  

Assume $\lambda^k_{k-a} = 0$.  Equivalently, $\lambda_{k-a}(u) = 0$ for all $u \in \fg_{-1,-a}$.  Then $0 = [\z \,,\, \lambda_{k-a}(u)] = \lambda_{k-a+1}([\z,u])$ for all $\z\in\fg_{0,1}$.    Lemma \ref{L:nondeg} implies $[\fg_{0,1} \,,\,\fg_{-1,-a}] = \fg_{-1,1-a}$.  It follows that $\lambda_{k-a+1}(\fg_{-1,1-a}) = 0$.  Equivalently, $\lambda^k_{k-a+1} = 0$.  
\end{proof}

\begin{proof}[Proof of Lemma \ref{L:lambdak=0}(c)]
From \eqref{E:sqL1}, the component of $\square_{|\Lambda^k}$ taking values in $\Lambda^k_{k-\sfa}$ is 
\begin{equation} \nonumber 
  \square(\lambda^k_{-1},\cdots,\lambda^k_{k-\sfa}
         )_{\Lambda^k_{k-\sfa}} \ = \ 
  \sum_{\stack{\a\in\Delta(w,a)}{k<a}} 
  \cL'_{E_\a}\circ\cL'_{E_{-\a}}(\lambda^k_{k-\sfa}) \ - \ 
  \sum_{\ell=1}^{\sfa-k} 
  \left( \sum_{\z_j\in\fg_{0,\sfa-k-\ell}}
  \cL''_{\z_j}\circ\cL'_{\x_j}(\lambda^k_{-\ell}) \right) .
\end{equation}
An application of Lemma \ref{L:lambdak=0}(a) to this expression yields 
\begin{equation} \label{E:sqk} \renewcommand{\arraystretch}{1.5}
\begin{array}{rcl}
  \displaystyle
  \square(\lambda^k_{-1},\cdots,\lambda^k_{k-\sfa}
         )_{\Lambda^k_{k-\sfa}} 
  & = & \displaystyle
  \sum_{\stack{\a\in\Delta(w,a)}{k<a}} 
  \cL'_{E_\a}\circ\cL'_{E_{-\a}}(\lambda^k_{k-\sfa}) - 
  \sum_{\z_j\in\fg_{0,0}}
  \cL''_{\z_j}\circ\cL'_{\x_j}(\lambda^k_{k-\sfa}) \\
  & & \displaystyle
  + \sum_{\z_j\in\fm^{\sfa-k-1}}
  \cL'_{\x_j}\circ\cL'_{\z_j}(\lambda^k_{k-\sfa})\, ,
\end{array}\end{equation}
where 
$$
  \fm^{\sfa-k-1} \ := \ \fg_{0,1} \op \cdots \op \fg_{0,\sfa-k-1} \,.
$$ 
Define $\wh\square : \Lambda^k_{k-\sfa} \to \Lambda^k_{k-\sfa}$ by 
$$
  \wh\square \ := \ 
  \sum_{\stack{\a\in\Delta(w,a)}{k<a}} \cL'_{E_\a}\circ\cL'_{E_{-\a}} 
  + \sum_{\z_j\in\fm^{\sfa-k-1}}
  \cL'_{\x_j}\circ\cL'_{\z_j} 
  - \sum_{\z_j\in\fg_{0,0}}  \cL''_{\z_j}\circ\cL'_{\x_j}\, .
$$
Then \eqref{E:sqk} reads
$\square(\lambda^k_{-1},\cdots,\lambda^k_{k-\sfa}
  )_{\Lambda^k_{k-\sfa}} = \wh\square(\lambda^k_{k-\sfa})$.  The left-hand side must vanish, since $\lambda^k$ takes values in $\cH^1_{1,k}$.  So it remains to show that 
$\tker\,\wh\square=\{0\}$.

Recall that $\Lambda^k_{k-\sfa} = \fg_{0,k-\sfa} \ot \fg_{1,\sfa}$.  Note that $\wh\square$ is a $\fg_{0,0}$--module map.  Therefore $\Lambda^k_{k-\sfa}$ admits a $\fg_{0,0}$--module decomposition into $\wh\square$--eigenspaces.    The eigenvalues of $\wh\square$ are computed as follows.  Let $\fg_{0,-B} \subset \fg_{0,k-\sfa}$ be an irreducible $\fg_{0,0}$--submodule of highest weight $-\b \in \Delta(\fg_{0,-B})$; let $\fg_{1,C} \subset \fg_{1,\sfa}$ be an irreducible module of highest weight $\c \in \Delta(\fg_{1,C})$.  Let $U \subset \fg_{0,-B} \ot \fg_{1,C} \subset \Lambda^k_{k-\sfa}$ be an irreducible module of highest weight $\pi$.  From Lemmas \ref{L:C00} \& \ref{L:knapp} we deduce that 
\begin{equation} \label{E:scalar2}
  \wh\square \hbox{ acts on $U$ by a scalar } 
  \stackrel{(\ast)}{\ge} \ 
  (\c,\b) \ + 
  \underbrace{ \sum_{\a\pm\b\in\Delta} |\a|^2 \ -
    \sum_{\stack{\a+\b\in\Delta}{\a-\b \not\in \Delta}} (\a,\b)
  }_{\mathrm{sums \ over \ } \{\a \in \Delta(w,a) \ | \ k < a \} }
   \ \ge \ 0 \, ,
\end{equation}
and equality holds at $(\ast)$ if and only if $\pi = \c-\b$ and $[\fm^{\sfa-k-1} \,,\, \fg_{-\b}] = \{0\}$.   The latter implies $[\fg_{0,1} \,,\, \fg_{-\b}]=\{0\}$.  Equivalently, $\fg_{-\b} \not\subset [\fg_{0,-1} \,,\, \fg_{0,1-b}]$ where $b = \beta(Z_w) = \sfa-k>1$.  On the other hand $\fg_{-\b}$ is obtained from the highest root space by successive brackets with $\fg_{-\a_j}$, where $\a_j$ is a simple root. Since $-\b$ is a highest $\fg_{0,0}$--weight, this implies $\fg_{-\b} \subset [\fg_{-1,0} \,,\, \fg_{1,-b} ]$; but $\fg_{1,-b} = \{0\}$, yielding a contradiction.  Thus, strict inequality holds at $(\ast)$, and $\wh\square$ acts by a positive scalar.  Thus $\tker\,\wh\square=\{0\}$.
\end{proof}

\subsection{The finish} \label{S:finish}

Assume that condition $\sfH_1$ is satisfied.  By Proposition \ref{P:lambda=0} and Corollary \ref{C:new_coh} there exists $\mu : \cF^0 \to \cH^1_2$ such that $\vartheta_{1,<\sfa}=\mu(\vartheta_{\fn_w^\perp})$.

\begin{definition} \label{D:H2}
We say that \emph{$\sfH_2$ is satisfied} if there exist no irreducible $\fg_{0,0}$--sub-modules $\fg_{1,E} \subset \fg_{1,\sfa-1}$ and $\fg_{1,C} \subset \fg_{1,\sfa}$ with highest weights $\e \in \Delta(\fg_{1,E})$ and $\c\in\Delta(\fg_{1,C})$, respectively, such that the following two conditions hold
$\{0\} \not= \left[ \fg_\e \,,\, \fg_{-\c} \right] = 
  \left[ \fg_\e \,,\, \fg_{-1,-\sfa} \right]$.  

We say that \emph{$\sfH_+$ is satisfied} if both $\sfH_1$ (Definition \ref{D:H1}) and $\sfH_2$ are satisfied.
\end{definition}

\begin{proposition} \label{P:lm=0}
If $\sfH_+$ is satisfied (Definition \ref{D:H2}), then $\vartheta_{1,<\sfa}=0$ on $\cF^0$.
\end{proposition}

Propositions \ref{P:Bw_rigid}, \ref{P:lambda=0} and Proposition \ref{P:lm=0} establish the following.

\begin{theorem} \label{T:Bw_rigid}
If $\sfH_+$ is satisfied (Definition \ref{D:H2}), then the Schubert system $\cB_w$ is rigid.  In particular, $X_w$ is Schubert rigid.
\end{theorem}

\noindent In Section \ref{S:BR} we identify those Schubert varieties $X_w$ for which $\sfH_+$ is satisfied.  The remainder of this section is devoted to the proof, by induction, of Proposition \ref{P:lm=0}.  Let $\mu^k$ denote the component of $\mu$ taking values in $\tM^k$ (cf. Section \ref{S:action}).  We begin with $\tM^{2\sfa-1} = \fg_{1,\sfa-1} \ot \fg_{1,\sfa}$.

\begin{lemma}\label{L:2,2a-1}
Assume $\sfH_1$ is satisfied.  If $\sfH_2$ is satisfied, then the eigenvalues of the Laplacian on $\tM^{2\sfa-1}$ are strictly positive.  In particular, $\cH^1_{2,2\sfa-1} = \{0\}$ and $\mu^{2\sfa-1}$ vanishes.

If $\sfH_2$ fails for a pair $\b,\c$, then $\cH^1_{2,2\sfa-1}$ contains an irreducible $\fg_{0,0}$--module of highest weight $\c+\b$.
\end{lemma} 

\begin{proof}
It suffices to show that the eigenvalues of $\square$ on $M_{2\sfa-1} = \fg_{1,\sfa-1} \ot \fg_{1,\sfa}$ are strictly positive.  The computations involved are similar to those of Section \ref{S:step1}.  Let $\e \in \Delta(\fg_{1,\sfa-1})$ be a highest weight of an irreducible $\fg_{1,E} \subset \fg_{1,\sfa-1}$, and let $\gamma \in \Delta(\fg_{1,\sfa})$ be a highest weight of an irreducible $\fg_{1,C} \subset \fg_{1,\sfa}$.  Fix an irreducible $\fg_{0,0}$--module $U \subset \fg_{1,E} \ot \fg_{1,C}$ of highest weight $\pi$.  From \eqref{E:sqL2}, Lemmas \ref{L:C00} \& \ref{L:knapp}(c) we deduce
\begin{equation*} 
  \square \hbox{ acts on $U$ by a scalar } \ge \ 
  -(\c,\e) \ +  \sum_{\stack{\a \in \Delta(w,\sfa)}{\a-\e\in\Delta}} (\a,\e)
   \ \ge \ 0 \, .
\end{equation*}
This sum may vanish under two conditions.  Both require $\pi = \e+\c$.  First, it may be the case that $[\fg_\e \,,\, \fg_{-1,-\sfa}] = \{0\} \subset \fg_{0,-1}$.  This forces $\fg_\e \not\in [ \fg_{0,-1} \,,\, \fg_{1,\sfa}]$, which contradicts Lemma \ref{L:nondeg}.  Second, it may be the case that $[\fg_\e \,,\, \fg_{-1,-\sfa}] = [\fg_\e \,,\, \fg_{-\c} ] \not= \{0\}$.  This is precisely the case that the condition $\sfH_2$ fails.  The lemma is established.
\end{proof}

Assume that $\sfH_+$ is satisfied, and that there exists $0 \le k \le 2\sfa-2$ such that $\mu^{>k}$ vanishes on $\cF^0$.  Lemma \ref{L:2,2a-1} assures us that the inductive hypothesis holds for $k = 2\sfa-2$.  We will show that $\mu^k$ also vanishes $\cF^0$.  Let $\mu^k_c$ denote the component of $\mu$ taking value in $\tM^k_c$, see \eqref{E:Muk} and recall that $\mu^k_c = 0$ if either $c < 0$ or $k<c$.  

\begin{lemma} \label{L:muk=0} \quad
\begin{a_list}
\item Given $c>k-\sfa$ and $\z\in\fg_{0,\sfa-k+c}$, we have  $\cL''_\z(\mu^k_c) = -\cL'_\z(\mu^k_{k-\sfa})$.
\item
If $\mu^k_{k-\sfa} = 0$, then 
$\mu^k = (\mu^k_{k-\sfa},\ldots,\mu^k_{\sfa-1}) = 0$.
\item %
The component $\mu^k_{k-\sfa}$ vanishes.
\end{a_list}
\end{lemma}

\noindent Lemma \ref{L:muk=0} yields $\mu^{\ge k}=0$, completing the induction and establishing Proposition \ref{P:lm=0}.  

\begin{proof}[Proof of Lemma \ref{L:muk=0}(a)]
Let $\mu_c$ denote the component of $\mu$ taking values in $\fg_{1,c}\ot\fn_w^*$.  The vanishing of $\mu^{>k}$ is equivalent to $\mu_c(u) = 0$ for all $u \in \fg_{-1,-a}$ with $c+a > k$.  An application of the Maurer-Cartan equation \eqref{E:mce} to $\vartheta_{1,c} = \mu_c(\vartheta_{\fn_w})$ yields
\begin{equation} \label{E:dmu}
  \half \, \td \mu_c \wedge \vartheta_{\fn_w} \ = \ 
  \mu_c\left( [ \vartheta_{0,\ge0} \,,\, \vartheta_{\fn_w} ] \right) \ - \ 
  \left[ \vartheta_{0,\ge0} \,,\, \vartheta_{1,<\sfa} \right]_{1,c} \, .
\end{equation}
In particular, if $\z\in\fg_{0,m}\subset \fg_{0,>0}$ and $u \in \fg_{-1,-a}\subset\fn_w$ with $a+c > k$, then evaluating \eqref{E:dmu} on $\hat\z\wedge\hat u$ yields
$0 = \mu_c( [ \z,u] ) - [ \z , \mu_{c-m}(u) ]$.  
Set $a = \sfa$ and $m = \sfa-k+c$.  Under the identification $\fn_w^* \simeq \fn_w^+$, this yields $\cL'_\z(\mu^k_c) = -\cL''_\z(\mu^k_{c+p})$.
\end{proof}

\begin{proof}[Proof of Lemma \ref{L:muk=0}(b)]
The proof, which is identical to that of Lemma \ref{L:lambdak=0}(b), is left to the reader. 
\end{proof} 

\begin{proof}[Proof of Lemma \ref{L:muk=0}(c)]
From \eqref{E:sqL2}, the component of $\square_{|\tM^k}$ taking value in $\tM^k_{k-\sfa}$ is
\begin{eqnarray} \nonumber
  \square(\mu^k_{k-\sfa} , \ldots , \mu^k_{\sfa-1})_{\tM^k_{k-\sfa}}
   & = & \sum_{\stack{\a\in\Delta(w,a)}{k-\sfa<a}}
  \cL'_{E_\a}\circ\cL'_{E_{-\a}} (\mu^k_{k-\sfa}) \ - 
  \sum_{\stack{k-\sfa\le c<\sfa}{\z_j\in\fg_{\sfa-k+c}}}
  \cL''_{\z_j}\circ\cL'_{\xi_j} (\mu^k_c) \\
  \label{E:sq_mu}   & \stackrel{(\dagger)}{=} & 
  \sum_{\stack{\a\in\Delta(w,a)}{k-\sfa<a}}
  \cL'_{E_\a}\circ\cL'_{E_{-\a}} (\mu^k_{k-\sfa}) \ + 
  \sum_{\z_j\in\fm^{2\sfa-1-k}} 
  \cL'_{\x_j}\circ\cL'_{\z_j} (\mu^k_{k-\sfa}) \\
  \nonumber & &  
  - \  \sum_{\z_j\in\fg_{0,0}} 
  \cL''_{\z_j}\circ\cL'_{\xi_j} (\mu^k_{k-\sfa})  \, ;
\end{eqnarray}
the equality $(\dagger)$ is a consequence of Lemma \ref{L:muk=0}(a).

Define $\wt\square : \tM^k_{k-\sfa} \to \tM^k_{k-\sfa}$ by
$$
  \wt\square \ := \ 
  \sum_{\stack{\a\in\Delta(w,a)}{k-\sfa<a}} \cL'_{E_\a}\circ\cL'_{E_{-\a}} \ + \ 
  \sum_{\z_j\in\fm^{2\sfa-1-k}} \cL'_{\z_j}\circ\cL'_{\xi_j} \ - \ 
  \sum_{\z_j\in\fg_{0,0}} \cL''_{\z_j}\circ\cL'_{\xi_j} \, .
$$
Then \eqref{E:sq_mu} reads $\square(\mu^k_{k-\sfa} , \ldots , \mu^k_{\sfa-1})_{\tM^k_{k-\sfa}} = \wt\square(\mu^k_{k-\sfa})$.  Since $\mu^k$ is $\cH^1_{2,k}$--valued, the left-hand side of this equation vanishes.  To complete the proof, it suffices to show that $\tker\,\wt\square = \{0\}$. 

Note that $\wt\square$ is a $\fg_{0,0}$--module map.  Therefore $\tM^k_{k-\sfa}$ admits a $\fg_{0,0}$--module decomposition into $\wt\square$--eigenspaces.  To see that $\tker\,\wt\square$ is trivial, let $\b \in \Delta(\fg_{1,k-\sfa})$ be a highest weight of an irreducible $\fg_{1,B} \subset \fg_{1,k-\sfa}$, and let $\gamma \in \Delta(\fg_{1,\sfa})$ be a highest weight of an irreducible $\fg_{1,C} \subset \fg_{1,\sfa}$.  Fix an irreducible $\fg_{0,0}$--module $U \subset \fg_{1,B} \ot \fg_{1,C}$ of highest weight $\pi$.  From \eqref{E:sqL2}, Lemmas \ref{L:C00} and \ref{L:knapp}.c we deduce
\begin{equation*} 
  \wt\square \hbox{ acts on $U$ by a scalar } \stackrel{(\ast)}{\ge} \ 
  -(\c,\b) \ +  \sum_{\stackrel{\a \in \Delta(w,a)}{k-\sfa<a}} (\a,\b)
   \ \ge \ 0 \, ,
\end{equation*}
and equality holds in $(\ast)$ if and only if $[\fm^{2\sfa-1-k} \,,\, \fg_{\b}] = \{0\}$.
Since $\fg_{0,1} \subset \fm^{2\sfa-1-k}$, equality at $(\ast)$ implies $[ \fg_{0,1} \,,\, \fg_{\b}] = \{0\}$.  This contradicts Lemma \ref{L:nondeg}.  Thus, strict inequality holds at $(\ast)$, and $\wt\square$ acts on $U$ by a positive scalar.  
\end{proof}

\section{Schubert rigidity}\label{S:BR}

By Theorem \ref{T:Bw_rigid}, the Schubert system $\cB_w$ is rigid when $\sfH_+$ is satisfied (Definition \ref{D:H2}).  In this section we will prove Theorem \ref{T:BR}.  Throughout the section 
$$
  \ttJ \ = \ \{\ttj_1 , \ldots , \ttj_\sfp\}
$$
is the index set of Section \ref{S:CHSS_Dw}, ordered so that $\ttj_1 <\cdots < \ttj_\sfp$; for convenience we set 
$$
  \ttj_0 \ := \ 0 \quad \hbox{ and } \quad \ttj_{\sfp+1} \ := \ 1 + \tmax\{1,\ldots \hat \tti \ldots , n \} \,.
$$
We assume that $X_w$ is proper, so that $\ttJ \not= \emptyset$ (Remark \ref{R:Jw}).  The admissible pairs $(\ttJ,\sfa)$ are listed in Corollary \ref{C:bigone}.

\begin{theorem}\label{T:BR}
The Schubert varieties satisfying $\sfH_+$ are listed below. \smallskip\\
{\bf (a)} In $A_n / P_\tti$ with $1 < \tti< n$.  Define $\sfq\in\bZ$ by $\ttj_\sfq<\tti<\ttj_{\sfq+1}$.
\begin{i_list}
\item[$\spadesuit:$] $(\sfp,\sfq)=(2\sfa+2,\sfa+1)$ with $1 < \ttj_{\ell} - \ttj_{\ell-1}$, for all $1 < \ell \le \sfp$, and $1 < \tti - \ttj_\sfq,\,\ttj_{\sfq+1} -\tti$;
\item[$\heartsuit:$] $(\sfp,\sfq)=(2\sfa+1,\sfa+1)$ with $1 < \ttj_{\ell} - \ttj_{\ell-1}$, for all $1 < \ell \le \sfp+1$, and $1 < \ttj_{\sfq+1} -\tti$;
\item[$\diamondsuit:$] $(\sfp,\sfq)=(2\sfa+1,\sfa)$ with $1 < \ttj_{\ell} - \ttj_{\ell-1}$, for all $1 \le \ell \le \sfp$, and $1 < \tti -\ttj_\sfq$;
\item[$\clubsuit:$] $(\sfp,\sfq)=(2\sfa,\sfa)$ with $1 < \ttj_{\ell} - \ttj_{\ell-1}$ for all $1 \le \ell \le \sfp+1$.
\end{i_list}
\noindent {\bf (b)} In $D_n/P_1$: the $X_w$ with $\sfa=0$ and $\ttJ = \{ n-1 \}$ or $\ttJ = \{n\}$.
\smallskip\\
\noindent{\bf (c)} In $C_n/P_n$: 
\begin{i_list}
  \item when $\sfp=\sfa$, any $\ttJ$ with $1< \ttj_\ell - \ttj_{\ell-1}$ for all $1 \le \ell \le \sfp$;
  \item when $\sfp=\sfa+1$, any $\ttJ$ with $1< \ttj_\ell - \ttj_{\ell-1}$ for all $2 \le \ell \le \sfp+1$.
\end{i_list}
\noindent{\bf (d)} In $D_n/P_n$, set $\sfs_0 = \lceil \frac{\sfp+1}{2} \rceil$ and $\sfs_1 = \lceil \frac{\sfp}{2} \rceil$:
\begin{i_list}
\item Either $\sfp=\sfa$ and $n-1\not\in\ttJ$, or $\sfp = \sfa+1$ and $n-1 \in \ttJ$; in both cases $1 < \ttj_\ell-\ttj_{\ell-1}$ for all $1 \le \ell\le \sfp$ and $2 < \ttj_\sfs-\ttj_{\sfs-1}$ with $\sfs = \lceil \frac{\sfp+1}{2} \rceil$.
\item $\sfp=\sfa+1$ and $n-1 \not\in \ttJ$ with $1<\ttj_\ell - \ttj_{\ell-1}$ for all $2 \le \ell \le \sfp$, and $2<\ttj_{\sfs+1} - \ttj_{\sfs}$ with $\sfs=\lceil \frac{\sfp}{2} \rceil$.
\end{i_list}
\noindent{\bf (e)} In $E_6/P_1 \simeq E_6/P_6$ and $E_7/P_7$: see Tables \ref{t:E6} \& \ref{t:E7}.\footnote{
The first column of Tables \ref{t:E6} \& \ref{t:E7} expresses $w$ as a composition of reflections associated to simple roots (acting on the left): for example, $(6542)$ denotes $\sigma_6\sigma_5\sigma_4\sigma_2 \in W^\fp$, where $\sigma_i \in W$  is the reflection associated to the simple root $\a_i$.  The values $\sfa^*$ and $\ttJ^*$ are given Lemma \ref{L:PD} and Proposition \ref{P:pd}, respectively.}
\end{theorem}

\begin{table}[h] 
\renewcommand{\arraystretch}{1.2}
\caption{Proper $X_w \subset E_6/P_6$ satisfying $\sfH_+$.}
\begin{tabular}{|r|c|c|c|c|c|}
\hline
$w$ & $\sfa$ & $\ttJ$ & $\tdim\,X_w$ & $\sfa^*$ & $\ttJ^*$ \\
\hline
$(6542)$ & $0$ & $\{3\}$ & 4 & $1$ & $\{3\}$ \\
$(65431)$ & $0$ & $\{2\}$ & $5$ & $1$ & $\{5\}$ \\
$(65432413)$ & $1$ & $\{4\}$ & $8$ & $1$ & $\{4\}$ \\
$(65432456)$ & $0$ & $\{1\}$ & $8$ & $0$ & $\{1\}$ \\
$(65432451342)$ & $1$ & $\{5\}$ & $11$ & $0$ & $\{2\}$ \\
$(654324561345)$ & $1$ & $\{3\}$ & $12$ & $0$ & $\{3\}$ \\
\hline 
\end{tabular} \label{t:E6}
\end{table}

\begin{table}[h]
\renewcommand{\arraystretch}{1.2}
\caption{Proper $X_w \subset E_7/P_7$ satisfying $\sfH_+$.}
\begin{tabular}{|r|c|c|c|c|c|}
\hline
$w$ & $\sfa$ & $\ttJ$ & $\tdim\,X_w$ & $\sfa^*$ & $\ttJ^*$ \\
\hline 
$(76542)$ & $0$ & $\{3\}$ & $5$ & $2$ & $\{5\}$ \\
$(765431)$ & $0$ & $\{2\}$ & $6$ & $1$ & $\{2\}$ \\
$(765432413)$ & $1$ & $\{4\}$ & $9$ & $2$ & $\{4\}$ \\
$(7654324567)$ & $0$ & $\{1\}$ & $10$ & $1$ & $\{6\}$ \\
$(765432451342)$ & $1$ & $\{5\}$ & $12$ & $1$ & $\{3\}$ \\
$(7654324561345)$ & $2$ & $\{3,6\}$ & $13$ & $2$ & $\{1,5\}$ \\
$(76543245671342)$ & $2$ & $\{1,5\}$ & $14$ & $2$ & $\{3,6\}$ \\
$(765432456713456)$ & $1$ & $\{3\}$ & $15$ & $1$ & $\{5\}$ \\
$(76543245613452431)$ & $1$ & $\{6\}$ & $17$ & $0$ & $\{1\}$ \\
$(765432456713456245)$ & $2$ & $\{4\}$ & $18$ & $1$ & $\{4\}$ \\
$(765432456713456245342)$ & $1$ & $\{2\}$ & $21$ & $0$ & $\{2\}$ \\
$(7654324567134562453413)$ & $2$ & $\{5\}$ & $22$ & $0$ & $\{3\}$ \\
\hline
\end{tabular} \label{t:E7}
\end{table}

\begin{remarks*} 
By Theorem \ref{T:Bw_rigid} the varieties listed in Theorem \ref{T:BR} are Schubert rigid.  They are the Schubert varieties for which there exist first-order obstructions to the existence of nontrivial integral varieties of the Schubert system.  In particular, those varieties not listed need not be flexible: there may exist higher-order obstructions.

By Proposition \ref{P:conj}, a Schubert variety is Schubert rigid if and only if its conjugate is Schubert rigid.  Making use of \eqref{E:a'J'}, we see that the list of Schubert varieties satisfying $\sfH_+$ is also closed under conjugation.
\end{remarks*}

By Corollary \ref{C:pd}, $[X_{w^*}]$ is the Poincar\'e dual of $[X_w]$.

\begin{corollary} \label{C:pdH+}
A Schubert variety satisfies $\sfH_+$ if and only if its Poincar\'e dual does.
\end{corollary}

\begin{proof}
It is clear from Tables \ref{t:E6} \& \ref{t:E7} that the corollary holds for the exceptional CHSS.  To prove the corollary for the classical CHSS we apply Proposition \ref{P:pd} and Remark \ref{R:a*} to rephrase Theorem \ref{T:BR} in Table \ref{t:H+}.  

We say that a subset $\ttS$ of the Dynkin diagram $\d_\fg$ is \emph{orthogonal} if distinct pairs $\ttj,\ttk \in \ttS$ are not connected.  Equivalently, the corresponding simple roots $\a_\ttj$ and $\a_\ttk$ are orthogonal with respect to the Killing form.  Additionally, define $\ttS_{<\tti} = \{ \ttj \in \ttS \ | \ \ttj < \tti\}$.  The sets $\ttS_{\le\tti}$, $\ttS_{>\tti}$ and $\ttS_{\ge\tti}$ are defined analogously.  Let 
$$
  \ttD \ := \ \ttJ \, \cup \, \{\tti\} \quad \hbox{ and } \quad 
  \ttD^* \ := \ \ttJ^* \, \cup \, \{\tti\}
$$
The corollary follows from inspection of Table \ref{t:H+}, which lists the proper Schubert varieties in the classical CHSS that satisfy $\sfH_+$.
\end{proof}

\begin{table}[h] \renewcommand{\arraystretch}{1.2}
\caption{The proper Schubert varieties in the classical CHSS satisfying $\sfH_+$}
\begin{tabular}{|c|ll|}
\hline
  & $\spadesuit$ & $\ttD$ orthogonal \\
$A_n/P_\tti$  
  & $\heartsuit$ & $\ttD_{<\tti}$, $\ttD_{\ge\tti}$, $\ttD_{\ge\tti}^*$ 
   orthogonal  \\
\footnotesize{$1 < \tti < n$}  
  & $\diamondsuit$ & $\ttD_{>\tti}$, $\ttD_{\le\tti}$, $\ttD_{\le\tti}^*$ 
   orthogonal  \\
  & $\clubsuit$ & $\ttD^*$  orthogonal  \\ \hline
\end{tabular}\smallskip

\begin{tabular}{|c|l|}
\hline
$D_n/P_1$ & $\sfa=0$ and $\ttJ = \{n-1\}$ or $\{n\}$\\
\hline
\end{tabular} \smallskip

\begin{tabular}{|c|ll|}
\hline
 & $\ttD^*$ orthogonal, & $\sfp=\sfa$ \\
\rb{$C_n/P_n$} & $\ttD$ orthogonal, & $\sfp=\sfa+1$ \\  \hline
\end{tabular} \smallskip

\begin{tabular}{|c|l|l|l|}
\hline
 & & $\sfp=\sfa$ {\small{\&}} $n-1\not\in\ttJ$ & \\
$D_n/P_n$ & $\ttJ$ orthogonal
 & $\sfp=\sfa+1$ {\small{\&}} $n-1 \in\ttJ$ 
 & \rb{$1 \not\in\ttJ$ {\small{\&}} $2 < \ttj_{\sfs}-\ttj_{\sfs-1}$,  
       $\sfs = \lceil \frac{\sfp+1}{2}\rceil$} \\ \cline{3-4}
 & & \multicolumn{2}{|l|}{$\sfp=\sfa+1$ {\small{\&}} $n-1\not\in\ttJ$ 
       with $2 < \ttj_{\sfs+1}-\ttj_\sfs$,  
       $\sfs = \lceil \frac{\sfp}{2}\rceil$} \\ \hline
\end{tabular}
\label{t:H+}
\end{table}


\begin{remark} \label{R:BRpart}
Applying Proposition \ref{P:part-aJ} and Table \ref{t:suit} to Theorem \ref{T:BR}(a) yields a partition description of the proper $X_\pi \subset \tGr(\tti,n+1)$ satisfying $\sfH_+$.  They are precisely those partitions $\pi = (p_1{}^{q_1},\ldots,p_r{}^{q_r})$ satisfying 
$$
  1 < q_\ell \,,\ q'_\ell \,, \quad \hbox{for all } 2 \le \ell \le r \,,
  \quad\hbox{and}
$$
\hbox{\hspace{20pt}}
\begin{tabular}{ll}
$1<q'_1$ when $\pi \in \heartsuit$;$\quad$ &
$1 < q_1\,,\ q_1'$ when $\pi \in \clubsuit$;\\
$1 < q_1$ when $\pi \in \diamondsuit$; &
no additional constraints when $\pi \in \spadesuit$.
\end{tabular}
\end{remark}

\begin{proof}[Outline of proof of Theorem \ref{T:BR}]
Sections \ref{S:B1} \& \ref{S:D1} identify the Schubert varieties of the quadric hypersurface that satisfy $\sfH_+$ yielding Theorem \ref{T:BR}(b).  Parts (c) \& (d) are proved in Sections \ref{S:Cn} \& \ref{S:Dn}, respectively.  The reader who works though these arguments will be able to prove (a) as an exercise.  We used LiE \cite{LiE} to verify Theorem \ref{T:BR}(e).
\end{proof}

Throughout the remainder of Section \ref{S:BR} the notation $\gamma$ will denote a highest root of $\fg_{1,C} \subset \fg_{1,\sfa}$; $-\b = -\a_{\ttj}$ will denote a highest root of $\fg_{0,B} \subset \fg_{0,-1}$, $\ttj \in \ttJ$; and $\e$ will denote a highest root of $\fg_{1,E} \subset \fg_{1,\sfa-1}$.  

\subsection{Quadric hypersurface \boldmath $Q^{2n-1} = B_n/P_1$ \unboldmath }\label{S:B1}

By Corollary \ref{C:bigone}, we have $\ttJ = \{\ttj\}$ and $\sfa \in \{ 0,1 \}$.  We begin with the case that $\sfa = 1$.  In this case $\fg_{1,1} = \fg_{1,C}$ is an irreducible $\fg_{0,0}$--module with $C = (1)$; cf. the proof of Lemma \ref{L:bigone}.  We have
\begin{eqnarray*}
  \c & = & \a_1+\cdots+\a_{\ttj}+2(\a_{\ttj+1} + \cdots + \a_n) \, ,\\
  \Delta(\fg_{1,1}) & = & \{ \a_1+\cdots+\a_k \,,\ 
  \a_1+\cdots+\a_k + 2(\a_{k+1} + \cdots +\a_n) \ | \ k\ge\ttj \} \, .
\end{eqnarray*}
Condition $\sfH_1$ fails for $\b = \a_{\ttj}$.

Next assume that $\sfa=0$.  Then $\fg_{1,0} = \fg_{1,C}$ is again irreducible with $C = (0)$, and 
\begin{equation} \label{E:quad0}
  \c \ = \ \a_1 + \cdots + \a_{\ttj-1} \quad \hbox{ and } \quad
  \Delta(\fg_{1,0}) \ = \
  \{ \a_1 + \cdots + \a_k \ | \ k < \ttj \} \, .
\end{equation}
The condition $\sfH_1$ fails for $\b = \a_{\ttj}$.

We conclude that none of the proper Schubert varieties $X_w \subset Q$ satisfy $\sfH_+$.

\subsection{Quadric hypersurface \boldmath $Q^{2n-2} = D_n/P_1$ \unboldmath }\label{S:D1}

There are four cases to consider in Corollary \ref{C:bigone}.  {\bf (1)} Begin with $\sfa=0$ and $\ttJ = \{\ttj\}$.  Suppose that $\ttj<n-1$.  Then $\c$ and $\Delta(\fg_{1,0})$ are given by \eqref{E:quad0}.  The condition $\sfH_1$ fails for $\b = \a_\ttj$.  

Now suppose $\ttj \in \{ n-1,n\}$.  The two cases are symmetric; so without loss of generality, $\ttj = n$.  Then
$$
  \c \ = \ \a_1 + \cdots + \a_{n-1} \,,\quad
  \Delta(\fg_{1,0}) \ = \ \{ \a_1 + \cdots + \a_j \ | \ j < n \} \, .
$$
Condition $\sfH_1$ is satisfied.  Condition $\sfH_2$ is vacuous as $\sfa=0$.

{\bf (2)} In the case $\sfa=0$ and $\ttJ = \{n-1,n\}$, $\gamma$ and $\Delta(\fg_{1,0})$ are given by \eqref{E:quad0} with $\ttj$ replaced by $n-1$.  The condition $\sfH_1$ fails for both $\b = \a_{n-1},\a_n$.  

{\bf (3)} Next suppose that $\sfa=1$ and $\ttJ = \{ \ttj < n-1 \}$.  Then
\begin{eqnarray*}
  \c & = & \a_1 + \cdots + \a_\ttj + 2(\a_{\ttj+1} 
  + \cdots + \a_{n-2}) + \a_{n-1} + \a_n \, \\
  \Delta(\fg_{1,1}) & = & \{ \a_1+\cdots+\a_{n-2}+\a_n \,, \ 
  \a_1 + \cdots + \a_k \,, \\
  & & \hbox{\hspace{5pt}}
  \a_1 + \cdots + \a_k + 2 (\a_{k+1} + \cdots + \a_{n-1}) 
  + \a_{n-1} + \a_n \ | \ k \ge \ttj \} .
\end{eqnarray*}
The condition $\sfH_1$ fails for $\b = \a_\ttj$.

{\bf (4)} If $\sfa = 1$ and $\ttJ = \{n-1,n\}$, then $\fg_{1,1} = \fg_{1,(1,0)} \op \fg_{1,(0,1)} = \fg_{\c_1} \op \fg_{\c_2}$ where 
$$
  \c_1 \ = \ \a_1 + \cdots + \a_{n-1} \quad \hbox{ and } \quad 
  \c_2 \ = \ \a_1 + \cdots + \a_{n-2} + \a_n \,.
$$
The condition $\sfH_1$ fails for the pairs $(\c_1,\b=\a_n)$ and $(\c_2,\b=\a_{n-1})$.

\subsection{Lagrangian grassmannian \boldmath $C_n/P_n$ \unboldmath }\label{S:Cn}

By Corollary \ref{C:bigone}, $\sfa \le |\ttJ| \le \sfa+1$.

\begin{proposition*} 
The proper Schubert varieties $X_w \subset C_n/P_n$ satisfying $\sfH_+$ are
\begin{i_list}
\item $|\ttJ|=\sfa$ with $1 < \ttj_1$ and $1< \ttj_\ell - \ttj_{\ell-1}$, for all $1 <  \ell \le \sfa$; and
\item $|\ttJ|=\sfa+1$ with $1<\ttj_\ell - \ttj_{\ell-1}$, for all $1 < \ell \le \sfa+1$, and $\ttj_{\sfa+1} < n-1$ \,.
\end{i_list}
\end{proposition*}

\begin{proof}[Proof of (i)]
The proposition is proved in two claims.  We first warm-up with a simple case, and then prove the general case.

\begin{claim*} 
The Schubert variety $X_w$ associated to $\ttJ = \{ \ttj \}$ with $\sfa=1$ satisfies $\sfH_+$ if and only if $1<\ttj$.
\end{claim*}

\noindent The $\fg_{0,0}$--module $\fg_{1,\sfa} = \fg_{1,C}$ is irreducible with $C = (1)$, highest weight $\c = \a_1 + \cdots + \a_{\ttj} + 2 (\a_{\ttj+1} + \cdots + \a_{n-1}) + \a_n$ and 
\begin{eqnarray*}
  \Delta(\fg_{1,\sfa}) & = & \{ 
  \a_k + \cdots + \a_n \,,\ 
  \a_k + \cdots + \a_{\ell-1} + 2(\a_\ell + \cdots + \a_{n-1} ) + \a_n 
  \ | \ k \le \ttj < \ell \} \, .
\end{eqnarray*}
Condition $\sfH_1$ is always satisfied.  The $\fg_{0,0}$--module $\fg_{1,\sfa-1} = \fg_{1,0} = \fg_{1,E}$ is irreducible with $E = (0)$ and highest weight $\e = 2 (\a_{\ttj+1} + \cdots + \a_{n-1}) + \a_n$.  Condition $\sfH_2$ is satisfied if and only if $\ttj > 1$.  This establishes the claim.

\begin{claim*} 
The Schubert subvarieties $X_w \subset C_n/P_n$ with $\sfp=\sfa$ satisfy $\sfH_+$ if and only if $1 < \ttj_\ell - \ttj_{\ell-1}$ for all $2 \le \ell \le \sfa$, and $1 < \ttj_1$.
\end{claim*}

\noindent We have 
$$
  \fg_{1,\sfa} \ = \ \bigoplus_{q=0}^{\lfloor {\frac{\sfa}{2}} \rfloor} 
  \fg_{1,C_q} \quad \hbox{where} \quad
  C_q \ = \ (0^q , 1^{\sfa-2q} , 2^q) \, .
$$
Let $\c_q$ denote the highest weight of $\fg_{1,C_q}$.  Then
\begin{eqnarray*}
  \c_q & = & \a_{\ttj_q+1} + \cdots + \a_{\ttj_{\sfa-q}} + 
  2 ( \a_{\ttj_{\sfa-q}+1} + \cdots + \a_{n-1} ) + \a_n \\
  \Delta(\fg_{1,C_q}) & = & \{ \a_j + \cdots + \a_{k-1} + 
  2 ( \a_k + \cdots + \a_{n-1} ) + \a_n \ | \ \ttj_q < j \le \ttj_{q+1} \,,\\
  & & \hbox{\hspace{218pt}} \ttj_{\sfa-q} < k \le \ttj_{\sfa-q+1} \} \, .
\end{eqnarray*}
Excluding the case that $\sfa = 2\sfr+1$ is odd and $q=\sfr$, the conditions $[\fg_{-\b} \,,\, \fg_{\c_q}] = \{0\} \not= [\fg_\b \,,\, \fg_{\c_q}]$ will hold if and only if $\b = \a_{\ttj_q}$ or $\b = \a_{\ttj_{\sfa-q}}$.  Observe that $[\fg_{\a_{\ttj_q}} \,,\, \fg_{1,\sfa}] = [\fg_{\a_{\ttj_q}} \,,\, \fg_{1,C_q}]$ and $[\fg_{\a_{\ttj_{\sfa-q}}} \,,\, \fg_{1,\sfa}] = [\fg_{\a_{\ttj_{\sfa-q}}} \,,\, \fg_{1,C_q}]$.  We have $[\fg_{\a_{\ttj_q}} \,,\, \fg_{1,C_q}] = [\fg_{\a_{\ttj_q}} \,,\, \fg_{\c_q} ]$ if and only if $\ttj_{\sfa-q} + 1 = \ttj_{\sfa-q+1}$.  Similarly, $[\fg_{\a_{\ttj_{\sfa-q}}} \,,\, \fg_{1,C_q}] = [\fg_{\a_{\ttj_{\sfa-q}}} \,,\, \fg_{\c_q}]$ if and only if $\ttj_q +1 = \ttj_{q+1}$.  In summary, $\sfH_1$ is satisfied if and only if $1< \ttj_1$ and $1 < \ttj_{q+1} - \ttj_q$ (with the caveat that the last inequality is not required for $\ell=\sfr$ when $\sfa = 2\sfr+1$.)

To address the condition $\sfH_2$ note that 
$$
  \fg_{1,\sfa-1} \ = \ \bigoplus_{r=0}^{\lfloor {\frac{\sfa-1}{2}} \rfloor} 
  \fg_{1,E_r} \quad \hbox{where} \quad
  E_r \ = \ (0^{r+1} , 1^{\sfa-1-2r} , 2^r ) \, .
$$
The highest weight of $\fg_{1,E_r}$ is 
$$
  \e_r \ = \ \a_{\ttj_{r+1}+1} + \cdots + \a_{\ttj_{\sfa-r}} + 
  2 ( \a_{\ttj_{\sfa-r}+1} + \cdots + \a_{n-1} ) + \a_n \, .
$$
We have $[\fg_{-\e_r} \,,\, \fg_{\c_q}] \not=\{0\}$ if only if $r=q$ or $r=q-1$.  In particular, $[\fg_{-\e_r} \,,\, \fg_{1,\sfa}] \supset \fg_{\c_r-\e_r} \op \fg_{\c_{r+1}-\e_r}$.  Both summands are nonzero as long as $\sfa-1 \not=2r$.  (In this case there is no $C_{r+1}$.)   So $\sfH_2$ may only fail for $\sfa=2r+1$ and pairs of the form $(\e_r,\c_r)$.  The condition $\sfH_2$ is satisfied when $1 < \ttj_{r+1} - \ttj_r$. This completes the proof of the claim.
\end{proof}

\begin{proof}[Proof of (ii)]
The proof is given in a series of three claims.  The first two are warm-ups, the third establishes the general result.

\begin{claim*}
The Schubert variety associated to $\sfa=0$ and $\ttJ = \{ \ttj \}$ satisfies $\sfH_+$ if and only if $\ttj < n-1$.
\end{claim*}

\noindent We have $\fg_{1,\sfa} = \fg_{1,C}$ where $C = (0)$.  The highest weight of $\fg_{1,C}$ is the root $\c = 2(\a_{\ttj_1+1} + \cdots + \a_{n-1})+\a_n$, and 
\begin{eqnarray*}
  \Delta(\fg_{1,\sfa}) & = & \{ 
  \a_k + \cdots + \a_n \,,\ 
  \a_k + \cdots + \a_{\ell-1} + 2(\a_\ell + \cdots + \a_{n-1} ) + \a_n \,,\\
  & & \hbox{\hspace{166pt}}
  2(\a_\ell + \cdots + \a_{n-1}) + \a_n \ | \ \ttj < k , \ell \} \, .
\end{eqnarray*}
Condition $\sfH_1$ will fail with $\b = \a_{\ttj}$ only if $\ttj = n-1$.  Since $\fg_{1,\sfa-1} = \{0\}$ condition $\sfH_2$ is vacuous.  The claim follows.

\begin{claim*} 
The Schubert variety $X_w$ associated to $\ttJ = \{\ttj_1 \,,\ \ttj_2 \}$ with $\sfa=1$ satisfies $\sfH_+$ if and only if $1 < \ttj_2 - \ttj_1$ and $\ttj_2 < n-1$.
\end{claim*}

\noindent The module $\fg_{1,\sfa} = \fg_{1,C}$ is irreducible with $C = (0,1)$ and 
\begin{eqnarray*}
  \c & = &  \a_{\ttj_1+1} + \cdots + \a_{\ttj_2} + 
  2 ( \a_{\ttj_2+1} + \cdots + \a_{n-1} ) + \a_n \,, \\
  \Delta(\fg_{1,\sfa}) & = & \{ 
  \a_k + \cdots + \a_n \,,\\
  & & \hbox{\hspace{6pt}} 
  \a_k + \cdots + \a_{\ell-1} + 2(\a_\ell + \cdots + \a_{n-1} ) + \a_n 
  \ | \ \ttj_1 < k \le \ttj_2 < \ell \} \, .
\end{eqnarray*}
If $\b = \a_{\ttj_1}$, then $\sfH_1$ will fail if and only if $\ttj_2 = n-1$.  The module $\fg_{1,\sfa-1} = \fg_{1,E}$ is irreducible with $E = (0,0)$ and highest weight $\e = 2(\a_{\ttj_2+1} + \cdots + \a_{n-1}) + \a_n$.  Condition $\sfH_2$ will be satisfied if and only if $\ttj_2 - \ttj_1 > 1$.  This proves the claim.

\begin{claim*} 
The Schubert subvarieties $X_w \subset C_n/P_n$ with $\sfp=\sfa+1$ satisfy $\sfH_+$ if and only if satisfying $1 < \ttj_\ell - \ttj_{\ell-1}$ for all $2 \le \ell \le \sfa+1$, and $\ttj_{\sfa+1} < n-1$.
\end{claim*}

\noindent The proof of the claim, which is very similar to the proof of general case of (i), is left to the reader.
\end{proof}

\subsection{Spinor variety \boldmath $D_n/P_n$ \unboldmath }\label{S:Dn}

\begin{proposition*} 
\begin{a_list}
\item
If $\sfa=0$, then $\sfH_+$ is satisfied if and only if $\ttJ \not= \{ n-2 \}$.
\item
If $0 < \sfa = 2\sfr-1,\,2\sfr$ and $n-1 \not\in \ttJ$, then $\sfH_+$ is satisfied if and only if $\ttj_\ell - \ttj_{\ell-1} > 1$, for all $\sfp-\sfa < \ell \le \sfp$, and $\ttj_{\sfp-\sfr+1} - \ttj_{\sfp-\sfr} > 2$.
\item
If $0 < \sfa = 2\sfr,\,2\sfr+1$ and $n-1 \in \ttJ$, then $\sfH_+$ is satisfied if and only if $|\ttJ| = \sfa+1$ and $\ttj_\ell - \ttj_{\ell-1} > 1$, for all $1 \le \ell \le \sfp$, and $\ttj_{\sfp-\sfr} - \ttj_{\sfp-\sfr-1} > 2$.
\end{a_list}
\end{proposition*}

\begin{proof}
Review Corollary \ref{C:bigone} and Remark \ref{R:Dn}.  The analysis proceeds as in Section \ref{S:Cn}.  We leave it to the reader to verify the following.

{\bf (A)}  Assume that $n-1 \not\in \ttJ$.  Then $\sfa \le \sfp \le \sfa+1$.  

{\bf (A.1)}  Suppose that $\sfa=0$.  Then $\ttJ = \{\ttj\}$ and $\sfH_1$ is satisfied if and only if $\ttj \not= n-2$.  The condition $\sfH_2$ is vacuous.

{\bf (A.2)}  Suppose that 
$\sfa = 2\sfr-1$ or $\sfa=2\sfr$, with $\sfr>0$.  Then $\sfH_1$ is satisfied if and only if $\ttj_\ell - \ttj_{\ell-1} > 1$, for all $\sfp-\sfa < \ell \le \sfp$, and $\ttj_{\sfp-\sfr+1} - \ttj_{\sfp-\sfr} > 2$.  In this case $\sfH_2$ is also satisfied.


{\bf (B)}  Suppose that $n-1 \in \ttJ$.  Then $\sfa+1 \le |\ttJ|=\sfp \le \sfa+2$.  If $\sfp = \sfa+2$, then $\sfH_1$ fails for $\b = \a_{\ttj_1}$ and $\c$ the highest weight of $\fg_{1,C} \subset \fg_{1,\sfa}$ with $C = (0,1^\sfa,0)$.  So we assume $\sfp = \sfa+1$.  

{\bf (B.1)}  If $\sfa = 0$, then $\sfH_1$ is satisfied and $\sfH_2$ is vacuous.  


{\bf (B.2)}  Suppose that $0 < \sfa = 2\sfr ,\,2\sfr+1$.  Then $\sfH_1$ is satisfied if and only if $\ttj_\ell - \ttj_{\ell-1} > 1$, for all $1 \le \ell \le \sfp$, and $\ttj_{\sfp-\sfr} - \ttj_{\sfp-\sfr-1} > 2$.  In this case $\sfH_2$ is also satisfied.
\end{proof}

\section{The Schur system $\cR_w$} \label{S:schur}

\subsection{Definition}\label{S:Rw}

Define $R_w = \bP \bI_w \cap \tGr({|w|},\fg_{-1})$.   The \emph{Schur system} $\cR_w \subset \tGr({|w|} , T \, X)$ is defined by $\cR_{w,z} = g_* R_w$, $z = gP$, $g \in G$.  A ${|w|}$--dimensional complex submanifold $\mfd \subset X$ is an \emph{integral manifold of $\cR_w$} if $T\mfd \subset \cR_w$.   A subvariety $Y \subset X$ is an \emph{integral variety of $\cR_w$} if the smooth locus $Y^0 \subset Y$ is an integral manifold of $\cR_w$.  The Schur system is \emph{rigid} if for every integral manifold $M$, there exists $g \in G$ such that $M \subset g\cdot X_w$.  If every integral variety $Y$ is of the form $g\cdot X_w$, then we say $X_w$ is \emph{Schur rigid}.

\begin{remark*} Similar to $\cB_w$, we have that $\cR_w$ rigid implies $X_w$ is Schur rigid.
\end{remark*}

The following theorem is proved in \cite{Walters} and \cite[\S2.8.1]{SchurRigid}. 

\begin{theorem}\label{T:Rw_homology}
A subvariety $Y \subset X$ is an integral variety of $\cR_w$ if and only if $[Y] = r[X_w] \in H_k(X)$ with $0 < r \in \bZ$.
\end{theorem}

\noindent The significance of the theorem in relation to our motivating question is that it implies: \emph{if $X_w$ is singular and Schur rigid, then $X_w$ is not homologous to a smooth variety.}

Lemma \ref{L:dual}(c) and Theorem \ref{T:Rw_homology} imply

\begin{corollary} \label{C:pd}
 $[X_{w^*}]$ is Poincar\'e dual to $[X_w]$.
\end{corollary}

\begin{proposition} \label{P:conj}
Let $G/P$ be an irreducible compact Hermitian symmetric space.  Let $\varphi$ be an automorphism of the Dynkin diagram $\d_\fg$ and $\varphi : \fg \to \fg$ denote the induced Lie algebra automorphism.  Set $\fp' = \varphi(\fp)$.  Given $w \in W^\fp$, set $w' = \varphi(w) \in W^{\fp'}$.  Then the Schubert system $\cB_w$ (resp. the Schur system $\cR_w$) on $G/P$ is rigid if and only if the Schubert system $\cB_{w'}$ (resp. the Schur system $\cR_{w'}$) on $G/P'$ is rigid.
\end{proposition}

\begin{proof}
The algebra isomorphism $\varphi : \fg \to \fg$ induces a Lie group isomorphism $\varphi : G \to G$ with the property that $\varphi(P) = P'$.  Thus, $\varphi$ induces a biholomorphism $\varphi : G/P \to G/P'$.  Let $\varphi_* : \tGr(|w|,T(G/P)) \to \tGr(|w'|,T(G/P'))$ denote the map induced by the push-forward.  Review the definitions of \emph{rigid} for the Schubert (Section \ref{S:schubert}) and Schur (above) systems.  

Let $\fg = \fg_{-1}' \op \fg_0' \op \fg_1'$ denote the graded decomposition of $\fg$ induced by $\fp'$ (Section \ref{S:Z}).  Note that $\varphi$ maps $\fg_0$ to $\fg_0'$, inducing a Lie algebra isomorphism, and $\fg_{-1}$ to $\fg_{-1}'$; moreover, the latter is a $\fg_0 \simeq \fg_0'$--module map.  Additionally, $\Delta(w') = \varphi(\Delta(w))$, so that $\varphi(\fn_w) = \fn_{w'}$.  Thus, $\varphi_*(\bI_w) = \bI_{w'}$.  It follows that $\varphi_*(\cR_w) = \cR_{w'}$ and $\varphi_*(\cB_w) = \cB_{w'}$.
\end{proof}

\begin{remark} \label{R:tw_triv}
It may happen that $\tw^{|w|}\fg_{-1}$ is irreducible as an $\fg_0$--module, so that $R_w = \tGr(|w|,\fg_{-1})$.  In this case $H_{2|w|}(X) = \bZ$ is generated by a single $[X_w]$.  Every subvariety $Y \subset X$ of dimension $|w|$ is an integral variety of $\cR_w$, so $X_w$ cannot be Schur rigid.  This is the case for (i) all $w$ of length one; (ii) all $w$ associated the CHSS $\bP^n = A_n/P_1 \simeq A_n/P_n$ and $Q^{2n-1} = B_n/P_1$; and (iii) for all $w$ such that $|w|\not= n-1$ associated to the CHSS $Q^{2n-2} = B_n/P_1$.   By Theorem \ref{T:BR}, none of these $X_w$ satisfy $\sfH_+$.
\end{remark}

Note that $B_w \subset R_w$ is the $G_0$--orbit of the highest weight line $\fn_w \in \bI_w$.   The following proposition was proven by Bryant in the case that $X$ is a Grassmannian \cite{SchurRigid}, and by Hong in the general setting \cite[Proposition 2.10]{MR2276624}.  (The hypotheses of Hong's Proposition 2.10 are slightly more restrictive, but her proof establishes our Proposition \ref{P:Rrigid}.)

\begin{proposition} \label{P:Rrigid}
The Schur system $\cR_w$ is rigid if and only if the Schubert system $\cB_w$ is rigid and $B_w = R_w$.
\end{proposition}

\noindent Having studied the Schubert system (Section \ref{S:BR}), we now wish to determine when $B_w = R_w$.  

\subsection{Lemmas}

We claim that $R_w$ is connected.  To see this note that each connected component  of $R_w$ contains a closed orbit in $\bP\bI_w$.  Because the reductive $\fg_0$ acts irreducibly on $\bI_w$, the space $\bP\bI_w$ contains a unique closed orbit.  Hence $R_w$ is connected.  From this we deduce the following.

\begin{lemma}[{\cite{SchurRigid,MR2276624}}]  \label{L:Rw_cntd}
The equality $B_w = R_w$ holds if and only if $\widehat T_{\fn_w} B_w = \widehat T_{\fn_w} R_w$.
\end{lemma}

Regard $\fg_{0,-}$ as a $\fg_{0,0}$--submodule of $\fn_w^\perp \ot \fn_w^*$.  Since $\fg_{0,0}$ is reductive, there exists a $\fg_{0,0}$--module $\ft_w$ with the property that $\fn_w^\perp \ot \fn_w^* = \fg_{0,-} \op \ft_w$.  From the definition of $\partial$ it follows that 
$$
  H^1_0(\fn_w,\fg^\perp_w) \ \stackrel{\eqref{E:Hk_deg}}{=} 
  \ \frac{\fn_w^\perp \ot \fn_w^*}
  {\tim \{ \partial : \fg_{0,-} \to \fn_w^\perp \ot \fn_w^* \}} \ = \ \ft_w \, .
$$
The identifications $\fg_{0,-} \simeq T_{\fn_w} B_w$ and $\fn_w^\perp \ot \fn_w^* \simeq T_{\fn_w} \tGr({|w|},\fg_{-1})$ yield $T_{\fn_w} \tGr({|w|},\fg_{-1}) \simeq T_{\fn_w} B_w \,\op\, \ft_w$.  This establishes the following.

\begin{lemma}[{\cite[Proposition 3.5]{MR2276624}}] \label{L:H10}
The graded cohomology $H^1_0(\fn_w,\fg^\perp_w)$ is the $\fg_{0,0}$--module complement of $T_{\fn_w} B_w$ in $T_{\fn_w} \tGr(|w|,\fg_{-1})$.
\end{lemma}

Let $\{ \c_1 \,,\, \c_2 \,,\, \ldots \,,\, \c_{|w|} \}$ be an enumeration of the roots $\Delta(w)$.  Define
$$
  \langle w \rangle \ := \ \textstyle\sum_j \c_j \quad \hbox{ and } \quad
  \bv_w \ := \ E_{-\c_1} \wedge \cdots \wedge E_{-\c_{|w|}} \ \in \ \bI_w \, .
$$
Then $\bv_w$ spans the highest weight line $\fn_w \hookrightarrow \bP\bI_w$ of weight $-\langle w\rangle$.  Note that $\xi\in\fn_w^\perp \ot \fn_w^*$ naturally acts on $\bv_w$ by 
$$
  \xi\bv_w \ := \ E_{-\b} \wedge (E_{\c} \lefthook \bv_w) \,,\quad
  \hbox{ when } \xi = E_{-\b} \ot E_\c \, . 
$$
Moreover,   
\begin{eqnarray*}
  \widehat T_{\fn_w} \tGr({|w|},\fg_{-1}) & = & 
  \tspan\{ \bv_w \,,\, (\fn_w^\perp \ot \fn_w^*)\bv_w\} \ \subset \ 
  \tw^{|w|} \fg_{-1} \, ,\\
  \widehat T_{\fn_w} R_w & = & 
  \bI_w \,\cap\, \tspan\{ \bv_w \,,\, (\fn_w^\perp \ot \fn_w^*)\bv_w\} \,,\\
  \widehat T_{\fn_w} B_w & = &  \tspan\{ \bv_w \,,\, (\fg_{0,-})\bv_w \} \, .
\end{eqnarray*}
Thus, $\widehat T_{\fn_w}B_w = \widehat T_{\fn_w} R_w$ (and $B_w = R_w$ by Lemma \ref{L:Rw_cntd}) if and only if $\bI_w \cap (\ft_w)\bv_w = \{0\}$.

\begin{lemma}[{\cite[Proposition 3.4]{MR2276624}}] \label{L:B=R}
The equality $B_w = R_w$ \emph{fails} if and only if there exists a  $\fg_{0,0}$--highest weight vector $0 \not= \xi \in \ft_w = H^1_0(\fn_w,\fg^\perp_w) \simeq \cH^1_0$ with the property that $\xi\bv_w\in\bI_w$.
\end{lemma}

\begin{proof}
Note that both $\bI_w$ and $(\ft_w)\bv_w$ are $\fg_{0,0}$--modules.  Let $U \subset \ft_w$ be an irreducible $\fg_{0,0}$--submodule with highest weight vector $\xi \in U$.  Then either $U\bv_w \subset \bI_w$ or $\bI_w \cap U\bv_w = \{0\}$.  The former holds if and only if $\xi\bv_w \subset \bI_w$.
\end{proof}

In Section \ref{S:RR} we apply Proposition \ref{P:Rrigid} and Lemma \ref{L:B=R} to show that the Schubert varieties satisfying $\sfH_+$ (Section \ref{S:BR}) are Schur rigid.

\begin{definition*}
Let $\Pi(w)$ be the set of pairs $(\c,\b) \in \Delta(\fg_1) \times \Delta(\fg_1)$ such that:
 \begin{enumerate}
 \item the root $\c$ is the highest weight of an irreducible $\fg_{0,0}$--module $\fg_{1,C} \subset \fn_w^+$; 
 \item the root $-\b$ is the highest weight of an irreducible $\fg_{0,0}$--module $\fg_{-1,-B} \subset \fn^\perp_w$ (equivalently, $\b$ is a lowest weight of $\fg_{1,B}$); and
 \item $\c-\b$ is \emph{not} a root of $\fg$.  
 \end{enumerate} Let $U_{\c-\b} \subset \fg_{-1,-B} \ot \fg_{1,C}$ denote the Cartan product.
\end{definition*}

\begin{lemma} \label{L:cH10}
$\cH^1_0 = \bigoplus_{(\c,\b) \in \Pi(w)}  U_{\c-\b}$.
\end{lemma}

\noindent The proof of Lemma \ref{L:cH10} is given in Section \ref{S:PcH10}.  \medskip

Given an element $\pi$ of the $\fg_{0,0}$--root lattice, let $\bI^{-\pi}_w \subset \bI_w$ be the weight space of weight $-\langle w \rangle-\pi$.  By Lemma \ref{L:B=R}, 
\begin{equation} \label{E:B=Rtest}
\begin{array}{l}
\hbox{\emph{the equality $B_w = R_w$ fails if and only if}}\\
\hbox{\emph{there exists $(\c,\b) \in \Pi(w)$ such that $\xi\bv_w \in \bI^{\c-\b}_w$.}}
\end{array}\end{equation}

Let $\be = \{ \e_j\}_{j=1}^t \subset \Delta(\fg_0)$ be an ordered sequence.  Define 
\begin{equation*} 
  \be . \bv_w \ := \ 
  E_{\e_t} . ( \ldots E_{\e_2} .( E_{\e_1} . \bv_w ) \ldots  ) 
  \ \in \ \bI_w \, . 
\end{equation*}

\begin{lemma} \label{L:Iwt}
Fix $(\c,\b) \in \Pi(w)$ and set $\sfs := (\b-\c)(Z_w)$.  The weight space $\bI^{\c-\b}_w$ is spanned by vectors of the form $\bb.\bv_w$ where $\bb$ ranges over ordered sequences $\bb = \{-\b_j\}_{j=1}^\sfs \subset \Delta(\fg_{0,-1})$ such that $\b-\c = \sum_1^\sfs \b_j$.  In particular, $\b-\c$ is a sum of positive roots, and $\sfs > 1$.
\end{lemma}

\begin{remark} \label{R:s}
Suppose $\sfa(w) = 0$ and $\fg_{-1} = \fg_{-1,0} \op \fg_{-1,-1}$.  The condition $\sfs>1$ of Lemma \ref{L:Iwt} implies that $B_w = R_w$ holds.
\end{remark}

\begin{proof}[Proof of Lemma \ref{L:Iwt}]
By standard weight theory, $\bI^{-\pi}_w$ is spanned by vectors of the form 
$\be . \bv_w$ where $\be$ ranges over all ordered sequences $\{ -\e_j \}_{j=1}^t$, with each $\e_j$ a simple root and $-\langle w \rangle - \sum_j \e_j = -\langle w \rangle - \pi$; equivalently $\sum_j \e_j = \pi$.  (Here $t$ may vary with the choice of sequence $\{-\e_j\}$.)  

Note that $\e_j$ must lie in $\Delta(\fg_{0,1}) \sqcup \Delta(\fg_{0,0})$.  Since $\fg_{0,\ge0}$ is the stabilizer of $\fn_w$ in $\fg_0$, we have $E_{-\e_j} . \bv_w = 0$ if and only if $\e_j \in \Delta(\fg_{0,0})$.  So, without loss of generality, $\e_1 \in \Delta(\fg_{0,1})$.  If $\e_2 \in \Delta(\fg_{0,0})$, then 
\begin{eqnarray*}
  E_{-\e_2} . (E_{-\e_1} . \bv_w) & = &
  [E_{-\e_2} \,,\, E_{-\e_1}] . \bv_w + E_{-\e_1} . (E_{-\e_2} . \bv_w ) \\
  & = & [E_{-\e_2} \,,\, E_{-\e_1}] . \bv_w \ + \ 
  c\, E_{-\e_1} . \bv_w \, ,
\end{eqnarray*}
for some $c\in\bC$.  Either $[E_{-\e_2} \,,\, E_{-\e_1}] = 0$, or $\e_1+\e_2 \in \Delta(\fg_{0,1})$.  It follows now from an inductive argument that the vectors $\bb.\bv_w$ span $\bI^{-\pi}_w$.

Fix $(\c,\b) \in \Pi(w)$ with $\b-\c=\pi$.  The inequality $\sfs \ge 1$ follows from the definition of $\Pi(w)$.  If equality holds, $\b-\c= \b_1$ is a root.  This contradicts the definition of $\Pi(w)$.  
\end{proof}

\begin{remark} \label{R:Iwt}
Suppose that \emph{(a)} for every ordered sequence $\bb$ of Lemma \ref{L:Iwt} it is the case that $[E_{\b_j} \,,\, E_{\b_k} ] = 0$ for all $1\le j,k\le \sfs$, and \emph{(b)} for any every pair $\bb$ and $\bb'$ of ordered sequences, the unordered sets are equal.  Then $\bb . \bv_w = \bb' . \bv_w$.  It follows that $\tdim\,\bI^{-\pi}_w = 1$.
\end{remark}

\subsection{Proof of Lemma \ref{L:cH10}} \label{S:PcH10}

Given $0 < \ell \in \bZ$, let $H^1_{0,-\ell}$ be the component of $H^1_0(\fn_w,\fg^\perp_w)$ of $Z_w$--degree $-\ell$.  Then 
\begin{equation} \label{E:opH10}
  H^1_0(\fn_w,\fg_w^\perp) \ = \ \bigoplus_{0 < \ell} H^1_{0,-\ell} \, .
\end{equation}
Fix $0 < \ell \in \bZ$, let $T^k$ be the component of $\fg^\perp_w \ot \tw^k\fn_w^+$ of $(Z_\tti,Z_w)$--bidegree $(0,-\ell)$.  Then 
\begin{eqnarray}
  \nonumber
  T^0 & = & \fg_{0,-\ell} \,,\\
  \label{E:F1}
  T^1 & = & T^1_0 \, \op \, 
 T^1_1 \,\op \,\cdots\,\op\,T^1_{\ell-1}\,,\quad
  \hbox{with } \ T^1_m \ := \ \fg_{-1,m-\sfa-\ell} \ot \fg_{1,\sfa-m} \,,
\end{eqnarray}
and $T^k = \{0\}$ for all $k\not=0,1$.  Since the Lie algebra cohomology differential preserves the $(Z_\tti,Z_w)$--bidegree, we have a subcomplex
\begin{equation} \label{E:subcpx}
  \{0\} \ \longrightarrow \ T^0 \ 
  \stackrel{\partial}{\longrightarrow} \ T^1 \ 
  \longrightarrow \ \{0\} \, ,
\end{equation}
and
$$
  H^1_{0,-\ell} \ = \ T^1 \ \tmod \ \partial T^0 
  \ =: \ H^1(T^\bullet , \partial) \, .
$$
Let $\partial_m$ denote the component of $\partial:T^0\to T^1$ taking value in $T^1_m$.  Consider the associated subcomplex
\begin{equation*}
  \{0\} \ \longrightarrow \ T^0 \ 
  \stackrel{\partial_m}{\longrightarrow} \ T^1_m \ 
  \longrightarrow \ \{0\} \, .
\end{equation*}
Note that $\partial_m$ is a $\fg_{0,0}$--module map.  Therefore, the cohomology 
\begin{equation}   \label{E:Hm}
  H^1(T^\bullet,\partial_m) \ := \ T^1_m \ \tmod \ \partial_m(T^0)
\end{equation}
is a $\fg_{0,0}$--module.

\begin{lemma} \label{L:BRcoh}
The $\fg_{0,0}$--modules $H^1_{0,-\ell} = T^1/\partial(T^0)$ and $\op_m H^1(T^\bullet,\partial_m)$ are isomorphic.
\end{lemma}

The lemma is proved in Section \ref{S:PBRcoh}.  We continue here with a discussion of the consequences.  Observe that $\partial^*_m := \sfd^*{}_{|T^1_m} : T^1_m \to T^0$ is adjoint to $\partial_m$ with respect to the positive definite Hermitian inner product $\langle \cdot , \cdot \rangle$ of Section \ref{S:adj}.  Let 
$$
  \square_m \ := \ \partial_m\circ\partial_m^* \, : \, T^1_m \, \to \, T^1_m
$$ 
denote the associated Laplacian.  Analogous to Proposition \ref{P:new_coh}, there exists a $\fg_{0,0}$--module isomorphism $H^1(T^\bullet,\partial_m) \simeq \cH^1_{\partial_m} := \tker\,\square_m$.  Therefore, corollary to Lemma \ref{L:BRcoh}, we have
\begin{equation} \label{E:BRcoh}
  H^1_{0,-\ell} \ \simeq \ \bigoplus_{m=0}^{\ell-1} \cH^1_{\partial_m} 
\end{equation}
as $\fg_{0,0}$--modules

Section \ref{S:action} yields the following expression
\begin{eqnarray*}
  \square_m & \stackrel{\eqref{E:partial}}{=} & 
  \sum_{\a\in\Delta(w,\sfa-m)} \epsilon_{E_\a}\circ\cL'_{E_{-\a}}\circ\partial^*_m \\
  & \stackrel{\eqref{E:abelian}}{=} & \sum_{\a\in\Delta(w,\sfa-m)} \left( 
  \epsilon_{E_\a}\circ\cL'_{E_{-\a}}\circ\partial^*_m \,+\, 
  \partial^* \circ \epsilon_{E_\a}\circ\cL'_{E_{-\a}} \right) \\
  & \stackrel{(\ast)}{=} & 
  \sum_{\a\in\Delta(w,\sfa-m)} \cL'_{E_\a}\circ\cL'_{E_{-\a}} \ - 
  \sum_{\z_j \in \fg_{0,0}} \cL''_{\z_j}\circ\cL'_{\xi_j}
  \ \stackrel{\eqref{E:abelian}}{=} \ 
  - \sum_{\z_j \in \fg_{0,0}} \cL''_{\z_j}\circ\cL'_{\xi_j} \, .
\end{eqnarray*}
The equality $(\ast)$ is obtained from an application of Lemma \ref{L:formulas}(b), followed by applications of Lemma \ref{L:formulas}(a,c).

Let $\fg_{-1,-B} \subset \fg_{-1,m-\sfa-\ell}$ and $\fg_{1,C} \subset \fg_{1,\sfa-m}$ be irreducible $\fg_{0,0}$--modules of highest weights $-\b\in\Delta(\fg_{-1,-B})$ and $\c \in \Delta(\fg_{1,C})$, respectively.  Let $U \subset \fg_{-1,-B} \ot \fg_{1,C} \subset T^1_m$ be an irreducible module of highest weight $\pi$.  By Lemma \ref{L:C00}, the Laplacian $\square_m$ acts on $U$ by a scalar $c \ge (\b,\c)$, with equality if and only if $\pi = \c-\b$.  By \eqref{E:abelian} $\c+\b \not\in \Delta(\fg)$.  So $(\c,\b) \ge 0$, and equality holds if and only if $\c-\b \not\in\Delta(\fg)$.  Thus $U = U_\pi \subset \cH^1_{\partial_m}$ if and only if $\pi = \c-\b$ for some $(\c,\b) \in \Pi(w)$.  Given \eqref{E:BRcoh}, this establishes Lemma \ref{L:cH10}.

\subsection{Proof of Lemma \ref{L:BRcoh}} \label{S:PBRcoh}

The proof of the lemma is by induction.  Define a filtration $\{0\} \subset F^1 T^\bullet \subset F^0 T^\bullet = T^\bullet$ of the complex \eqref{E:subcpx} by 
 \[
 F^1T^0 = \{0\}, \qquad F^1T^1 = T^1_1 \oplus \cdots \oplus T^1_{\ell-1}.
 \] 
 We may identify $\tGr^0 T^1 := F^0T^1/F^1T^1$ with $T^1_0$ as $\fg_{0,0}$--modules.  Consider the associated spectral sequence $\{E^{p,q}_r\}$.  (See \cite[Chapter 3.5]{MR1288523} for a treatment of the spectral sequence associated to a filtered complex.)  We have $E_2^{p,q} = E_\infty^{p,q}$ with $E_\infty^{0,0} = \tker\{\partial: T^0 \to T^1\}$, 
$$
  E_\infty^{0,1} \ = \ \frac{T^1}{F^1T^1 + \partial T^0} \,,
  \hbox{ and } \quad
  E_\infty^{1,0} \ = \ \frac{F^1T^1}{F^1T^1 \cap \partial T^0} \quad
$$ 
and all other $E_\infty^{p,q} = \{0\}$.  The vector space identification (see \cite[Chapter 3.5]{MR1288523})
$$
  H^1_{0,-\ell} \ = \ H^1(T^\bullet,\partial) \ \simeq \ 
  \tGr^0H^1(T^\bullet,\partial) \, \op \, \tGr^1H^1(T^\bullet,\partial) 
  \ = \ E_\infty^{0,1} \,\op\, E_\infty^{1,0}
$$
is a $\fg_{0,0}$--module identification.  Additionally, $E_\infty^{0,1} \simeq H^1(T^\bullet , \partial_0)$ as $\fg_{0,0}$--modules, cf. \eqref{E:Hm}.  

The module $E_\infty^{1,0}$ is described as follows.  Let $F^1\partial = \partial_1+\cdots+\partial_{\ell-1}$ denote the component of $\partial : T^0 \to T^1$ taking values in $F^1T^1$, with respect to the decomposition \eqref{E:F1}.  Then 
$$
  \{0\} \ \longrightarrow \ \cC^0 \ := \ T^0  \ 
  \stackrel{F^1\partial}{\longrightarrow} \ \cC^1 \ := \ F^1T^1
  \ \longrightarrow \  \{0\}
$$ 
is a complex, and $E_\infty^{1,0} = H^1(\cC^\bullet,F^1\partial)$.  Thus,
\begin{equation}\label{E:m=1}
  H^1_{0,-\ell} \ \simeq \ 
  H^1(T^\bullet,\partial_0) \, \op \, H^1(\cC^\bullet,F^1\partial) \, .
\end{equation}

This completes the first step in the induction.  In order to state the inductive hypothesis let $0 \le m < \ell$ and define $\cC_m^0 := T^0$ and 
\begin{equation} \label{E:C1m}
  \cC_m^1 \ := \ T^1_m \ \op \ 
  \underbrace{T^1_{m+1} \op \cdots \op T^1_{\ell-1}}_{=: \ F^1\cC^1_m} \, .
\end{equation}
Let $F^m\partial = \partial_m + \cdots + \partial_{\ell-1}$ denote the component of $\partial : T^0 \to T^1$ taking values in $\cC_m^1$, with respect to the decomposition \eqref{E:F1}.  Then 
\begin{equation} \label{E:cpx_m}
  0 \ \longrightarrow \ T^0 \ \stackrel{F^m\partial}{\longrightarrow} \ 
  \cC^1_m \ \longrightarrow \ \{0\}
\end{equation}
is a complex; let $H^1(\cC^\bullet_m , F^m\partial)$ denote the associated cohomology.  Observe that  
\begin{equation} \label{E:indhyp}
  H^1(T^\bullet,\partial) \ = \ 
  \bigg( \bigoplus_{m<r} H^1(T^\bullet,\partial_m) \bigg)
  \ \op \ H^1(\cC^\bullet_{r} , F^r\partial) \, .
\end{equation}
holds for $r=1$ by \eqref{E:m=1}.  Inductive hypothesis:  \eqref{E:indhyp} holds for some $r=r_o$ with $0 < r_o < \ell-1$.  We will show that \eqref{E:indhyp} holds for $r=r_o+1$.  The lemma will then follow.

Define a filtration $\{0\} \subset F^1\cC^\bullet_m \subset F^0\cC^\bullet_m := \cC^\bullet_m$ on the complex \eqref{E:cpx_m} by $F^1\cC^0_m = \{0\}$ and \eqref{E:C1m}.  Spectral sequence computations identical to those yielding \eqref{E:m=1} produce
\begin{equation*} 
  H^1(\cC^\bullet_m , F^m\partial) \ = \ 
  H^1(T^\bullet , \partial_m ) \ \op \ 
  H^1(\cC^\bullet_{m+1} , \partial^{m+1}) \, .
\end{equation*}
Combined with the inductive hypothesis, this implies that \eqref{E:indhyp} holds for $r = r_o+1$.  This completes the induction and establishes the lemma.

\section{Schur rigidity} \label{S:RR}

Theorem \ref{T:BR} lists the proper Schubert varieties which satisfy $\sfH_+$; these are the varieties for which there exist first-order obstructions to Schubert flexibility.  The main result of the paper is
 
\begin{theorem} \label{T:RR}
The varieties listed in Theorem \ref{T:BR} are Schur rigid.
\end{theorem}

\begin{proof}[Outline of proof of Theorem \ref{T:RR}]   By hypothesis, in all sections to follow, $X_w$ will be assumed to satisfy $\sfH_+$, hence is listed in Theorem \ref{T:BR} and is Schubert rigid.
 By Proposition \ref{P:Rrigid}, it suffices to show that the equality $B_w = R_w$ holds.  This is done in the sections that follow.  The computational tools are \eqref{E:B=Rtest} and Lemma \ref{L:Iwt}. 
\end{proof}

Be aware that the action of $\bb$ on $\tw^{|w|}\fg_{-1}$ is induced by the adjoint action of $\fg_0$ on $\fg_{-1}$, while the action of $\xi = E_{-\b} \ot E_{\c}\in \fn_w^\perp \ot \fn_w^*$ is given by $\epsilon_{E_{-\b}} \circ \iota_{E_\c}$; that is $\xi\bv_w := E_{-\b} \wedge( E_\c \lefthook \bv_w)$.

\subsection{Comparison with Hong's results} \label{S:hong}

Before proving the theorem, we discuss the relationship of Theorem \ref{T:RR} to the main results in \cite{MR2276624, MR2191767}.  In the smooth case ($\sfa = 0$) we recover precisely Hong's result in Section \ref{S:hist}; see Table \ref{t:sm}.

\begin{table}[h] 
\renewcommand{\arraystretch}{1.2}
\caption{The proper smooth Schubert varieties satisfying $\sfH_+$.}
\begin{tabular}{|c|c|l|}
\hline
$G/P$ & $\ttJ$ & Description of $X_w$ 
    \\ \hline \hline 
 & $\{\tti-1\}$ {\small{\&}} $\{\tti+1\}$ 
    & maximal linear subspaces $\bP^{n-\tti+1}$ {\small{\&}} $\bP^{\tti}$, resp.
    \\ \cline{2-3}
$A_n/P_\tti$ & $\{ \ttj_1 , \ttj_2 \}$ 
     & $\tGr( \tti-\ttj_1 \,,\, \ttj_2-\ttj_1)$ 
     \\ \cline{2-3}
 & \multicolumn{2}{|l|}{with $1 < \tti < n$ 
                        and $1 < \tti-\ttj_1 \,, \ \ttj_2 - \tti$.}    
     \\ \hline
$D_n/P_1$ & $\{n-1\}$ {\small{\&}} $\{n\}$ & maximal linear subspaces $\bP^{n-1}$
     \\ \hline
$C_n/P_n$ & $\{\ttj\}$ & $C_{n-\ttj}/P_{n-\ttj}$ with $\ttj<n-1$ 
     \\ \hline
 & & $D_{n-\ttj}/P_{n-\ttj}$, if $\ttj<n-3$; 
     \\ \cline{3-3} 
\raisebox{1.5ex}[0pt]{$D_n/P_n$} & \raisebox{1.5ex}[0pt]{$\{ \ttj \}$} 
     & $\bP^3$ with $\ttj=n-3$, and $\bP^{n-1}$ with $\ttj=n-1$
     \\ \hline 
$E_6/P_6$ & $\{1\}$, $\{2\}$ {\small{\&}} $\{3\}$ 
     & $Q^8 = D_5/P_1$, $\bP^5$ {\small{\&}} $\bP^4$, resp. 
     \\ \hline 
$E_7/P_7$ & $\{1\}$, $\{2\}$ {\small{\&}} $\{3\}$ 
     & $Q^{10} = D_6/P_1$, $\bP^6$ {\small{\&}} $\bP^5$, resp.  
     \\ \hline
\end{tabular} 
\label{t:sm}
\end{table}

The main result of \cite{MR2191767} is 

\begin{theorem*}[Hong] \label{T:Hong-singular} 
Let $\pi = (p_1{}^{q_1},\ldots,p_r{}^{q_r}) \in \ttP(\tti,n+1)$, $p_r \not= 0$ with conjugate $\pi' = (p_1'{}^{q_1'},\ldots,p_r'{}^{q_r'}) \in \ttP(n+1-\tti,n+1)$, $p_r' \not= 0$. If $q_i,q_i' \geq 2$ for all $1 \leq i \leq r$, then the Schubert variety $X_\pi \subset \tGr(\tti,n+1) = A_n / P_\tti$ is Schur rigid.
\end{theorem*}

\noindent Hong's theorem assumes, in particular, that $1 < q_1 \,,\ q_1'$.  Comparing with Remark \ref{R:BRpart} we see that Hong's theorem omits some Schur rigid $X_\pi$ with $\pi \in \spadesuit$, $\heartsuit$ and $\diamondsuit$.
 
\subsection{Grassmannian \boldmath $\tGr(\tti,n+1) = A_n/P_\tti$ \unboldmath} \label{S:Ak_RR}

In this section we will prove that $B_w = R_w$ holds for the varieties in Theorem \ref{T:BR}(a).  By Proposition \ref{P:Rrigid} it suffices to show that $R_w = B_w$ when condition $\sfH_+$ holds.  

The Schubert variety $X_w$ is given by Corollary \ref{C:bigone}.  In the case that $|\ttJ|=1$, we have $\fg_{-1} = \fg_{-1,0} \op \fg_{-1,-1}$, and $R_w = B_w$ is given by Remark \ref{R:s}.

Next consider the case that $|\ttJ|>1$.  By the lemma below we may assume $(\b-\c)(Z_w) = 2$.

\begin{lemma} \label{L:B=R_2}
The equality $B_w = R_w$ fails if and only if there exists $(\c,\b) \in \Pi(w)$ with $(\b-\c)(Z_w) = 2$ such that $\xi\bv_w = E_{-\b}\wedge(E_\c \lefthook \bv_w) \in \bI^{\c-\b}_w$.
\end{lemma}

\noindent The lemma is proven below.  It also holds for $C_n$ and $D_n$, but not $E_6$ and $E_7$.

Recall that $\b-\c$ is a sum of positive roots (Lemma \ref{L:Iwt}).  Since $\b$ is a lowest $\fg_{0,0}$--weight, it is of the form $\b = \a_\ttj + \cdots + \a_{\ttj'}$, where $\ttj,\ttj' \in \ttJ$ and $\ttj < \tti < \ttj'$.  Similarly, since $\c$ is a highest $\fg_{0,0}$--weight, it is of the form  $\c = \a_{\ttk+1} + \cdots + \a_{\ttk'-1}$, where $\ttk,\ttk' \in \ttJ \cup \{0,n+1\}$ and $\ttk < \tti < \ttk'$.  The condition that $\b-\c$ be a sum of positive roots, but not a root itself, implies that $\ttj \le \ttk$ and $\ttk' \le \ttj'$.  Moreover, $(\b-\c)(Z_w) = 2$ forces $\ttj = \ttk$ and $\ttk' = \ttj'$ so that $\b-\c = \a_\ttj + \a_{\ttj'}$.  By Remark \ref{R:Iwt} the weight space $\bI^{\c-\b}_w$ is one-dimensional and spanned by $\bb . \bv_w$, where $\bb = \{ -\a_\ttj \,,\, -\a_{\ttj'}\}$.  From Lemma \ref{L:B=R_2} we conclude that 
\begin{equation} \label{E:B=R_LI}
\begin{array}{c}
  \hbox{\emph{$B_w = R_w$ holds if and only if $\bb.\bv_w$ and $\xi.\bv_w$
              are linearly}} \\
  \hbox{\emph{independent for all $(\c,\b) \in \Pi(w)$ with 
              $(\b-\c)(Z_w)=2$.}}
\end{array}\end{equation}

Define $\b_1,\b_2 \in \Delta(\fg_{0,1})$ by $\bb = \{ -\b_1 , -\b_2 \}$.  A priori
$$
  E_{-\b_1} . \bv_w \ = \ \sum_{\stack{\n\in\Delta(w)}{\nu+\b_1\in\Delta}} 
  c_\n \, E_{-\n-\b_1} \wedge (E_\n \lefthook \bv_w) \, ,
$$
for some $c_\n\in\bC$.  However, since $E_{-\n-\b_1} \in \fn_w$ for all $\n$ such that $\n(Z_w) < \sfa$, we have  
\begin{eqnarray*}
  E_{-\b_1} . \bv_w & = & 
  \sum_{\stack{\nu\in\Delta(\fg_{1,\sfa})}{\nu+\b_1\in\Delta}}
  c_\nu \, E_{-\nu-\b_1} \wedge (E_{\nu} \lefthook \bv_w)
  \ = \ 
  \sum_{\stack{\nu\in\Delta(\fg_{1,\sfa})}{\nu+\b_1\in\Delta}}
  c_\nu \,(E_{-\nu-\b_1} \ot E_\nu).\bv_w
\end{eqnarray*}
for nonzero coefficients $c_\n$.  Similarly,
\begin{equation} \label{E:b.vw}
\begin{array}{rcl}
  \bb . \bv_w & = & \displaystyle
  \sum_{\stack{\nu\in\Delta(\fg_{1,\sfa})}{\stack{\nu+\b_1\in\Delta}
                                          {\nu+\b_1+\b_2\in\Delta}}}
  c^1_\nu \, E_{-\nu-\b_1-\b_2} \wedge (E_{\nu} \lefthook \bv_w)
  + 
  \sum_{\stack{\nu\in\Delta(\fg_{1,\sfa})}{\stack{\nu+\b_1\in\Delta}
              {\mu=\nu-\b_2\in\Delta}}}
  c^2_\nu \, E_{-\nu-\b_1} \wedge (E_{\mu} \lefthook \bv_w)
  \\ & & \displaystyle 
  + \sum_{\stack{\nu,\mu\in\Delta(\fg_{1,\sfa})}
                {\mu+\b_2,\nu+\b_1\in\Delta}} 
  c_{\nu,\mu} \,( E_{-\n-\b_1}\wedge E_{-\b_2-\mu}) \wedge 
  \left( (E_\n\wedge E_\m) \lefthook \bv_w \right) \, ,
\end{array}\end{equation}
for nonzero coefficients $c^1_\nu$, $c^2_\n$ and $c_{\n,\m}$.  

Consider the first sum of \eqref{E:b.vw}.  The condition that both $\nu+\b_1$ and $\nu + \b_1+\b_2$ be roots uniquely determines $\nu$: it must be the case that $\nu = \c$.  Thus, the sum is empty if $\c(Z_w) = \sfa-1$, and a nonzero multiple of $\xi\bv_w$ if $\c(Z_w) = \sfa$.  

Next observe that second sum of \eqref{E:b.vw} is over the set of all 
$$
  \{\mu \in \Delta(\fg_{1,\sfa-1}) \ | \ \mu+\b_2 \,,\, 
    \mu + \b_1+\b_2\in \Delta\} \, .
$$
Again, the condition that both $\mu+\b_2$ and $\mu+\b_2+\b_1$ be roots uniquely determines $\mu$: it must be the case that $\mu = \c$.  So the sum is a nonzero multiple of $\xi\bv_w$ if $\c(Z_w) = \sfa-1$, and empty otherwise.  

From the observations above we conclude that $\bb.\bv_w$ and $\xi\bv_w$ are linearly independent if and only if the third sum of \eqref{E:b.vw} is nonzero.  Equivalently, 
\begin{equation} \label{E:3rdSum}
\begin{array}{c}
  \hbox{\emph{$\bb.\bv_w$ and $\xi\bv_w$ are linearly independent 
              if and only if there exist}} \\
  \hbox{\emph{distinct $\nu,\mu\in\Delta(\fg_{1,\sfa})$ such that 
              $\nu+\b_1$ and $\mu+\b_2$ are distinct roots.}}
\end{array}\end{equation}
There are two cases to consider: $\b(Z_w) = \sfa+2$ and $\c(Z_w) = \sfa$; or $\b(Z_w) = \sfa+1$ and $\c(Z_w) = \sfa-1$.  Suppose that $\ttj = \ttj_r$ and $\ttj' = \ttj_s$.
In the first case, the $\mu,\nu$ of \eqref{E:3rdSum} exist if and only if one of the following holds:
\begin{eqnarray*}
  1 < \ttj_{r+1} - \ttj_r & 
  \hbox{(or $\tti-\ttj_r > 1$ if $\ttj_r < \tti < \ttj_{r+1}$),} & 
  \quad \hbox{ or} \\
  1 < \ttj_s - \ttj_{s-1} &
  \hbox{(or $\ttj_s - \tti > 1$ if $\ttj_{s-1} < \tti < \ttj_s$).} & 
\end{eqnarray*}
In the second case the $\mu,\nu$ of \eqref{E:3rdSum} exist if and only if one of the following holds:
$$
  1 < \ttj_{r} - \ttj_{r-1} \qquad \hbox{or} \qquad  
  1 < \ttj_{s+1} - \ttj_s \, .
$$
(If $r = 1$, then $\ttj_0 := 0$; if $s = |\ttJ|$, then $\ttj_{s+1} := n$.)

Modulo the proof of Lemma \ref{L:B=R_2}, this completes the proof of the proposition. 

\begin{proof}[Proof of Lemma \ref{L:B=R_2} in the case that $G=A_n$]
Given \eqref{E:B=Rtest}, it suffices to show the following:  if there exists $(\b,\c)\in\Pi(w)$ such that $\xi\bv_w \in \bI^{\c-\b}_w$, then there exists $(\b_o,\c_o) \in \Pi(w)$ such that $\xi_o\bv_w \in \bI^{\c_o-\b_o}_w$ and $(\b_o-\c_o)(Z_w) = 2$.

The condition $(\b-\c)(Z_w) > 2$ holds if and only if either $\ttj < \ttk$ or $\ttk' < \ttj'$.  Suppose that $\ttj = \ttj_r$ and $\ttj < \ttk$.  Suppose that $\b(Z_w) > \sfa+1$.  Set $\be = \{ \a_{\ttj_r} \,,\, \a_{\ttj_r+1} \,,\ldots,\, \a_{\ttj_{r+1}-1}\}$.  Then $\be.(\xi\bv_w)$ is a nonzero multiple of $E_{-\b_o} \wedge (E_{\c} \lefthook \bv_w) = \xi_o\bv_w$ where $\b_o = \a_{\ttj_{r+1}} + \cdots + \a_{\ttj'}\in\Delta(\fg_1)$ is, like $\b$, a lowest $\fg_{0,0}$--weight, and $\b_o(Z_w) = \b(Z_w)-1$.  In particular, if $\xi\bv_w \in \bI_w^{\c-\b}$, then $\xi_o\bv_w \in \bI_w^{\c-\b_o}$.  Continuing inductively, we may assume that either $\ttj = \ttk$, or $\b(Z_w) = \sfa+1$.

Similarly, if $\ttk' < \ttj'=\ttj_s$ and $\b(Z_w) > \sfa+1$, set $\be = \{ \a_{\ttj_s} \,,\, \a_{\ttj_s-1} \,,\ldots,\, \a_{\ttj_{s-1}+1}\}$.  Then $\be.(\xi\bv_w)$ is a nonzero multiple of $E_{-\b_o} \wedge (E_{\c} \lefthook \bv_w) = \xi_o\bv_w$ where $\b_o = \a_{\ttj} + \cdots + \a_{\ttj_{s-1}}\in\Delta(\fg_1)$ is, like $\b$, a lowest $\fg_{0,0}$--weight, and $\b_o(Z_w) = \b(Z_w)-1$.  Again, if $\xi\bv_w \in \bI_w^{\c-\b}$, then $\xi_o\bv_w \in \bI_w^{\c-\b_o}$.  Continuing inductively, we may assume that either $\ttj = \ttk$ and $\ttj'=\ttk'$, or $\b(Z_w) = \sfa+1$.

If $\ttj = \ttk$ and $\ttj' = \ttk'$, then $(\b-\c)(Z_w)=2$.  So assume that $\ttj < \ttk$ and $\b(Z_w) = \sfa+1$.   Set $\ttk = \ttj_t$ and $\be = \{ \a_{\ttj_t} \,,\, \a_{\ttj_t-1} \,,\ldots,\, \a_{\ttj_{t-1}+1} \}$.  Then $\be.(\xi\bv_w)$ is a nonzero multiple of $\xi_o\bv_w = E_{-\b} \wedge (E_{\c_o} \lefthook \bv_w)$, where $\c_o = \a_{\ttj_{t-1}+1} + \cdots + \a_{\ttk'}$ is, like $\c$, a highest $\fg_{0,0}$--weight, and $\c_o(Z_w) = \c(Z_w)+1$.  As above, if $\xi\bv_w \in \bI_w^{\c-\b}$, then $\xi_o\bv_w \in \bI_w^{\c_o-\b}$.  Continuing inductively we may assume that $\ttj = \ttk$.  

By an analogous argument we may assume that $\ttj' = \ttk'$.
\end{proof}

\subsection{Quadric hypersurface \boldmath $Q^{2n-2} = D_n/P_1$ \unboldmath }

It is an immediate consequence of Remark \ref{R:s} that $B_w = R_w$ holds for the varieties in Theorem \ref{T:BR}(b).

\subsection{Lagrangian grassmannian \boldmath $C_n/P_n$ \unboldmath} \label{S:Cn_RR}

In this section we show that $B_w = R_w$ holds for the varieties in Theorem \ref{T:BR}(c).

Let $(\c,\b) \in \Pi(w)$.  As a lowest $\fg_{0,0}$--weight $\b$ is of one of the following forms
\begin{subequations} 
\begin{eqnarray}
  \label{E:Cb1}
  \b & = & \a_{\ttj_r} + \cdots + \a_{\ttj_s-1} + 
  2 ( \a_{\ttj_s} + \cdots + \a_{n-1}) + \a_n \,,\\
  \label{E:Cb2}
  \b & = & 2 (\a_{\ttj_{s}} + \cdots + \a_{n-1}) + \a_n \,,
  \qquad \b \ = \ \a_{\ttj_r} + \cdots + \a_n 
\end{eqnarray}
\end{subequations}
for some $\ttj_r , \ttj_s \in \ttJ = \{ \ttj_1 , \ldots , \ttj_\sfp\}$.  (Cf. the proof of Lemma \ref{L:bigone}.  Also, Corollary \ref{C:bigone} asserts that $\sfa \le p \le \sfa+1$.)  Similarly, as a $\fg_{0,0}$--highest weight, $\c$ is of one of the following forms
\begin{subequations} 
\begin{eqnarray}
  \label{E:Cc1}
  \c & = & \a_{\ttj_t+1} + \cdots + \a_{\ttj_u} + 
  2 ( \a_{\ttj_u+1} + \cdots + \a_{n-1}) + \a_n \,,\\
  \label{E:Cc2}
  \c & = & 2 (\a_{\ttj_u+1} + \cdots + \a_{n-1}) + \a_n \,, \qquad
  \c \ = \ \a_{\ttj_t+1} + \cdots + \a_n 
\end{eqnarray}
\end{subequations}
for some $\ttj_t,\ttj_u \in \ttJ\cup\{0\}$, with $\ttj_0 := 0$.  (The third $\c$ assumes that $n-1 \in \ttJ$.)  The requirement that $\b-\c$ be a sum of positive roots, but not a root itself, rules out the second root in \eqref{E:Cb2}.

Lemma \ref{L:B=R_2} also holds for $G = C_n$.  The proof is given at the end of this section.  The condition $(\b-\c)(Z_w) = 2$ implies that the pair $(\c,\b)$ is of one of the two following forms
\begin{eqnarray*}
  \b \ = \ \c + \a_\ttj + \a_\ttj' & = & \a_\ttj + \cdots + \a_{\ttj'-1} + 
  2 ( \a_\ttj + \cdots + \a_{n-1}) + \a_n \,,\\
  \hbox{or}\quad
  \b \ = \ \c + 2\,\a_\ttj & = & 2 (\a_\ttj + \cdots + \a_{n-1}) + \a_n \, .
\end{eqnarray*}
We express $\b-\c$ uniformly as $\a_\ttj + \a_\ttj'$ with $\ttj\le\ttj'\in \ttJ$.  The condition that $\b-\c$ is not a root forces $0 \le \ttj' -\ttj \not= 1$.  

Set $\bb = \{ -\a_\ttj \,,\, -\a_{\ttj'} \} =: \{ -\b_1 \,,\,-\b_2\}$.  By Remark \ref{R:Iwt} the weight space $\bI^{\c-\b}_w$ is one-dimensional and spanned by $\bb.\bv_w$.  In particular, \eqref{E:B=R_LI} holds.  The vector $\bb.\bv_w$ is given by \eqref{E:b.vw}.  Arguments analogous to those of Section \ref{S:Ak_RR} yield \eqref{E:3rdSum}.  Suppose $\ttj = \ttj_r$ and $\ttj' = \ttj_s$.  There are two cases to consider: 

{\bf (A)} Suppose $\b(Z_w) = \sfa+2$ and $\c(Z_w) = \sfa$. First assume $1 < \ttj'-\ttj$.  The pair $\mu,\nu$ of \eqref{E:3rdSum} exists if and only if at least one of the following holds:
\begin{equation} \label{E:A.2}
  1< \ttj_{r+1} -\ttj_r 
  \qquad \hbox{or} \qquad 
  1 < \ttj_{s+1} - \ttj_s \,. \qquad
  \hbox{(If $s = |\ttJ|$, then $\ttj_{s+1} := n$.)}
\end{equation}  

Next suppose $\ttj' = \ttj$.  The pair $\mu,\nu$ exists if and only if \eqref{E:A.2} holds.

{\bf (B)} Suppose $\b(Z_w) = \sfa+1$ and $\c(Z_w) = \sfa-1$.  First assume $1 < \ttj'-\ttj$.  The pair $\mu,\nu$ of \eqref{E:3rdSum} exists if and only if at least one of the following holds:
\begin{equation} \label{E:B.1}
  1< \ttj_r -\ttj_{r-1} 
  \qquad \hbox{or} \qquad 
  1 < \ttj_s - \ttj_{s-1} \,.  \qquad
  \hbox{(If $r :=1$, then $\ttj_{0} := 0$.)}
\end{equation}

Next suppose $\ttj' = \ttj$.  The pair $\mu,\nu$ exists if and only if \eqref{E:B.1} holds.

\medskip

\noindent Modulo Lemma \ref{L:B=R_2}, which is proven below, this establishes the proposition.

\begin{proof}[Proof of Lemma \ref{L:B=R_2} for $G=C_n$]
Given \eqref{E:B=Rtest}, it suffices to show the following:  if there exists $(\b,\c)\in\Pi(w)$ such that $\xi\bv_w \in \bI^{\c-\b}_w$, then there exists $(\b_o,\c_o) \in \Pi_(w)$ with $(\b_o-\c_o)(Z_w)=2$ and $\xi_o\bv_w \in \bI^{\c_o-\b_o}_w$.

{\bf (a.a)} Begin with the case that $\b$ and $\c$ are of the form \eqref{E:Cb1} and \eqref{E:Cc1}, respectively.  The condition that $\b-\c$ be a sum of positive roots, but not a root itself, implies that either $\ttj_q < \ttj_s+1 < \ttj_r < \ttj_t+1$, or  $\ttj_r < \ttj_s+1$. 

{\bf (a.a.1)}  Assume $\ttj_q < \ttj_s+1 < \ttj_r < \ttj_t+1$.  Suppose that $\b(Z_w) > \sfa+1$ and $\ttj_q < \ttj_s$.  Set $\be = \{ \a_{\ttj_q} \,,\, \a_{\ttj_q+1} \,,\ldots,\, \a_{\ttj_{q+1}-1} \}$.  Then $\be . (\xi\bv_w)$ is a nonzero multiple of $E_{-\b_o} \wedge (E_\c\lefthook\bv_w) = \xi_o\bv_w$, where $\b_o  = \a_{\ttj_{q+1}} + \cdots + \a_{\ttj_{r}-1} + 2 ( \a_{\ttj_{r}} + \cdots + \a_{n-1}) + \a_n$.  In particular, $\b_o$ is also a $\fg_{0,0}$--lowest weight, and $\b_o(Z_w) = \b(Z_w) - 1$.  Continuing inductively, we may assume that either $(\b-\c)(Z_w) = 2$, or $\b(Z_w) = \sfa+1$, or $\ttj_q = \ttj_s$.  In the first case we are done.

Suppose $\b(Z_w) > \sfa+1$ and $\ttj_r < \ttj_t$, and set $\be = \{ \a_{\ttj_r} \,,\, \a_{\ttj_r+1} \,,\ldots,\, \a_{\ttj_{r+1}-1} \}$.  Then $\be.(\xi\bv_w)$ is a nonzero multiple of $E_{-\b_o} \wedge (E_\c\lefthook\bv_w) = \xi_o\bv_w$, where $\b_o  = \a_{\ttj_q} + \cdots + \a_{\ttj_{r+1}-1} + 2 ( \a_{\ttj_{r+1}} + \cdots + \a_{n-1}) + \a_n$.  In particular, $\b_o$ is also a $\fg_{0,0}$--lowest weight, and $\b_o(Z_w) = \b(Z_w) - 1$.  Continuing inductively, we may assume that either $\b(Z_w) = \sfa+1$, or both $\ttj_q = \ttj_s$ and $\ttj_r = \ttj_t$.  In the latter case $(\b-\c)(Z_w) = 2$ and we are done.  So we assume $\b(Z_w) = \sfa+1$.

If $\ttj_q < \ttj_s$, set $\be = \{ \a_{\ttj_s} \,,\, \a_{\ttj_s-1} \,,\ldots,\, \a_{\ttj_{s-1}+1} \}$.  Then $\be.(\xi\bv_w)$ is a nonzero multiple of $E_{-\b} \wedge (E_{\c_o}\lefthook\bv_w) = \xi_o\bv_w$ where $\c_o = \a_{\ttj_{s-1}+1} + \cdots + \a_{\ttj_t} + 2 ( \a_{\ttj_t+1} + \cdots + \a_{n-1}) + \a_n$.  Note that $\c_o$ is, like $\c$, a $\fg_{0,0}$--highest weight, and $\c_o(Z_w) = \c(Z_w) + 1$.  Continuing inductively, we may assume that $\ttj_q = \ttj_s$.  

It remains to consider the case that $\ttj_r < \ttj_t$.  As above, an inductive argument, with $\be = \{ \a_{\ttj_t} \,,\,\a_{\ttj_t-1} \,,\ldots,\, \a_{\ttj_{t-1}+1} \}$, reduces us to the case that $\ttj_q = \ttj_s$ and $\ttj_r = \ttj_t$.  To summarize: we have produced an element $\xi_o\bv_w$ in the $\fg_0$ orbit of $\xi\bv_w$ of the form 
\begin{equation} \label{E:Cjk}
  \begin{array}{rcl}
  \b_o & = & \a_\ttj + \cdots + \a_{\ttk-1} + 2(\a_\ttk + \cdots + \a_{n-1}) + \a_n
  \,,\\
  \c_o & = & \a_{\ttj+1} + \cdots + \a_\ttk + 
  2(\a_{\ttk+1} + \cdots + \a_{n-1}) + \a_n \,,
  \end{array}
\end{equation}
for some $\ttj , \ttk \in \ttJ$, satisfying $(\c_o,\b_o) \in \Pi(w)$ and $(\b_o-\c_o)(Z_w) = 2$.

{\bf (a.a.2)}  Manipulations similar the those of (a.a.1) yield an element $\xi_o\bv_w$ in the $\fg_0$--orbit of $\xi\bv_w$ of the form 
\begin{equation} \label{E:Cjj}
  \b_o \ = \ 2(\a_\ttj + \cdots + \a_{n-1}) + \a_n
  \,,\quad
  \c_o \ = \ 2(\a_{\ttj+1} + \cdots + \a_{n-1}) + \a_n \,,
\end{equation}
for some $\ttj \in \ttJ$.  Note that $(\c_o,\b_o) \in \Pi(w)$ and $(\b_o-\c_o)(Z_w) = 2$.

{\boldmath $(\ast,\ast)$\unboldmath} Similar arguments in the remaining cases yield $\xi_o\bv_w$ in the $\fg_0$--orbit of $\xi\bv_w$ with $\xi_o = E_{-\b_o} \ot E_{\c_o}$ and $(\c_o , \b_o) \in \Pi(w)$ of the form \eqref{E:Cjk} or \eqref{E:Cjj}.
\end{proof}

\subsection{Spinor variety \boldmath $D_n/P_n$ \unboldmath } \label{S:Dn_RR}

In this section we show that $B_w = R_w$ holds for the varieties in Theorem \ref{T:BR}(d).

Let $(\c,\b) \in \Pi(w)$.  As a lowest $\fg_{0,0}$--weight $\b$ is of one of the following forms
\begin{subequations} 
\begin{eqnarray}
  \label{E:Db1}
  \b & = & \a_{\ttj_r} + \cdots + \a_n \quad (n-1 \in \ttJ) \,,
  \qquad
  \b \ = \ \a_{\ttj_r} + \cdots + \a_{n-2} + \a_n \,,\\
  \label{E:Db2}
  \b & = & \a_{\ttj_r} + \cdots + \a_{\ttj_s-1} + 
  2 (\a_{\ttj_s} + \cdots + \a_{n-2} ) + \a_{n-1} + \a_n \,,\\
  \label{E:Db3}
  \b & = & \a_{\ttj_s-1} + 
  2 (\a_{\ttj_s} + \cdots + \a_{n-2} ) + \a_{n-1} + \a_n  
  \quad (\ttj_s-1 \not\in \ttJ) \,,
\end{eqnarray}
\end{subequations}
for some $\ttj_r , \ttj_s \in \ttJ = \{ \ttj_1 , \ldots , \ttj_\sfp\}$.  (Cf. the proof of Lemma \ref{L:bigone} and Corollary \ref{C:bigone}.)    Similarly, as a $\fg_{0,0}$--highest weight, $\c$ is of one of the following forms
\begin{subequations} 
\begin{eqnarray}
  \label{E:Dc1}
  \c & = & \a_{\ttj_t+1} + \cdots + \a_n 
  \quad\hbox{or}\quad \a_{\ttj_t+1} + \cdots + \a_{n-2} + \a_n \quad 
  (n-1\in \ttJ) \,,\\
  \label{E:Dc3}
  \c & = & \a_{\ttj_t+1} + \cdots + \a_{\ttj_u} + 
  2 (\a_{\ttj_u+1} + \cdots + \a_{n-2} ) + \a_{n-1} + \a_n \,,\\
  \label{E:Dc4}
  \c & = & \a_{\ttj_t+1} + 
  2 (\a_{\ttj_t+2} + \cdots + \a_{n-2} ) + \a_{n-1} + \a_n  \quad
  (\ttj_t+1 \not\in \ttJ) \,,
\end{eqnarray}
\end{subequations}
for some $\ttj_t,\ttj_u \in \ttJ\cup\{0\}$, with $\ttj_0 := 0$.  The requirement that $\b-\c$ be a sum of positive roots, but not a root itself, rules out the second root in \eqref{E:Db1}.  The conditions $\ttj_s-1 \not\in \ttJ$ in \eqref{E:Db3} and $\ttj_t+1 \not\in \ttJ$ in \eqref{E:Dc4} hold for all $X_w$ satisfying $\sfH_+$; see Theorem \ref{T:BR}(d).

Lemma \ref{L:B=R_2} holds for $G=D_n$; see the end of this section.    The condition $(\b-\c)(Z_w) = 2$ implies that the $\b-\c$ is of one of the two following forms
$$
  \a_\ttj + \a_{\ttj'} \quad \hbox{ or } \quad 
  \a_{\ttj-1} + 2\a_\ttj + \a_{\ttj+1} \,,
$$
with $\ttj \not= \ttj' \in \ttJ$.  Thus, 
\begin{subequations}
\begin{eqnarray} \label{E:bb1}
  \bb & = & \{ -\a_\ttj \,,\, -\a_{\ttj'} \} \quad\hbox{or}\\
  \label{E:bb2}
  \bb & = & \{ -\a_{\ttj-1} - \a_\ttj \,,\, -\a_\ttj - \a_{\ttj+1} \} \, .
\end{eqnarray}
\end{subequations}

By Remark \ref{R:Iwt} the weight space $\bI^{\c-\b}_w$ is one-dimensional and spanned by $\bb.\bv_w$.  As in Sections \ref{S:Ak_RR} \& \ref{S:Cn_RR}, the statements \eqref{E:B=R_LI} and \eqref{E:3rdSum} hold.  

Suppose $\ttj = \ttj_t < \ttj' = \ttj_u$.  First suppose that $\bb$ is of the form \eqref{E:bb1}.  Then $\c$ is of the form \eqref{E:Dc3}.  The pair $\mu,\nu$ of \eqref{E:3rdSum} exists if and only if either
\begin{eqnarray*}
  1< \ttj_{t+1} - \ttj_{t}  \quad \hbox{or} \quad 
  1< \ttj_{u+1} - \ttj_{u} & & 
  \hbox{(when $\c(Z_w) = \sfa$);} \\
  1< \ttj_{t} - \ttj_{t-1}  \quad \hbox{or} \quad 
  1< \ttj_{u} - \ttj_{u-1} & & 
  \hbox{(when $\c(Z_w) = \sfa-1$).}
\end{eqnarray*}
(If $t=1$, then $\ttj_0 := 0$.  If $u = |\ttJ|$, then $\ttj_{u+1} := n-2$.)  Next suppose $\bb$ is of the form \eqref{E:bb2}.  Then $\c$ is of the form \eqref{E:Dc4}.  The pair $\mu,\nu$ of \eqref{E:3rdSum} exists if and only if
\begin{eqnarray*}
  2< \ttj_{t+1} - \ttj_{t} & & \hbox{(when $\c(Z_w) = \sfa$);} \\
  2< \ttj_{t} - \ttj_{t-1} & & \hbox{(when $\c(Z_w) = \sfa-1$).}
\end{eqnarray*}
When comparing these inequalities to Theorem \ref{T:BR}(d), it is helpful make the following observations:
\begin{i_list}
\item  Suppose $n-1\not\in\ttJ$.  In this case $\c(Z_w) = 2(\sfp - t)$.  If $\c(Z_w) = \sfa-1$, then $t = \sfp-\sfr+1$.  If $\c(Z_w) = \sfa$, then $t = \sfp-\sfr$.
\item Assume $n-1 \in \ttJ$.  In this case $\c(Z_w) = 2(\sfp-t)-1$.  If $\c(Z_w) = \sfa-1$, then $t = \sfp-\sfr$.  If $\c(Z_w) = \sfa$, then $t = \sfp-\sfr-1$.
\end{i_list}

\noindent Modulo Lemma \ref{L:B=R_2}, this establishes the proposition.

\begin{proof}[Proof of Lemma \ref{L:B=R_2} for $G=D_n$]
The proof is similar to those for $G = A_n$ (Section \ref{S:Ak_RR}) and $G = C_n$ (Section \ref{S:Cn_RR}).  Details are left to the reader.
\end{proof}

\subsection{The exceptional CHSS}

It remains to show that $B_w = R_w$ holds for the varieties in Tables \ref{t:E6} \& \ref{t:E7}.

Lemma \ref{L:B=R_2} does not hold for the exceptional CHSS.  Rather than argue as in Sections \ref{S:Ak_RR}, \ref{S:Cn_RR} \& \ref{S:Dn_RR}, we verified (with the assistance of LiE \cite{LiE}) that for each $(\c,\b) \in \Pi(w)$, the corresponding $\xi\bv_w$ contains a nontrivial component in some $\bI_{w'} \not= \bI_w$ with $|w'| = |w|$.  Schur rigidity then follows from Theorem \ref{T:BR}(e), Proposition \ref{P:Rrigid} and Lemmas \ref{L:B=R} and \ref{L:cH10}.  (Indeed we can select $w'$ so that $\bv_{w'} = E_{-\b_o} \wedge ( E_{\c_o} \lefthook \bv_w)$, where $\c_o$ and $\b_o$ are respectively highest and lowest $\fg_{0,0}$--weights, $\c_o(Z_w) = \sfa$ and $\b_o(Z_w) = \sfa+1$.  From here, it is then not difficult to see -- in each case -- that $\xi\bv_w$ contains a nontrivial component in $\bI_{w'}$.)

\bibliography{refs.bib} 
\bibliographystyle{plain}
\end{document}